\documentclass{amsart}
\usepackage[margin=4cm]{geometry}
\usepackage{graphicx}
\usepackage{color}
\usepackage{enumerate}
\usepackage{tikz}
\usepackage{graphicx}
\usepackage{subcaption}
\usepackage{float}
\usepackage{comment}

\newtheorem{thm}{Theorem}[section]
\newtheorem{cor}[thm]{Corollary}
\newtheorem{lem}[thm]{Lemma}

\theoremstyle{definition}

\theoremstyle{remark}
\newtheorem{rem}[thm]{Remark}

\numberwithin{equation}{section}
\begin{document}

\title[Neumann eigenvalues of the biharmonic operator on domains]{Neumann eigenvalues of the biharmonic operator on domains: geometric bounds and related results\\
}%
\author{Bruno Colbois}%
\address{Bruno Colbois, Universit\`e de Neuch\^atel, Institute de Math\'ematiques, Rue Emile Argand 11, 2000 Neuch\^atel, Switzerland}%
\email{bruno.colbois@unine.ch}%
\author{Luigi Provenzano}%
\address{Luigi Provenzano, Universit\`a degli Studi di Padova, Dipartimento di Matematica, Via Trieste 63, 35121 Padova, Italy}%
\email{luigi.provenzano@math.unipd.it}%
\thanks{}%

\date{}%
%\dedicatory{}%
%\commby{}%
% ----------------------------------------------------------------
\begin{abstract}
We study an eigenvalue problem for the biharmonic operator with Neumann boundary conditions on domains of Riemannian manifolds. We discuss the weak formulation and the classical boundary conditions, and we describe a few properties of the eigenvalues. Moreover, we establish upper bounds compatible with the Weyl's law under a given lower bound on the Ricci curvature.
\end{abstract}
\maketitle
%\tableofcontents
% ----------------------------------------------------------------

%%%%%%%%%%%%%%%%%%%%%%%%%%%%%%%%%%%%%%%%%%%%%%%%%%%%%%%%%%%%%%%%%%%%%%%%%%%%%%%%%%%%%%%%%%%%%%%%%%%%%%%%%%%%%%%%%%%%%%%%%%%%%%%%%%%%%%%%
%%%%%%%%%%%%%%%%%%%%%%%%%%%%%%%%%%%%%%%%%%%%%%%%%%%%%%%%%%%%%%%%%%%%%%%%%%%%%%%%%%%%%%%%%%%%%%%%%%%%%%%%%%%%%%%%%%%%%%%%%%%%%%%%%%%%%%%%
%%%%%%%%%%%%%%%%%%%%%%%%%%%%%%%%%%%%%%%%%%%%%%%%%%%%%%%%%%%%%%%%%%%%%%%%%%%%%%%%%%%%%%%%%%%%%%%%%%%%%%%%%%%%%%%%%%%%%%%%%%%%%%%%%%%%%%%%

%%%%%%%%%%%%%%%%%%%%%%%%%%%%%%%%%%%%%%%%%%%%                INTRODUCTION                  %%%%%%%%%%%%%%%%%%%%%%%%%%%%%%%%%%%%%%%%%%%%% 

%%%%%%%%%%%%%%%%%%%%%%%%%%%%%%%%%%%%%%%%%%%%%%%%%%%%%%%%%%%%%%%%%%%%%%%%%%%%%%%%%%%%%%%%%%%%%%%%%%%%%%%%%%%%%%%%%%%%%%%%%%%%%%%%%%%%%%%%
%%%%%%%%%%%%%%%%%%%%%%%%%%%%%%%%%%%%%%%%%%%%%%%%%%%%%%%%%%%%%%%%%%%%%%%%%%%%%%%%%%%%%%%%%%%%%%%%%%%%%%%%%%%%%%%%%%%%%%%%%%%%%%%%%%%%%%%%
%%%%%%%%%%%%%%%%%%%%%%%%%%%%%%%%%%%%%%%%%%%%%%%%%%%%%%%%%%%%%%%%%%%%%%%%%%%%%%%%%%%%%%%%%%%%%%%%%%%%%%%%%%%%%%%%%%%%%%%%%%%%%%%%%%%%%%%%

\section{Introduction}

Let $(M,g)$ be a complete $n$-dimensional smooth Riemannian manifold and let $\Omega\subset M$ be a bounded domain, i.e., a bounded connected open set, with boundary $\partial\Omega$. We consider the following Neumann eigenvalue problem for the biharmonic operator:
\begin{equation}\label{neumann_classic}
\begin{cases}
\Delta^2 u=\mu u, & {\rm in\ }\Omega,\\
\frac{\partial^2u}{\partial\nu^2}=0, & {\rm on\ }\partial \Omega,\\
{\rm div}_{\partial \Omega}\left(\nabla_{\nu}\nabla u\right)_{\partial\Omega}+\frac{\partial\Delta u}{\partial\nu}=0, & {\rm on\ }\partial \Omega.
\end{cases}
\end{equation}
in the unknowns $u$ (the eigenfunction) and $\mu$ (the eigenvalue). Here $\Delta u={\rm div}(\nabla u)$ is the Laplacian (or Laplace-Beltrami operator) of $u$ on $(M,g)$, $\Delta^2 u=\Delta(\Delta u)$, $\nu$ denotes the outer unit normal to $\partial \Omega$, $\frac{\partial^2u}{\partial\nu^2}=\langle\nabla_{\nu}\nabla u,\nu\rangle$ is the second normal derivative, ${\rm div}_{\partial\Omega}$ is the divergence on $\partial\Omega$ with respect to the induced metric, and $F_{\partial\Omega}$ denotes the projection of $F\in TM$ on $T{\partial\Omega}$.

We recall that in the case $M=\mathbb R^n$ with the Euclidean metric, problem \eqref{neumann_classic} is well-known and has increasingly gained attention in recent years. We refer to \cite{ferraresso,buosopalinuro,bcp,chasman,chasmancircular,chasman2,Lap1997,pleijel_plate_1,pleijel_plate_2,kalamata} for the eigenvalue problem and to \cite{verchota} for the corresponding boundary value problem. Problem \eqref{neumann_classic} in dimension two represents Kirchhoff's solution to the problem of describing the transverse vibrations of a thin elastic plate with free edges. We refer to \cite{bourlard, giroire, nadai, nazaret} for more details and for historical information.

We also note that the corresponding Dirichlet problem for the biharmonic operator, which for planar domains is related to the study of the transverse vibrations of a thin elastic plate with clamped edges \cite{cohil}, has been extensively studied not only in the Euclidean setting, but also for domains in Riemannian manifolds, see e.g., \cite{cheng_clamped,wang_clamped,xia_clamped}. In particular, the Dirichlet problem on Euclidean domains and the analogous problem on Riemannian manifolds share many properties which can be derived by using similar arguments.

On the other hand, we have not been able to find the analogue of the biharmonic Neumann problem on Riemannian manifolds in the literature. The first aim of the present paper is to introduce problem \eqref{neumann_classic} on domains of Riemannian manifolds in a suitable way, derive the boundary conditions as well as the variational formulation. The problem which we obtain turns out to be the genuine generalization of the biharmonic Neumann problem for Euclidean domains. We remark that the standard technique used to derive the boundary conditions and the variational formulation of problem \eqref{neumann_classic}  in the Euclidean case it to multiply the eigenvalue equation $\Delta^2u=\mu u$ by a test function $\phi\in C^{\infty}(\Omega)$, integrate  the resulting equality over $\Omega$ and perform suitable integrations by parts. Computations become easy since we can exchange the order of partial derivatives. This is no more possible in the case of Riemannian manifolds, hence we have to take a longer path, described in Subsection \ref{sub_21}. An essential tool is Reilly's identity. It turns out that the strategy described in Subsection \ref{sub_21} allows to define alternatively problem \eqref{neumann_classic} also in the Euclidean case. 

%Thus the Neumann boundary conditions in \eqref{neumann_classic} and the weak formulation \eqref{variational_0} which we obtain coincide with the well-known ones in the case of Euclidean domains.

Then, we prove that problem \eqref{neumann_classic} admits an increasing sequence of eigenvalues of finite multiplicity diverging to $+\infty$ of the form
$$
-\infty<\mu_1\leq\mu_2\leq\cdots\leq\mu_{j}\leq\cdots\nearrow+\infty.
$$
If $\Omega$ is a bounded domain of $\mathbb R^n$ it is known that the eigenvalues are non-negative and satisfy the Weyl's asymptotic law
$$
\lim_{j\rightarrow+\infty}\frac{\mu_j}{j^{\frac{4}{n}}}=\frac{16\pi^4}{(\omega_n|\Omega|)^{\frac{4}{n}}},
$$
that is
\begin{equation}\label{weyl}
\mu_j\sim\frac{16\pi^4}{\omega_n^{\frac{4}{n}}}\left(\frac{j}{|\Omega|}\right)^{\frac{4}{n}}\,,\ \ \ {\rm as\ }j\rightarrow+\infty,
\end{equation}
where $\omega_n$ denotes the volume of the unit ball in $\mathbb R^n$ and $|\Omega|$ denotes the Lebesgue measure of $\Omega$, see e.g., \cite{Lap1997}.

An important question regarding the eigenvalues of  Neumann-type problems is that of finding upper bounds which are compatible with the Weyl's law. One of the main purposes of the present paper is that of finding upper bounds for $\mu_j$ which can be compared with \eqref{weyl} and which contain the correct geometric information.

In the case of Euclidean domains, Weyl-type upper bounds for $\mu_j$ are well-known and are of the form
\begin{equation}\label{laptev}
\mu_j\leq A_n\left(\frac{j}{|\Omega|}\right)^{\frac{4}{n}}
\end{equation}
with $A_n=\left(\frac{4+n}{4}\right)^{4/n}\frac{16\pi^4}{\omega_n^{4/n}}$,  see e.g., \cite{Lap1997}.  The proof in \cite{Lap1997} is in the spirit of the analogous result of Kr\"oger for the Neumann eigenvalues $m_j$ of the Laplacian on Euclidean domains, namely
\begin{equation}\label{kroger}
m_j\leq B_n\left(\frac{j}{|\Omega|}\right)^{\frac{2}{n}},
\end{equation}
where $B_n=\left(\frac{2+n}{2}\right)^{2/n}\frac{4\pi^2}{\omega_n^{2/n}}$, see \cite{Kro}. Note that \eqref{kroger} is compatible with the Weyl's law
\begin{equation}\label{weyl2}
m_j\sim\frac{4\pi^2}{\omega_n^{\frac{2}{n}}}\left(\frac{j}{|\Omega|}\right)^{\frac{2}{n}}\,,\ \ \ {\rm as\ }j\rightarrow+\infty,
\end{equation}
The proofs of \eqref{laptev} and \eqref{kroger} rely on harmonic analysis techniques and are hardly adaptable to the case of manifolds.

In the case of the eigenvalues of the Laplacian on manifolds, one of the first result in this direction is presented in \cite{buser}, where it is proved that
\begin{equation}\label{buser}
m_j\leq\frac{(n-1)^2}{4}\kappa^2+C_n\left(\frac{j}{|M|}\right)^{\frac{2}{n}}.
\end{equation}
Here $m_j$ denote, with abuse of notation, the eigenvalues of the Laplacian on a compact manifold (without boundary) $M$ with ${\rm Ric}\geq -(n-1)\kappa^2$, $\kappa\geq 0$. Results on domains have been obtained more recently. In \cite{colbois_maerten} the authors prove the following upper bound
\begin{equation}\label{colbois_maerten}
m_j\leq A_n\left(\frac{j}{|\Omega|}\right)^{\frac{2}{n}}+B_n\kappa^2,
\end{equation}
for the Neumann eigenvalues $m_j$ of the Laplacian on a domain $\Omega$ of a complete Riemannian manifold with ${\rm Ric}\geq-(n-1)\kappa^2$, $\kappa\geq 0$. The authors adopt a metric approach for the proof of \eqref{colbois_maerten}. 

We will use this approach also in the present paper in order to obtain upper bounds for $\mu_j$. In view of \eqref{weyl}, \eqref{weyl2} and \eqref{colbois_maerten}, it is natural to conjecture that the inequality
\begin{equation}\label{conjecture0}
\mu_j\leq A_n\left(\frac{j}{|\Omega|}\right)^{\frac{4}{n}}+B_n\kappa^4
\end{equation}
holds for any bounded domain of a complete Riemannian manifold with ${\rm Ric}\geq-(n-1)\kappa^2$, $\kappa\geq 0$. We are able to prove \eqref{conjecture0} for certain classes of domains and manifolds. In particular we prove \eqref{conjecture0} for domains of manifolds with non-negative Ricci curvature and $n=2,3,4$ (see Theorem \ref{low_dim_positive}), for domains of the standard sphere (see Theorem \ref{thm_sphere}) and of the standard hyperbolic space (see Theorem \ref{main_hyp}), for boundaryless manifolds (see Corollary \ref{cor_boundaryless}), and for convex domains (see Corollary \ref{cor_convex}). 

In the general case, we are able to prove an estimate of the form
\begin{equation}\label{conjecture1}
\mu_j \leq A_n\left(\frac{j}{|\Omega|}\right)^{\frac{4}{n}}+C(g)
\end{equation}
(see Theorem \ref{main_general}), where $C(g)$ has an explicit form and depends on $\kappa,r_{inj,\Omega},|\partial\Omega|$, where $r_{inj,\Omega}$ is the injectivity radius relative to $\Omega$ (see \eqref{injO} for the definition) and $|\partial\Omega|$ is the $n-1$-dimensional Hausdorff measure of the boundary. Estimate \eqref{conjecture1} is improved if we put additional hypothesis on $\Omega$ and $M$. In particular we provide more refined estimates in the case of domains with sufficiently small diameter in  manifolds with non-negative Ricci curvature (see Theorem \ref{thm_ric_pos_dist}) and in the case of Cartan-Hadamard manifolds (see Theorem \ref{main_cartan}).

It is important to remark that a bound of the form \eqref{conjecture1} is good in the sense that the geometry of the domain and of the manifold appears as an {\it additive} constant in front of the term encoding the asymptotic behavior, which has the correct form compared with the asymptotic law \eqref{weyl}.

As already mentioned, in order to prove the upper bounds, we adopt a metric approach. In particular, we exploit a result of decomposition of a metric measure space by disjoint {\it capacitors}, see \cite{gny,asma_conf} for more details, see also \cite{colboisgirouard,colbois_maerten}. Namely, given a domain $\Omega$, we find, for each $j\in\mathbb N$, a family of $j$ disjointly supported sets $A_i$, $i=1,...,j$, in $\Omega$ with sufficient volume. Associated to each set $A_i$ we build a function $u_i$ to test in the min-max formula for the eigenvalues (see formula \eqref{minmax}). Since the $u_i$ are disjointly supported, from \eqref{minmax} we deduce that it is sufficient to bound the Rayleigh quotient of each $u_i$ in order to upper bound $\mu_j$. Hence the functions $u_i$ have to be constructed in a proper way.

We remark that test functions for the biharmonic operator need to belong to the standard Sobolev space $H^2(\Omega)$. Usually, test functions are built in terms of distance-type functions, which are Lipschitz, but are not in general $H^2(\Omega)$ functions. In particular, test functions for the Laplacian eigenvalues are cut-off functions which are just Lipschitz regular. The application of the technique used in \cite{colbois_maerten} for the Laplacian is no straightforward in our situation, in fact it is notoriously a difficult task to build cut-off functions enjoying precise estimates for first and second derivatives, see e.g., \cite{bianchi_setti,cheeger_colding,guneysu}. We pay the price of the fact that we need cut-off functions in $H^2(\Omega)$ with well-behaved gradient and Laplacian by introducing into the estimates the quantities $r_{inj,\Omega}$ and $|\partial\Omega|$. Getting rid of these quantities in the general case seems a very difficult issue.

Looking at \eqref{weyl} and \eqref{weyl2}, one may wonder if there is some sort of relationship between $\mu_j$ and $m_j^2$ and if it is possible, in general, to recover upper estimates for $\mu_j$ from upper estimates on $m_j$. The answer is negative in general, in fact we provide examples showing that the ratio $\frac{\mu_j}{m_j^2}$ may be made arbitrarily large or arbitrarily close to zero.

Another interesting feature of problem \eqref{neumann_classic} is that it is possible to produce {\it negative} eigenvalues. This does not happen with the eigenvalues of the biharmonic Dirichlet problem on domains of manifolds. In particular, in subsection \ref{neg_eig} we prove that any domain of the standard hyperbolic space $\mathbb H^n$ admits at least $n$ negative eigenvalues. Moreover, we prove that there exist domains with an arbitrarily large number of negative eigenvalues, the absolute value of which can be made arbitrarily large. On the other hand, for domains in manifolds with positive Ricci curvature we prove a lower bound for the eigenvalues $\mu_j$ in terms of $m_j$, $\eta_j$ and a lower bound on the Ricci curvature (see Theorem \ref{lower_pos}), where $\eta_j$ denote the eigenvalues of the rough Laplacian on $\Omega$.

The present paper is organized as follows. In Section \ref{pre} we recall some preliminaries and introduce the notation. In Section \ref{sec_2} we describe the classical Neumann boundary conditions in \eqref{neumann_classic} and derive the weak formulation of the problem, proving that it is well-posed and characterizing its spectrum. In Section \ref{afew} we discuss a few properties of the eigenvalues. In particular we provide examples where the ratio $\frac{\mu_j}{m_j^2}$ can be made arbitrarily large or close to zero. We prove that any domain of the hyperbolic space admits at least $n$ negative eigenvalues, and that there exists domains with an arbitrarily large number of negative eigenvalues with  arbitrarily large absolute value. We also prove a lower bound for $\mu_j$ for domains on manifolds with positive Ricci curvature. In Section \ref{bounds} we recall the main technical results of decomposition of a metric measure space by capacitors, which allow to prove the upper estimates for the eigenvalues $\mu_j$ presented in the same section.

%%%%%%%%%%%%%%%%%%%%%%%%%%%%%%%%%%%%%%%%%%%%%%%%%%%%%%%%%%%%%%%%%%%%%%%%%%%%%%%%%%%%%%%%%%%%%%%%%%%%%%%%%%%%%%%%%%%%%%%%%%%%%%%%%%%%%%%%
%%%%%%%%%%%%%%%%%%%%%%%%%%%%%%%%%%%%%%%%%%%%%%%%%%%%%%%%%%%%%%%%%%%%%%%%%%%%%%%%%%%%%%%%%%%%%%%%%%%%%%%%%%%%%%%%%%%%%%%%%%%%%%%%%%%%%%%%
%%%%%%%%%%%%%%%%%%%%%%%%%%%%%%%%%%%%%%%%%%%%%%%%%%%%%%%%%%%%%%%%%%%%%%%%%%%%%%%%%%%%%%%%%%%%%%%%%%%%%%%%%%%%%%%%%%%%%%%%%%%%%%%%%%%%%%%%

%%%%%%%%%%%%%%%%%%%%%%%%%%%%%%%%%%%%%%%%%%%%             PRELIMINARIES AND NOTATION              %%%%%%%%%%%%%%%%%%%%%%%%%%%%%%%%%%%%%%%%

%%%%%%%%%%%%%%%%%%%%%%%%%%%%%%%%%%%%%%%%%%%%%%%%%%%%%%%%%%%%%%%%%%%%%%%%%%%%%%%%%%%%%%%%%%%%%%%%%%%%%%%%%%%%%%%%%%%%%%%%%%%%%%%%%%%%%%%%
%%%%%%%%%%%%%%%%%%%%%%%%%%%%%%%%%%%%%%%%%%%%%%%%%%%%%%%%%%%%%%%%%%%%%%%%%%%%%%%%%%%%%%%%%%%%%%%%%%%%%%%%%%%%%%%%%%%%%%%%%%%%%%%%%%%%%%%%
%%%%%%%%%%%%%%%%%%%%%%%%%%%%%%%%%%%%%%%%%%%%%%%%%%%%%%%%%%%%%%%%%%%%%%%%%%%%%%%%%%%%%%%%%%%%%%%%%%%%%%%%%%%%%%%%%%%%%%%%%%%%%%%%%%%%%%%%

\section{Preliminaries and notation}\label{pre}

Let $(M,g)$ be a complete $n$-dimensional smooth Riemannian manifold. For a bounded domain $\Omega$ in $M$, by $L^2(\Omega)$ we denote the space of measurable functions $f$ on $\Omega$ such that $\int_{\Omega}u^2dv<\infty$.

The Sobolev space $H^2(\Omega)$ is the completion of $C^{\infty}(\Omega)$ with respect to the norm 
\begin{equation}\label{norm-0}
\|f\|_{H^2(\Omega)}:=\left(\int_{\Omega}\left(|D^2f|^2+|\nabla f|^2+f^2\right)dv\right)^{\frac{1}{2}}.
\end{equation} 
The space $L^2(\Omega)$ is a Hilbert space when endowed with the standard scalar product
\begin{equation}\label{prod-00}
\langle f,g\rangle_{L^2(\Omega)}:=\int_{\Omega}fgdv.
\end{equation}
The space $H^2(\Omega)$ is a Hilbert space when endowed with the standard scalar product
\begin{equation}\label{prod-0}
\langle f,g\rangle_{H^2(\Omega)}:=\int_{\Omega}\left(\langle D^2f,D^2g\rangle+\langle\nabla f,\nabla g\rangle+fg\right)dv,
\end{equation} 
which induces the norm \eqref{norm-0}. 

The space $H^2_0(\Omega)$ is the completion of $C^{\infty}_c(\Omega)$ with respect to \eqref{norm-0}, there $C^{\infty}_c(\Omega)$ is the space of functions in $C^{\infty}(\Omega)$ compactly supported in $\Omega$. We refer to \cite{hebey} for an introduction to Sobolev spaces on Riemannian manifolds.

We denote here by $dv$ the Riemannian volume element of $M$ and by $d\sigma$ the induced $n-1$-dimensional volume element of $\partial \Omega$.
\medskip

Through all the paper, we denote by $\langle\cdot,\cdot\rangle$ the inner product on the tangent spaces of $M$ associated with the metric $g$, and, with abuse of notation, we shall denote by $\langle\cdot,\cdot\rangle$ also the induced metric on $\partial \Omega$. Let $\nabla$, $D^2$ and $\Delta$ denote the gradient, the Hessian and the Laplacian on $M$, respectively, and let $\nabla_{\partial \Omega}$, ${\rm div}_{\partial \Omega}$ and $\Delta_{\partial \Omega}$ denote the gradient, the divergence and the Laplacian on $\partial \Omega$ with respect to the induced metric, respectively. We denote by $\nu$ the outer unit normal to $\partial \Omega$. The {\it shape operator} of $\partial \Omega$, denoted by $S$, is defined for any $X\in T\partial \Omega$ as $S(X):=\nabla_X \nu$, where $\nabla_X \nu$ is the covariant derivative of $\nu$ along a vector field $X$. The {\it second fundamental form} of $\partial \Omega$, denoted by $II(X,Y)$, is defined as $II(X,Y):=\langle S(X),Y\rangle$ for all $X,Y\in T\partial \Omega$. We recall that the eigenvalues of $S$ are the principal curvatures of $\partial \Omega$.  We will denote by $\mathcal H:=\frac{1}{n-1}{\rm tr} S=\frac{1}{n-1}{\rm div}\nu_{|_{\partial \Omega}}$ the mean curvature of $\partial \Omega$. Let ${\rm Ric}(\cdot,\cdot)$ denote the Ricci tensor of $M$. Finally, for an open set $E\in M$ we denote by $|E|$ the standard Lebesgue measure of $E$. For a closed set $G\in M$ of finite $n-1$-dimensional Hausdorff measure, we still denote by $|G|$ the $n-1$-dimensional Hausdorff measure of $G$. 

We recall Bochner's formula:
\begin{equation}\label{bochner_f}
|D^2f|^2+{\rm Ric}(\nabla f,\nabla f)=\frac{1}{2}\Delta\left(|\nabla f|^2\right)-\langle\nabla\Delta f,\nabla f\rangle,
\end{equation}
holding pointwise for smooth functions $f$ on $\Omega$.

It is also useful to recall Reilly's formula, see \cite{reilly}:
\begin{multline}\label{reilly}
\int_{\Omega} (\Delta f)^2 dv-\int_{\Omega}|D^2f|^2+{\rm Ric}(\nabla f,\nabla f) dv\\
=\int_{\partial \Omega}(n-1)\mathcal H\left(\frac{\partial f}{\partial\nu}\right)^2+2\Delta_{\partial \Omega}f\frac{\partial f}{\partial\nu}+II(\nabla_{\partial \Omega}f,\nabla_{\partial \Omega}f) d\sigma,
\end{multline}
holding for smooth functions $f$ on $\Omega$.

We also recall Green's identity for smooth functions $f,g$  :
\begin{equation}\label{green}
\int_{\Omega} \left(\Delta f g-f\Delta g \right) dv=\int_{\partial \Omega}\left(\frac{\partial f}{\partial\nu}g-f\frac{\partial g}{\partial\nu}\right)d\sigma.
\end{equation}

We recall that for any smooth vector field $F$ on $T\partial\Omega$ and any function $f$ defined on $\partial\Omega$,
\begin{equation}\label{div}
\int_{\partial\Omega}\langle F,\nabla_{\partial\Omega} f\rangle d\sigma=\int_{\partial\Omega}{\rm div}_{\partial\Omega}F fd\sigma,
\end{equation}
that is, the divergence operator is the adjoint of the gradient. In particular, \eqref{div} holds with $\partial\Omega$ replaced by any complete smooth (boundaryless) Riemannian manifold $(M,g)$, and ${\rm div}_{\partial\Omega}$, $\nabla_{\partial\Omega}$ replaced by the divergence and gradient on $M$, respectively.
\medskip

Finally, by $\mathbb N$ we denote the set of positive natural numbers.

%%%%%%%%%%%%%%%%%%%%%%%%%%%%%%%%%%%%%%%%%%%%%%%%%%%%%%%%%%%%%%%%%%%%%%%%%%%%%%%%%%%%%%%%%%%%%%%%%%%%%%%%%%%%%%%%%%%%%%%%%%%%%%%%%%%%%%%%
%%%%%%%%%%%%%%%%%%%%%%%%%%%%%%%%%%%%%%%%%%%%%%%%%%%%%%%%%%%%%%%%%%%%%%%%%%%%%%%%%%%%%%%%%%%%%%%%%%%%%%%%%%%%%%%%%%%%%%%%%%%%%%%%%%%%%%%%
%%%%%%%%%%%%%%%%%%%%%%%%%%%%%%%%%%%%%%%%%%%%%%%%%%%%%%%%%%%%%%%%%%%%%%%%%%%%%%%%%%%%%%%%%%%%%%%%%%%%%%%%%%%%%%%%%%%%%%%%%%%%%%%%%%%%%%%%

%%%%%%%%%%%%%%%%%%%%%%%%%%%%%%%%%%%%%%%%%%%%             THE EIGENVALUE PROBLEM                 %%%%%%%%%%%%%%%%%%%%%%%%%%%%%%%%%%%%%%%%

%%%%%%%%%%%%%%%%%%%%%%%%%%%%%%%%%%%%%%%%%%%%%%%%%%%%%%%%%%%%%%%%%%%%%%%%%%%%%%%%%%%%%%%%%%%%%%%%%%%%%%%%%%%%%%%%%%%%%%%%%%%%%%%%%%%%%%%%
%%%%%%%%%%%%%%%%%%%%%%%%%%%%%%%%%%%%%%%%%%%%%%%%%%%%%%%%%%%%%%%%%%%%%%%%%%%%%%%%%%%%%%%%%%%%%%%%%%%%%%%%%%%%%%%%%%%%%%%%%%%%%%%%%%%%%%%%
%%%%%%%%%%%%%%%%%%%%%%%%%%%%%%%%%%%%%%%%%%%%%%%%%%%%%%%%%%%%%%%%%%%%%%%%%%%%%%%%%%%%%%%%%%%%%%%%%%%%%%%%%%%%%%%%%%%%%%%%%%%%%%%%%%%%%%%%

\section{The eigenvalue problem for the biharmonic operator with Neumann boundary conditions}\label{sec_2}

In this section we  describe the classical Neumann boundary conditions in \eqref{neumann_classic}, as well as the weak formulation of the problem. This is done in Subsection \ref{sub_21}.

Then we prove that problem \eqref{variational_0} is well-posed under suitable assumptions on the domain, and admits an increasing sequence of eigenvalues of finite multiplicity bounded from below and diverging to $+\infty$. This is done in Subsection \ref{sub_22}.

%%%%%%%%%%%%%%%%%%%%%%%%%%%%%%%%%%%%%%%%%%%%%%%%%%%%%%%%%%%%%%%%%%%%%%%%%%%%%%%%%%%%%%%%%%%%%%%%%%%%%%%%%%%%%%%%%%%%%%%%%%%%%%%%%%%%%%%%

%%%%%%%%%%%%%%%%%%%%%%%%%%%%%%%%%%%%%%%%%%%%               CLASSICAL B.C.                  %%%%%%%%%%%%%%%%%%%%%%%%%%%%%%%%%%%%%%%%%%%%% 

%%%%%%%%%%%%%%%%%%%%%%%%%%%%%%%%%%%%%%%%%%%%%%%%%%%%%%%%%%%%%%%%%%%%%%%%%%%%%%%%%%%%%%%%%%%%%%%%%%%%%%%%%%%%%%%%%%%%%%%%%%%%%%%%%%%%%%%%

\subsection{Classical Neumann boundary conditions and weak formulation}\label{sub_21}

We consider the following variational problem:

\begin{equation}\label{variational_0}
\int_M \langle D^2u,D^2\phi\rangle+{\rm Ric}(\nabla u,\nabla\phi)dv=\mu\int_M u\phi dv\,,\ \ \ \forall\phi\in H^2(\Omega),
\end{equation}
in the unknowns $u\in H^2(\Omega)$ and $\mu\in\mathbb R$. Problem \eqref{variational_0} is the variational (weak) formulation of problem \eqref{neumann_classic}, as stated in the following theorem. We remark that it is not straightforward to recognize the left-hand side of \eqref{variational_0} as the right quadratic form for an eigenvalue problem for the biharmonic operator with Neumann boundary conditions. One would like to take the simpler quadratic form $\int_{\Omega}\Delta u\Delta\phi dv$, which however provides an ill-defined problem, see Remark \ref{rem33}.

\begin{thm}\label{equivalence}
Let $(M,g)$ be a complete $n$-dimensional smooth Riemannian manifold and let $\Omega$ be a smooth ($C^{\infty}$) bounded domain in $M$. Given a solution $(u,\mu)$ of problem \eqref{variational_0} such that $u\in C^4(\Omega)\cap C^3(\overline\Omega)$, then $(u,\mu)$ solves problem \eqref{neumann_classic}. Vice versa, any solution $(u,\mu)$ of problem \eqref{neumann_classic} is a solution of problem \eqref{variational_0}.
\end{thm}

Actually, we will prove that \eqref{variational_0} is the weak formulation of the following  eigenvalue problem:
\begin{equation}\label{neumann_classic_2}
\begin{cases}
\Delta^2 u=\mu u, & {\rm in\ }\Omega,\\
(n-1)\mathcal H\frac{\partial u}{\partial\nu}+\Delta_{\partial \Omega}u -\Delta u=0, & {\rm on\ }\partial \Omega,\\
\Delta_{\partial \Omega}\left(\frac{\partial u}{\partial\nu}\right)-{\rm div}_{\partial \Omega} S(\nabla_{\partial \Omega}u)+\frac{\partial\Delta u}{\partial\nu}=0, & {\rm on\ }\partial \Omega,
\end{cases}
\end{equation}
in the unknowns $u$ (the eigenfunction) and $\mu$ (the eigenvalue). Then, we will show that the two boundary conditions in \eqref{neumann_classic_2} coincide with those of \eqref{neumann_classic}, namely we  will prove the following lemma.

\begin{lem}\label{equivalence_bc}
Let $(M,g)$ be a complete $n$-dimensional smooth Riemannian manifold and let $\Omega$ be a smooth bounded domain in $M$. Then, for any $u\in C^3(\overline\Omega)$
\begin{equation}\label{bc1}
(n-1)\mathcal H\frac{\partial u}{\partial\nu}+\Delta_{\partial \Omega}u -\Delta u=-\frac{\partial^2 u}{\partial\nu^2}
\end{equation}
and
\begin{equation}\label{bc2}
\Delta_{\partial \Omega}\left(\frac{\partial u}{\partial\nu}\right)-{\rm div}_{\partial \Omega} S(\nabla_{\partial \Omega}u)+\frac{\partial\Delta u}{\partial\nu}={\rm div}_{\partial \Omega}\left(\nabla_{\nu}\nabla u\right)_{\partial\Omega}+\frac{\partial\Delta u}{\partial\nu}.
\end{equation}
\end{lem} 

%We are now ready to prove Theorem \ref{equivalence}.

\begin{proof}[Proof of Theorem \ref{equivalence}]
Assume that a function $u\in C^4(\Omega)\cap C^3(\overline\Omega)$ and a real number $\mu$ are solution of the  eigenvalue equation
\begin{equation}\label{eigenvalue_eq}
\Delta^2 u=\mu u\,,\ \ \ {\rm in\ }\Omega.
\end{equation}
We multiply both sides of \eqref{eigenvalue_eq} by a function $\phi\in C^{\infty}$ and integrate over $\Omega$, obtaining thanks to \eqref{green}
\begin{equation}\label{step-1}
\int_{\Omega}\Delta^2 u\phi dv=\int_{\Omega} \Delta u\Delta\phi dv+\int_{\partial \Omega}\left(\frac{\partial\Delta u}{\partial\nu}\phi-\Delta u\frac{\partial\phi}{\partial\nu}\right)d\sigma=\mu\int_{\Omega} u\phi dv.
\end{equation}
We set, for a function $f\in C^2(\Omega)$
\begin{equation}\label{Q_M}
Q_{\Omega}(f):=\int_{\Omega} (\Delta f)^2 dv-\int_{\Omega}\left(|D^2f|^2+{\rm Ric}(\nabla f,\nabla f) \right) dv
\end{equation}
and
\begin{equation}\label{Q_pM}
Q_{\partial \Omega}(f):=\int_{\partial \Omega}(n-1)\mathcal H\left(\frac{\partial f}{\partial\nu}\right)^2+2\Delta_{\partial \Omega}f\frac{\partial f}{\partial\nu}+II(\nabla_{\partial \Omega}f,\nabla_{\partial \Omega}f) d\sigma.
\end{equation}
We note that
\begin{equation}\label{polar_M}
\frac{1}{4}(Q_{\Omega}(u+\phi)-Q_{\Omega}(u-\phi))=\int_{\Omega}\Delta u\Delta\phi dv-\int_{\Omega} \left(\langle D^2u,D^2\phi\rangle+{\rm Ric}(\nabla u,\nabla\phi)\right)dv
\end{equation}
and that
\begin{multline}\label{polar_pM}
\frac{1}{4}(Q_{\partial\Omega}(u+\phi)-Q_{\partial \Omega}(u-\phi))\\
=\int_{\partial \Omega}(n-1)\mathcal H\frac{\partial u}{\partial\nu}\frac{\partial \phi}{\partial\nu}+\Delta_{\partial \Omega}u\frac{\partial \phi}{\partial\nu}+\frac{\partial u}{\partial\nu}\Delta_{\partial \Omega}\phi+II(\nabla_{\partial \Omega}u,\nabla_{\partial M}\phi) d\sigma.
\end{multline}
Reilly's formula \eqref{reilly} implies that $Q_{\Omega}(u+\phi)-Q_{\Omega}(u-\phi)=Q_{\partial\Omega}(u+\phi)-Q_{\partial \Omega}(u-\phi)$, thus from \eqref{polar_M} and \eqref{polar_pM} we deduce that
\begin{multline}%\label{step-2-0}
\int_{\Omega}\Delta u\Delta\phi dv-\int_{\Omega}\left( \langle D^2u,D^2\phi\rangle-{\rm Ric}(\nabla u,\nabla\phi)\right)dv\\
=\int_{\partial \Omega}(n-1)\mathcal H\frac{\partial u}{\partial\nu}\frac{\partial \phi}{\partial\nu}+\Delta_{\partial \Omega}u\frac{\partial \phi}{\partial\nu}+\frac{\partial u}{\partial\nu}\Delta_{\partial \Omega}\phi+II(\nabla_{\partial \Omega}u,\nabla_{\partial M}\phi) d\sigma,
\end{multline}
hence
\begin{multline}\label{step-2-0}
\int_{\Omega}\Delta u\Delta\phi dv=\int_{\Omega} \langle D^2u,D^2\phi\rangle+{\rm Ric}(\nabla u,\nabla\phi)dv\\
=\int_{\partial \Omega}(n-1)\mathcal H\frac{\partial u}{\partial\nu}\frac{\partial \phi}{\partial\nu}+\Delta_{\partial \Omega}u\frac{\partial \phi}{\partial\nu}+\frac{\partial u}{\partial\nu}\Delta_{\partial \Omega}\phi+II(\nabla_{\partial \Omega}u,\nabla_{\partial M}\phi) d\sigma.
\end{multline}
Using \eqref{step-2-0} in \eqref{step-1}, we can re-write \eqref{step-1} as
\begin{multline}\label{step-2}
\int_{\Omega}\Delta^2 u\phi dv=\int_{\Omega} \langle D^2u,D^2\phi\rangle+{\rm Ric}(\nabla u,\nabla\phi)dv\\
+\int_{\partial \Omega}\left((n-1)\mathcal H\frac{\partial u}{\partial\nu}+\Delta_{\partial \Omega}u -\Delta u\right)\frac{\partial\phi}{\partial\nu}d\sigma\\
+\int_{\partial \Omega}\left(II(\nabla_{\partial \Omega}u,\nabla_{\partial \Omega}\phi) +\frac{\partial\Delta u}{\partial\nu}\phi +\frac{\partial u}{\partial\nu}\Delta_{\partial \Omega}\phi \right)d\sigma
=\mu \int_{\Omega} u\phi dv.
\end{multline}
We note now that
\begin{multline}\label{parts-1}
\int_{\partial \Omega}II(\nabla_{\partial \Omega}u,\nabla_{\partial \Omega}\phi)d\sigma=\int_{\partial \Omega}\langle S(\nabla_{\partial \Omega}u),\nabla_{\partial \Omega}\phi\rangle d\sigma\\
=-\int_{\partial \Omega}{\rm div}_{\partial \Omega} S(\nabla_{\partial \Omega}u)\phi d\sigma,
\end{multline}
where the second equality follows from \eqref{div}, and that
\begin{equation}\label{parts-2}
\int_{\partial\Omega}\frac{\partial u}{\partial\nu}\Delta_{\partial \Omega}\phi d\sigma=\int_{\partial \Omega}\Delta_{\partial \Omega}\left(\frac{\partial u}{\partial\nu}\right)\phi d\sigma.
\end{equation}
Thanks to \eqref{parts-1} and \eqref{parts-2}, \eqref{step-2} can be rewritten as follows
\begin{multline}\label{step-3}
\int_{\Omega}\Delta^2 u\phi dv=\int_{\Omega} \langle D^2u,D^2\phi\rangle+{\rm Ric}(\nabla u,\nabla\phi)dv\\
+\int_{\partial \Omega}\left((n-1)\mathcal H\frac{\partial u}{\partial\nu}+\Delta_{\partial \Omega}u -\Delta u\right)\frac{\partial\phi}{\partial\nu}d\sigma\\
+\int_{\partial \Omega}\left( \Delta_{\partial \Omega}\left(\frac{\partial u}{\partial\nu}\right)-{\rm div}_{\partial \Omega} S(\nabla_{\partial \Omega}u)+\frac{\partial\Delta u}{\partial\nu}\right)\phi d\sigma
=\mu \int_{\Omega} u\phi dv.
\end{multline}
Assume now that the function $u$ satisfies the boundary conditions in \eqref{neumann_classic_2}. Then
\begin{equation}\label{variational}
\int_{\Omega} \langle D^2u,D^2\phi\rangle+{\rm Ric}(\nabla u,\nabla\phi)dv=\mu\int_{\Omega} u\phi dv\,,\ \ \ \forall\phi\in C^{\infty}(\Omega).
\end{equation}
From the definition of $H^2(\Omega)$ we deduce the validity of \eqref{variational_0}.

On the other hand, assume that there exist a solution $(u,\mu)\in (C^4(\Omega)\cap C^3(\overline\Omega))\times\mathbb R$ to \eqref{variational_0}. From \eqref{step-3}, by taking test functions $\phi\in C^{\infty}(\Omega)$ we immediately deduce that $u$ solves the differential equation \eqref{eigenvalue_eq} as well as the boundary conditions in \eqref{neumann_classic_2}, thus the pair $(u,\mu)$ is a solution of \eqref{neumann_classic_2}. This concludes the proof.
\end{proof}

We prove now Lemma \ref{equivalence_bc}

\begin{proof}[Proof of Lemma \ref{equivalence_bc}]
We start by proving \eqref{bc1}. Let $\left\{E_i\right\}_{i=1}^n$ an orthonormal frame in a neighborhood of a point $p\in\partial\Omega$ such that $\left\{E_i\right\}_{i=1}^{n-1}$ is a orthonormal frame of $\partial \Omega$ and $E_n=\nu$ is the outward unit normal to $\partial\Omega$. For a Lipschitz vector field $F$ in a neighborhood of $\partial \Omega$, we denote by $F_{\partial\Omega}:=\sum_{i=1}^{n-1}\langle F, E_i\rangle E_i$, hence in $p$
$$
F=F_{\partial \Omega}+\langle F,\nu\rangle\nu.
$$
Note that 
$$
\langle\nabla_{\nu}F,\nu \rangle=\langle\nabla_{\nu}(F_{\partial \Omega}+\langle F,\nu\rangle \nu),\nu \rangle=\langle\nabla_{\nu}F_{\partial \Omega},\nu \rangle+\langle\nabla_{\nu}\langle F,\nu\rangle \nu,\nu \rangle=\langle\nabla_{\nu}\langle F,\nu\rangle \nu,\nu \rangle,
$$
where we have used the fact that $\langle\nabla_{\nu}F_{\partial \Omega},\nu \rangle=0$. Moreover, $\sum_{i=1}^{n-1}\langle\nabla_{E_i}\nu,E_i\rangle={\rm div}\nu=(n-1)\mathcal H$. Thus we have
\begin{multline}\label{div_gen}
{\rm div}F_{|_{\partial \Omega}}=\sum_{i=1}^n\langle\nabla_{E_i}F,E_i\rangle=\sum_{i=1}^n\langle\nabla_{E_i}(F_{\partial\Omega}+\langle F,\nu\rangle\nu),E_i\rangle\\
=\sum_{i=1}^{n-1}\langle\nabla_{E_i}F_{\partial \Omega},E_i\rangle+\sum_{i=1}^n\langle\nabla_{E_i}\langle F,\nu\rangle\nu,E_i\rangle\\
={\rm div}_{\partial \Omega}F_{\partial \Omega}+\sum_{i=1}^{n-1}\langle\nabla_{E_i}\langle F,\nu\rangle\nu,E_i\rangle+\langle\nabla_{\nu}\langle F,\nu\rangle\nu,\nu \rangle\\
= {\rm div}_{\partial \Omega}F_{\partial \Omega}+\langle F,\nu\rangle\sum_{i=1}^{n-1}\langle\nabla_{E_i}\nu,E_i\rangle+\langle\nabla_{\nu}F,\nu \rangle\\
= {\rm div}_{\partial \Omega}F_{\partial \Omega}+(n-1)\mathcal H \langle F,\nu\rangle+\langle\nabla_{\nu}F,\nu \rangle,
\end{multline}
Now, noting that
$$
\nabla u_{|_{\partial \Omega}}=\nabla_{\partial \Omega}u+\frac{\partial u}{\partial\nu}\nu,
$$
and that, by definition
$$
\frac{\partial^2u}{\partial\nu^2}=\langle\nabla_{\nu}\nabla u,\nu\rangle.
$$
we immediately obtain from \eqref{div_gen} the following identity
\begin{multline*}
\Delta u_{|_{\partial \Omega}}={\rm div}_{\partial \Omega}\nabla_{\partial \Omega}u+(n-1)\mathcal H \frac{\partial u}{\partial\nu}+\langle\nabla_{\nu}\nabla u,\nu \rangle\\
=\Delta_{\partial \Omega}u+(n-1)\mathcal H\frac{\partial u}{\partial\nu}+\frac{\partial^2u}{\partial\nu^2}.
\end{multline*}
This proves \eqref{bc1}.

We prove now \eqref{bc2}. Let us consider the second boundary condition in \eqref{neumann_classic_2}. We need to show that 
$$
\Delta_{\partial \Omega}\left(\frac{\partial u}{\partial\nu}\right)-{\rm div}_{\partial M}S(\nabla_{\partial \Omega}u)={\rm div}_{\partial\Omega}\left(\nabla_{\nu}\nabla u\right)_{\partial \Omega},
$$
which can be re-written as
$$
{\rm div}_{\partial \Omega}\left(\nabla_{\partial \Omega}\left(\frac{\partial u}{\partial\nu}\right)-S(\nabla_{\partial \Omega}u)-\left(\nabla_{\nu}\nabla u\right)_{\partial\Omega}\right)=0.
$$
Actually we will prove that
$$
\nabla_{\partial \Omega}\left(\frac{\partial u}{\partial\nu}\right)-S(\nabla_{\partial\Omega}u)-\left(\nabla_{\nu}\nabla u\right)_{\partial\Omega}=0
$$
We note that for any vector fiel $X\in TM$
$$
\langle\nabla\left(\langle\nabla u,\nu\rangle\right),X\rangle=\langle\nabla_X\nabla u,\nu\rangle+\langle\nabla_X\nu ,\nabla u\rangle
=\langle\nabla_\nu\nabla u,X\rangle+\langle\nabla_{\nabla u}\nu ,X\rangle,
$$
since $D^2u$ and $\nabla\nu$ are symmetric. Thus $\nabla\left(\langle\nabla u,\nu\rangle\right)=\nabla_{\nu}\nabla u+\nabla_{\nabla u}\nu$. We have then
\begin{multline*}
\nabla_{\partial\Omega}\left(\frac{\partial u}{\partial\nu}\right)-S(\nabla_{\partial\Omega}u)-\left(\nabla_{\nu}\nabla u\right)_{\partial\Omega}=\\
\nabla_{\partial\Omega}\left(\langle\nabla u,\nu\rangle\right)-\nabla_{\nabla_{\partial \Omega}u}\nu-\nabla_{\nu}\nabla u+\langle\nabla_{\nu}\nabla u,\nu\rangle\nu\\
=\nabla\left(\langle\nabla u,\nu\rangle\right)-\langle\nabla\left(\langle\nabla u,\nu\rangle\right),\nu\rangle\nu-\nabla_{\nabla_{\partial \Omega}u}\nu-\nabla_{\nu}\nabla u+\langle\nabla_{\nu}\nabla u,\nu\rangle\nu\\
=\nabla_{\nu}\nabla u+\nabla_{\nabla u}\nu-\langle \nabla_{\nu}\nabla u,\nu\rangle\nu-\langle \nabla_{\nabla u}\nu,\nu\rangle\nu-\nabla_{\nabla_{\partial \Omega}u}\nu-\nabla_{\nu}\nabla u+\langle\nabla_{\nu}\nabla u,\nu\rangle\nu\\
=\nabla_{\nabla u}\nu-\nabla_{\nabla_{\partial \Omega}u}\nu-\langle \nabla_{\nabla u}\nu,\nu\rangle\nu=0,
\end{multline*}
since
$$
\nabla_{\nabla u}\nu=\nabla_{\nabla_{\partial \Omega}u}\nu+\nabla_{\langle\nabla u,\nu\rangle\nu}\nu=\nabla_{\nabla_{\partial \Omega}u}\nu
$$
and
$$
\langle \nabla_{\nabla u}\nu,\nu\rangle=\langle \nabla_{\nabla_{\partial \Omega}u}\nu,\nu\rangle+\langle \nabla_{\langle\nabla u,\nu\rangle\nu}\nu,\nu\rangle=0.
$$
In fact $\nabla_{\nu}\nu=0$ and $\langle\nabla_{\nabla_{\partial\Omega} u}\nu,\nu\rangle=0$. This proves \eqref{bc2}. The proof is now concluded.
\end{proof}

Since we will be interested in the variational problem \eqref{variational_0}, we can relax the hypothesis on the smoothness of $\Omega$. A sufficient condition for the solvability of \eqref{variational_0} is, e.g., that $\Omega$ is of class $C^1$, see Subsection \ref{sub_22}.
%Problem \eqref{neumann_classic} is then well-defined also in the case of domains of Riemannian manifolds and its variational formulation is given by \eqref{variational_0}. 

%In fact, boundary conditions in \eqref{neumann_classic} coincide with the well-known Neumann boundary conditions for the biharmonic operator on Euclidean domains (see e.g., \cite{kalamata,verchota}).

%The variational formulation \eqref{variational} actually makes sense for any $u,\phi\in H^2(\Omega)$, where $H^2(\Omega)$ is the standard Sobolev spaces of functions in $L^2(\Omega)$ with first and second weak derivatives in $L^2(\Omega)$. In what follows then we will mainly consider the weak formulation \eqref{variational_0}. In particular, in the next subsection we prove that problem \eqref{variational_0} admits an increasing sequence of eigenvalues of finite multiplicity, bounded from below, and diverging to plus infinity. We will then characterize the eigenvalues via the well-known min-max formula, which is an essential tool in order to study eigenvalue bounds.
\begin{rem}[The right quadratic form]\label{rem33}
By looking at \eqref{step-1} it is natural to ask what happens if we consider in the left-hand side of \eqref{variational_0} the more familiar quadratic form
\begin{equation}\label{quadratic_wrong}
\int_{\Omega} \Delta u\Delta \phi dv.
\end{equation}
The corresponding variational problem would read
\begin{equation}\label{variational_wrong}
\int_{\Omega} \Delta u\Delta \phi dv=\mu\int_{\Omega} u\phi dv\,,\ \ \ \forall\phi\in H^2(\Omega),
\end{equation}
in the unknowns $u\in H^2(\Omega)$, $\mu\in\mathbb R$. We note that this problem is not well-posed: it is immediate to see that all harmonic functions in $H^2(\Omega)$ are eigenfunctions corresponding to the eigenvalue $\mu=0$ . This is due to the fact that the quadratic form \eqref{quadratic_wrong} is not coercive in $H^2(\Omega)$, indeed we can add to the quadratic form \eqref{quadratic_wrong} a term $\gamma\int_{\Omega}u\phi dv$ with $\gamma>0$ arbitrarily large and obtain a scalar product whose induced norm is not equivalent to the standard one of $H^2(\Omega)$, see also Lemma \ref{coercivity}. In \cite{kalamata} it is proved that \eqref{variational_wrong} has an infinite kernel consisting of all harmonic functions in $H^2(\Omega)$. Moreover, if we rule out the kernel, problem \eqref{variational_wrong}  admits an increasing sequence of positive eigenvalues of finite multiplicity which coincide with the Dirichlet eigenvalues of the biharmonic operator. It is not difficult to adapt the results of \cite{kalamata} to the case of domains in a Riemannian manifold. The classical formulation of problem \eqref{variational_wrong} reads 
\begin{equation}\label{wrong_bc}
\begin{cases}
\Delta^2 u=\mu u\,, & {\rm in\ }\Omega,\\
\Delta u=0\,, & {\rm on\ }\partial\Omega,\\
\frac{\partial\Delta u}{\partial\nu}=0\,, & {\rm on\ }\partial\Omega.
\end{cases}
\end{equation}
We remark that Neumann boundary conditions are usually called ``natural boundary conditions'' and in a certain sense arises from ``solving a variational problem on the largest possible energy space'', which in this case is $H^2(\Omega)$. In this space, problem \eqref{variational_wrong} is evidently not well posed.

We also remark that the situation is completely different if we impose Dirichlet boundary conditions, namely if we consider problem
\begin{equation}\label{dirichlet_bc}
\begin{cases}
\Delta^2 u=\Lambda u\,, & {\rm in\ }\Omega,\\
u=0\,, & {\rm on\ }\partial\Omega,\\
\frac{\partial u}{\partial\nu}=0\,, & {\rm on\ }\partial\Omega,
\end{cases}
\end{equation}
in the unknowns $u$ (the eigenfunction) and $\Lambda$ (the eigenvalue). In this case, the corresponding weak formulation is
\begin{equation}\label{dirichlet_weak}
\int_{\Omega}\Delta u\Delta\phi dv=\Lambda\int_{\Omega}  u\phi dv\,,\ \ \ \forall\phi\in H^2_0(\Omega),
\end{equation}
in the unknowns $u\in H^2_0(\Omega)$, $\Lambda\in\mathbb R$. In this case boundary conditions are no more ``natural'' but are ``imposed'' with the choice of the energy space $H^2_0(\Omega)$. Actually, problem \eqref{dirichlet_weak} can be written in the form \eqref{variational_0} with the space $H^2(\Omega)$ replaced by $H^2_0(\Omega)$. In fact, it is easy to see that
\begin{equation}\label{dirichlet_equivalence}
\int_{\Omega}\Delta u\Delta\phi dv=\int_{\Omega} \langle D^2u,D^2\phi\rangle+{\rm Ric}(\nabla u,\nabla\phi)dv
\end{equation}
for all $u,\phi\in H^2_0(\Omega)$, see \eqref{step-1} and \eqref{step-3}. It turns out that \eqref{variational_0} and \eqref{variational_wrong} are equivalent in $H^2_0(\Omega)$. 

The situation is similar if $\Omega=M$ is a compact complete (boundaryless) smooth Riemannian manifold. In this case $H^2(M)=H^2_0(M)$ (see \cite{hebey}), hence \eqref{dirichlet_equivalence} holds for all $u,\phi\in H^2(M)$. Thus, the weak formulation of the biharmonic closed problem on $M$ is fairly simple, and actually it turns out that the eigenvalues of the biharmonic operator on $M$ coincide with the squares of the eigenvalues of the Laplacian on $M$, the eigenfunctions being the same. We refer to Subsection \ref{boundaryless} for more details.

Finally, we remark that one can also consider the variational problem
\begin{equation}\label{variational_3}
\int_{\Omega}\langle D^2u,D^2\phi\rangle dv=\mu\int_{\Omega}u\phi dv\,,\ \ \ \forall\phi\in H^2(\Omega),
\end{equation}
in the unknowns $u\in H^2(\Omega)$, $\mu\in\mathbb R$. As it is done in Subsection \ref{sub_22} it is possible to prove that problem \eqref{variational_3} is well-posed and admits an increasing sequence of non-negative eigenvalues of finite multiplicity. However, it is not always possible to recover an eigenvalue problem of the form \eqref{neumann_classic} starting from a smooth solution of \eqref{variational_3} as in the proof of Theorem \ref{equivalence}, except for few particular cases. In fact, by following the proof of Theorem \ref{equivalence}, we are left to deal with the term $\int_{\Omega}{\rm Ric}(\nabla u,\nabla\phi)dv$, and we would like to have an identity of the form
$$
\int_{\Omega}{\rm Ric}(\nabla u,\nabla\phi)dv=\int_{\Omega}L(u)\phi dv+\int_{\partial\Omega}B_1(u)\phi d\sigma+\int_{\partial\Omega}B_2(u)\frac{\partial\phi}{\partial\nu} d\sigma\,,\ \ \ \forall\phi\in H^2(\Omega),
$$
for suitable differential operators $L,B_1,B_2$. It is not in general possible to have explicit form for $L,B_1,B_2$ (they exist by Riesz Theorem). If $(M,g)$ is an Einstein manifold, that is, ${\rm Ric}=K g$, then $L(u)=-K\Delta u$, $B_1(u)=K\frac{\partial u}{\partial\nu}$ and $B_2(u)=0$. Thus, any smooth solution of \eqref{variational_3} solves
\begin{equation}\label{neumann_classic_einstein}
\begin{cases}
\Delta^2 u+K\Delta u=\mu u, & {\rm in\ }\Omega,\\
\frac{\partial^2u}{\partial\nu^2}=0, & {\rm on\ }\partial \Omega,\\
{\rm div}_{\partial \Omega}\left(\nabla_{\nu}\nabla u\right)_{\partial\Omega}+\frac{\partial\Delta u}{\partial\nu}+K\frac{\partial u}{\partial\nu}=0, & {\rm on\ }\partial \Omega.
\end{cases}
\end{equation}
Problem \eqref{neumann_classic_einstein} contains lower order terms in the eigenvalue equation and in the second boundary condition.
\end{rem}

%%%%%%%%%%%%%%%%%%%%%%%%%%%%%%%%%%%%%%%%%%%%%%%%%%%%%%%%%%%%%%%%%%%%%%%%%%%%%%%%%%%%%%%%%%%%%%%%%%%%%%%%%%%%%%%%%%%%%%%%%%%%%%%%%%%%%%%%

%%%%%%%%%%%%%%%%%%%%%%%%%%%%%%%%%%%%%%%%%%%%                NEUMANN EIGENVALUES            %%%%%%%%%%%%%%%%%%%%%%%%%%%%%%%%%%%%%%%%%%%%% 

%%%%%%%%%%%%%%%%%%%%%%%%%%%%%%%%%%%%%%%%%%%%%%%%%%%%%%%%%%%%%%%%%%%%%%%%%%%%%%%%%%%%%%%%%%%%%%%%%%%%%%%%%%%%%%%%%%%%%%%%%%%%%%%%%%%%%%%%

\subsection{Neumann eigenvalues of the biharmonic operator}\label{sub_22}

We prove here that, under suitable hypothesis on $\Omega$, problem \eqref{variational} admits an increasing sequence of eigenvalues of finite multiplicity bounded from below and diverging to $+\infty$. To do so we recast problem \eqref{variational_0} into an eigenvalue problem for a compact self-adjoint operator acting on a Hilbert space. First we note that \eqref{variational_0} can be re-written as
\begin{equation}\label{variational-2}
\int_{\Omega} \langle D^2u,D^2\phi\rangle+{\rm Ric}(\nabla u,\nabla\phi)+\gamma u\phi dv=\Gamma\int_{\Omega} u\phi dv\,,\ \ \ \forall\phi\in H^2(\Omega),
\end{equation}
where $\gamma\in\mathbb R$ is fixed, in the unknowns $u\in H^2(\Omega)$ and $\Gamma\in\mathbb R$. Clearly a pair $(u,\mu)\in H^2(\Omega)\times\mathbb R$ is a solution of \eqref{variational_0} if and only if the pair $(u,\mu+\gamma)\in H^2(\Omega)\times\mathbb R$ is a  solution of \eqref{variational-2}. We will study the eigenvalue problem in the equivalent formulation \eqref{variational-2} for suitable choices of $\gamma$.

We consider on $H^2(\Omega)$ the bilinear form 
\begin{equation}\label{prod}
\langle f,g\rangle_{\mathcal H^2(\Omega)}:=\int_{\Omega}\left(\langle D^2f,D^2g\rangle+{\rm Ric}(\nabla f,\nabla g)+\gamma fg\right)dv,
\end{equation}
with $\gamma>0$. We denote by $\mathcal H^2(\Omega)$ the space $H^2(\Omega)$ endowed with the form \eqref{prod}. We also set
\begin{equation}\label{norm}
\|f\|_{\mathcal H^2(\Omega)}^2:=\int_{\Omega}\left(|D^2f|^2+{\rm Ric}(\nabla f,\nabla f)+\gamma f^2\right)dv.
\end{equation}

We state the following lemma, whose proof we postpone at the end of this section.
\begin{lem}\label{coercivity}
Let $(M,g)$ be a complete $n$-dimensional smooth Riemannian manifold and let $\Omega$ be a bounded domain in $M$ with $C^1$ boundary. There exist $\gamma_0>0$ such that for all $\gamma>\gamma_0$, the bilinear form \eqref{prod} defines a scalar product in $H^2(\Omega)$ which induces on $H^2(\Omega)$ a norm which is equivalent to the standard one.
\end{lem}

Through all this subsection, we fix once and for all a positive number $\gamma>\gamma_0$, where $\gamma_0$ is as in Lemma \ref{coercivity}.

Then we define the operator $\mathcal P$ as an operator from $\mathcal H^2(\Omega)$ to its dual $\mathcal H^2(\Omega)'$ by setting
\begin{equation}
\mathcal P(u)[\phi]:=\int_{\Omega}\left(\langle D^2u,D^2\phi\rangle+{\rm Ric}(\nabla u,\nabla\phi)+\gamma u \phi\right) dv\,,\ \ \ \forall u,\phi\in \mathcal H^2(\Omega).
\end{equation}
By the Riesz Theorem it follows that $\mathcal P$ is surjective isometry. Then we consider the operator $J$ from $\mathcal H^2(\Omega)\subset L^2(\Omega)$  to $\mathcal H^2(\Omega)'$ defined by
\begin{equation}
J(u)[\phi]:=\int_{\Omega}u\phi dv\,,\ \ \ \forall u,\phi\in \mathcal H^2(\Omega).
\end{equation}
If the embedding $H^2(\Omega)\subset L^2(\Omega)$ is compact, then the operator $J$ is compact. Finally, we set 
\begin{equation}
T=\mathcal P^{(-1)}\circ J.
\end{equation}
If  $J$ is compact, since $\mathcal P$ is bounded, then also $T$ is compact. Moreover
$$
\langle T(u),\phi\rangle_{\mathcal H^2(\Omega)}=\langle u,\phi\rangle_{L^2(\Omega)},
$$
for all $u,\phi\in \mathcal H^2(\Omega)$. Hence $T$ is self-adjoint. Note that ${\rm Ker}\,T={\rm Ker}\, J=\left\{0\right\}$ and the non-zero eigenvalues of $T$ coincide with the reciprocals of the eigenvalues of \eqref{variational-2}, the eigenfunctions being the same. If $\Omega$ is of class $C^1$, then the embeddings $H^2(\Omega)\subset H^1(\Omega)\subset L^2(\Omega)$ are compact (see e.g., \cite[\S\,2]{aubin}).

We are now ready to  prove the following theorem.
\begin{thm}\label{eigenvalues}
Let $(M,g)$ be a smooth $n$-dimensional Riemannian manifold and let $\Omega$ be a bounded domain in $M$ with $C^1$ boundary. Then the eigenvalues of \eqref{variational_0} have finite multiplicity and are given by a non-decreasing sequence of real numbers $\left\{\mu_j\right\}_{j=1}^{\infty}$ bounded from below defined by
\begin{equation}\label{minmax}
\mu_j=\min_{\substack{U\subset H^2(M)\\{\rm dim}U=j}}\max_{0\ne u\in U}\frac{\int_{\Omega}|D^2u|^2+{\rm Ric}(\nabla u,\nabla u)dv}{\int_{\Omega} u^2 dv},
\end{equation}
where each eigenvalue is repeated according to its multiplicity.

Moreover, there exists a Hilbert basis of $\left\{u_j\right\}_{j=1}^{\infty}$ of $\mathcal H^2(\Omega)$ of eigenfunctions $u_j$ associated with the eigenvalues $\mu_j$. By normalizing the eigenfunctions with respect to \eqref{norm}, then $\big\{\frac{u_j}{\sqrt{\mu_j+\gamma}}\big\}_{j=1}^{\infty}$ define a Hilbert basis of $L^2(\Omega)$ with respect to its standard scalar product.
\end{thm}

\begin{proof}
By the Hilbert-Schmidt Theorem applied to the compact self-adjoint operator $T$ it follows that $T$ admits an increasing sequence of positive eigenvalues $\left\{q_j\right\}_{j=1}^{\infty}$, bounded from above, converging to zero and a corresponding Hilbert basis $\left\{u_j\right\}_{j=1}^{\infty}$ of eigenfunctions of $\mathcal H^2(\Omega)$. Since $q\ne 0$ is an eigenvalue of $T$ if and only if $\mu=\frac{1}{q}-\gamma$ is an eigenvalue of \eqref{variational_0} with the same eigenfunction, we deduce the validity of the first part of the statement. In particular, formula \eqref{minmax} follows from the standard min-max formula for the eigenvalues of compact self-adjoint operators.

To prove the final part of the theorem, we recast problem \eqref{variational-2} into an eigenvalue problem for the compact self-adjoint operator $T'=i\circ\mathcal P^{(-1)}\circ J'$,
where $ J'$ denotes the map from $L^2(\Omega)$ to  the dual of $\mathcal H^2(\Omega)$ defined by 
\begin{equation}
 J'(u)[\phi]:=\int_{\Omega}u\phi dv\,,\ \ \ \forall u\in L^2(\Omega),\phi\in \mathcal H^2(\Omega),
\end{equation}
and $i$ denotes the embedding of $\mathcal H^2(\Omega)$ into $L^2(\Omega)$. We apply again the Hilbert-Schmidt Theorem and observe that $T$ and $T'$ admit the same non-zero eigenvalues, and that the eigenfunctions of $T'$ can be chosen in $\mathcal H^2(\Omega)$ and coincide with the eigenfunctions of $T$. From \eqref{variational-2} we deduce that the normalized eigenfunction $u_j$ of $T$ with respect to \eqref{norm}, divided by $\sqrt{\mu_j+\gamma}$, form a orthonormal basis of $L^2(\Omega)$. This concludes the proof.
\end{proof}

We prove now Lemma \ref{coercivity}.

\begin{proof}[Proof of Lemma \ref{coercivity}]
It is easy to see that there exists $C>0$ (possibly depending on $\Omega$, $M$ and $\gamma$) such that for any $u\in H^2(\Omega)$ 
$$
\|u\|_{\mathcal H^2(\Omega)}^2\leq C \|u\|_{H^2(\Omega)}^2,
$$
in fact we can trivially take $C=\max\left\{1,\|{\rm Ric}\|_{L^{\infty}(\Omega)},\gamma\right\}$. 

We prove now the opposite inequality 
\begin{equation}\label{coerc0}
\|u\|_{\mathcal H^2(\Omega)}^2\geq \frac{1}{C} \|u\|_{H^2(\Omega)}^2,
\end{equation}
possibly re-defining the constant $C$. We note that for $\varepsilon\in(0,1)$
\begin{multline}
\|u\|_{\mathcal H^2(\Omega)}^2\geq \varepsilon \|u\|_{H^2(\Omega)}^2\\
+(1-\varepsilon)\int_{\Omega}\left(|D^2u|^2-\frac{\|{\rm Ric}\|_{L^{\infty}(\Omega)}+\varepsilon}{1-\varepsilon}|\nabla u|^2+\frac{\gamma-\varepsilon}{1-\varepsilon}u^2\right)dv
\end{multline}
Hence, in order to prove \eqref{coerc0} is is sufficient to prove that for any fixed $B>0$ there exists a constant $A>0$ such that
$$
\int_{\Omega}\left(|D^2u|^2-B|\nabla u|^2+A u^2\right)dv\geq 0.
$$
We argue by contradiction and assume that such constant does not exists. We find a sequence $\left\{u_k\right\}_{k=1}^{\infty}\subset H^2(\Omega)$ such that
$$
\int_{\Omega}\left(|D^2u_k|^2+k u_k^2\right)dv\leq B\int_{\Omega}|\nabla u_k|^2dv.
$$
We normalize the functions $u_k$ by setting $\int_{\Omega}|\nabla u_k|^2dv=1$. Hence $\int_{\Omega}|D^2u_k|^2dv\leq B$ and $\int_{\Omega}u_k^2dv\leq \frac{B}{k}$, thus the sequence $\left\{u_k\right\}_{k=1}^{\infty}$ is bounded in $H^2(\Omega)$. Passing to a subsequence, we have that $u_k\rightharpoonup\bar u$ in $H^2(\Omega)$ as $k\rightarrow +\infty$ (we have re-labeled the elements of the subsequence as $u_k$) and $u_k\rightarrow\bar u$ in $H^1(\Omega)$ by the compactness of the embedding $H^2(\Omega)\subset H^1(\Omega)$. Hence $\int_{\Omega}|\nabla \bar u|^2dv=\lim_{k\rightarrow+\infty}\int_{\Omega}|\nabla u_k|^2dv=1$ and $\int_{\Omega}\bar u^2dv=\lim_{k\rightarrow+\infty}\int_{\Omega}u_k^2dv=0$. Then we have found a function $\bar u\in H^2(\Omega)$ such that $\int_{\Omega}|\nabla\bar u|^2dv=1$ and $\int_{\Omega}\bar u^2dv=0$, a contradiction. This concludes the proof of \eqref{coerc0} and of the lemma.
\end{proof}

\section{A few properties of Neumann eigenvalues}\label{afew}

In this section we investigate a few properties of the eigenvalues $\mu_j$ of problem \eqref{variational_0}. In particular we study the behavior of the ratio $\frac{\mu_j}{m_j^2}$, where $m_j$ are the Neumann eigenvalues of the Laplacian on $\Omega$. In fact, in view of the asymptotic laws \eqref{weyl} and \eqref{weyl2}, it is natural to compare $\mu_j$ with $m_j^2$. In particular we show that this ratio can be arbitrarily large or arbitrarily close to zero. We denote by 
$$
0=m_1<m_2\leq\cdots\leq m_j\leq\cdots\nearrow +\infty
$$
the Neumann eigenvalues of the Laplacian on $\Omega$, which are given by
\begin{equation}\label{minmax_laplacian}
m_j=\min_{\substack{U\subset H^1(\Omega)\\{\rm dim}U=j}}\max_{0\ne u\in U}\frac{\int_{\Omega}|\nabla u|^2dv}{\int_{\Omega}u^2dv}.
\end{equation}
Here $H^1(\Omega)$ denotes the closure of $C^{\infty}(\Omega)$ with respect to the norm $\int_{\Omega}|\nabla u|^2+u^2 dv$.

We also consider the sign of the eigenvalues, proving that in some situations negative eigenvalues may appear. As a consequence we also provide examples where the ratio $\frac{\mu_j}{m_j^2}$ can be made negative and with arbitrarily large absolute value. In order to produce suitable examples, we restrict our analysis to the Euclidean space, to manifolds with ${\rm Ric}\geq (n-1)K>0$ and to the standard hyperbolic space $\mathbb H^n$.

%%%%%%%%%%%%%%%%%%%%%%%%%%%%%%%%%%%%%%%%%%%%%%%%%%%%%%%%%%%%%%%%%%%%%%%%%%%%%%%%%%%%%%%%%%%%%%%%%%%%%%%%%%%%%%%%%%%%%%%%%%%%%%%%%%%%%%%%

%%%%%%%%%%%%%%%%%%%%%%%%%%%%%%%%%%%%%%%%%%%%                     EUCLIDEAN SPACE           %%%%%%%%%%%%%%%%%%%%%%%%%%%%%%%%%%%%%%%%%%%%% 

%%%%%%%%%%%%%%%%%%%%%%%%%%%%%%%%%%%%%%%%%%%%%%%%%%%%%%%%%%%%%%%%%%%%%%%%%%%%%%%%%%%%%%%%%%%%%%%%%%%%%%%%%%%%%%%%%%%%%%%%%%%%%%%%%%%%%%%%

\subsection{Domains of the Euclidean space}\label{sub_eucl}
Through this subsection $(M,g)$ is the standard Euclidean space $\mathbb R^n$. It is well-known that if $\Omega$ is a bounded Lipschitz domain, then
$$
0=\mu_1=\mu_2=\cdots=\mu_{n+1}<\mu_{n+2}\leq\cdots\leq\mu_j\leq\cdots\nearrow+\infty,
$$
and the eigenspace corresponding to the eigenvalue $\mu=0$ is generated by $\left\{1,x_1,...,x_n\right\}$, see e.g., \cite{kalamata}.

We have the following theorem.
\begin{thm}\label{eucl}
For all $N\in\mathbb N$ there exists a sequence $\left\{\Omega_{\varepsilon,N}\right\}_{\varepsilon\in(0,\varepsilon_0)}$ such that
$$
\lim_{\varepsilon\rightarrow 0^+}\frac{\mu_j}{m_j^2}\rightarrow 0,
$$
for all $N+2\leq j\leq (N+1)(n+1)$,
\end{thm}
\begin{proof}
The domains providing the result are obtained by connecting with thin junctions a fixed domain $\Omega$ to $N$ domains $\Omega_1,...,\Omega_{N}$, $N\in\mathbb N$, of fixed volume and disjoint from $\Omega$, and by letting the size of the junctions go to zero. We will prove the theorem for $n=2$ and $N=1$.

For $\varepsilon\in(0,1)$, let $\Omega_{\varepsilon,1}=\Omega_{\varepsilon}:=\Omega_L\cup\Omega_R\cup R_{\varepsilon}$, where $\Omega_L=(-1,0)\times(0,1)$, $\Omega_R=(1,2)\times(0,1)$, and $R_{\varepsilon}:=\left\{x\in\mathbb R^2:0\leq x_1\leq 1\,,0<x_2<\varepsilon\right\}$.

Let $\left\{m_j\right\}_{j=1}^{\infty}$ denote the Neumann eigenvalues of $\Omega_L\cup\Omega_R$, let $\left\{\xi_j\right\}_{j=1}^{\infty}$ denote the eigenvalues of $-f''(t)=\xi f(t)$ in $(0,1)$ with Dirichlet boundary conditions, and let $\left\{m_j^{\varepsilon}\right\}_{j=1}^{\infty}$ denote the Neumann eigenvalues of $\Omega_{\varepsilon}$. It is known that the sequence  $\left\{m_j^{\varepsilon}\right\}_{j=1}^{\infty}$ converges to the sequence $\left\{m_j\right\}_{j=1}^{\infty}\cup \left\{\xi_j\right\}_{j=1}^{\infty}$, where in the union the eigenvalues have been ordered increasingly, see e.g., \cite{arrieta00}.
In particular, $m_1^{\varepsilon}=0$, $m_2^{\varepsilon}\rightarrow 0$ as $\varepsilon\rightarrow 0^+$, and $m_j^{\varepsilon}$ are uniformly bounded from below by some positive constant independent on $\varepsilon$ for $j\geq 3$.

Let now $\mu_j^{\varepsilon}$ denote the eigenvalues of \eqref{variational_0} in $\Omega_{\varepsilon}$. We prove that $\mu_j^{\varepsilon}\leq C\varepsilon$ for $j\leq 6$, where $C>0$ does not depend on $\varepsilon$.

To do so, let $\phi_L(x_1,x_2)\in C^{2}(\mathbb R^2)$ be such that $\phi_L(x_1,x_2)=1$ if $x_1<0$, $0\leq \phi_L(x_1,x_2)\leq 1$ for $0\leq x_1\leq\frac{1}{2}$, and $\phi_L(x_1,x_2)=0$ if $x_1>\frac{1}{2}$. By construction $|D^2\phi_L(x_1,x_2)|\leq c$ for some $c>0$. We define analogously $\phi_R(x_1,x_2)\in C^{2}(\mathbb R^2)$ which is supported in $\left\{x_1>\frac{1}{2}\right\}$ by setting $\phi_R(x_1,x_2)=\phi_L(1-x_1,x_2)$.

We set $u_L^1={\phi_L}_{|_{\Omega_{\varepsilon}}}$, $u_L^2={x_1\cdot \phi_L}_{|_{\Omega_{\varepsilon}}}$, $u_L^3={x_2\cdot \phi_L}_{|_{\Omega_{\varepsilon}}}$, $u_R^1={\phi_R}_{|_{\Omega_{\varepsilon}}}$, $u_R^2={x_1\cdot \phi_R}_{|_{\Omega_{\varepsilon}}}$, $u_R^3={x_2\cdot \phi_R}_{|_{\Omega_{\varepsilon}}}$. These functions are linearly independent and belong to $H^2(\Omega_{\varepsilon})$. Moreover, any $u$ in the space generated by $u_L^1,u_L^2,u_L^3,u_R^1,u_R^2,u_R^3$ with $\int_{\Omega_{\varepsilon}}u^2dv=1$ is easily seen to satisfy
$$
\int_{\Omega}|D^2u|^2dv\leq C\varepsilon,
$$
with $C$ independent of $\varepsilon$. This implies from \eqref{minmax} that $\mu_j^{\varepsilon}\leq C\varepsilon$ for $j\leq 6$. The proof is now complete in the case $n=2$ and $N=1$.

The proof for $n>2$ and $N>1$ is a standard adaptation of the arguments above.
\end{proof}

\begin{rem}
Let us consider a domain $\Omega_{\varepsilon,N}$ as in the proof of Theorem \ref{eucl}. Such a domain is usually called a {\it $N+1$-dumbbell}. We observe that if $N>n$, we have that $m_j,\mu_j\rightarrow 0$ as $\varepsilon\rightarrow 0^+$ for $n+2\leq j <N+2$. It is well-known that $m_1=0$ and $m_j=O(\varepsilon^{n+1})$ as $\varepsilon\rightarrow 0^+$ for $2\leq j\leq N+1$, see e.g., \cite{anne,jimbomorita}. Moreover, it is possible to show that in the case of a sufficiently regular $N+1$-dumbbell domain $\Omega_{\varepsilon,N}$, $\mu_j\rightarrow 0$ as $\varepsilon\rightarrow 0^+$ for $n+2\leq j\leq (N+1)(n+1)$, and  $\mu_{(N+1)(n+1)+1}$ is bounded away from zero, uniformly in $\varepsilon$. We refer to \cite{ferraresso} for the proof in the case $N=1$. Thanks to this fact, with the same arguments of \cite{anne} (see also \cite{jimbomorita}) it is possible to prove that $\mu_j=O(\varepsilon^{n-1})$ as $\varepsilon\rightarrow 0^+$ for $n+2\leq j\leq (N+1)(n+1)$. We omit the details of the computations which are standard but quite technical and go beyond the scopes of the present article. Anyway, we have that $\frac{\mu_j}{m_j^2}\rightarrow +\infty$ as $\varepsilon\rightarrow 0$ for all $n+2\leq j<N+2$. Thus in the Euclidean case dumbbell domains provide examples where either $\mu_j<m_j^2$ (for certain $j\in\mathbb N$) or  $\mu_j>m_j^2$ (for other $j\in\mathbb N$). 
\end{rem}

%As we shall see in the following subsections,  the situation described by Theorem \ref{eucl} is specific of the Euclidean space. Very different situations occur if ${\rm Ric}\geq (n-1)K>0$ or in the hyperbolic space $\mathbb H^n$. 

%%%%%%%%%%%%%%%%%%%%%%%%%%%%%%%%%%%%%%%%%%%%%%%%%%%%%%%%%%%%%%%%%%%%%%%%%%%%%%%%%%%%%%%%%%%%%%%%%%%%%%%%%%%%%%%%%%%%%%%%%%%%%%%%%%%%%%%%

%%%%%%%%%%%%%%%%%%%%%%%%%%%%%%%%%%%%%%%%%%%%                     POSITIVE RIC              %%%%%%%%%%%%%%%%%%%%%%%%%%%%%%%%%%%%%%%%%%%%% 

%%%%%%%%%%%%%%%%%%%%%%%%%%%%%%%%%%%%%%%%%%%%%%%%%%%%%%%%%%%%%%%%%%%%%%%%%%%%%%%%%%%%%%%%%%%%%%%%%%%%%%%%%%%%%%%%%%%%%%%%%%%%%%%%%%%%%%%%

\subsection{Domains in manifolds with ${\rm Ric}\geq (n-1)K>0$}\label{few_pos}
Through all this subsection $(M,g)$ will be a complete $n$-dimensional smooth Riemannian manifold with ${\rm Ric}\geq (n-1)K>0$. We note that for any domain $\Omega$ of class $C^1$ of $M$ we have $\mu_1=0$ and $\mu_2>0$. In fact, from \eqref{minmax} we immediately deduce that $\mu_j\geq 0$ for all $j\in\mathbb N$ and that  $\mu_1=0$ is an eigenvalue with corresponding eigenfunctions the constant functions. Constant functions are the only eigenfunctions associated with $\mu_1=0$. In fact, any eigenfunction $u$ corresponding to the eigenvalue $\mu=0$ satisfies 
$$
\int_{\Omega}|D^2u|^2+{\rm Ric}(\nabla u\nabla u)dv=0,
$$
hence $|\nabla u|=0$, thus $u$ is a constant. This implies that $\mu_2>0$. 

Let us denote by
$$
0\leq\eta_1\leq\eta_2\leq\cdots\leq\eta_j\leq\cdots\nearrow+\infty
$$
the eigenvalues of the rough Laplacian on $\Omega$ with Neumann boundary conditions. They are characterized by 
\begin{equation}\label{minmax_rough_laplacian}
\eta_j=\min_{\substack{\mathcal W\subset \mathcal H^1(\Omega)\\{\rm dim}\mathcal W=j}}\max_{0\ne \omega\in \mathcal W}\frac{\int_{\Omega}|\nabla \omega|^2dv}{\int_{\Omega}\omega^2dv},
\end{equation}
where $\mathcal H^1(\Omega)$ is the space of $1$-forms of class $H^1(\Omega)$, see e.g., \cite{colbois_rough} for more information on the eigenvalues of the rough Laplacian. We have the following.

\begin{thm}\label{lower_pos}
 Let $(M,g)$ be a complete $n$-dimensional smooth Riemannian manifold with ${\rm Ric}\geq (n-1)K>0$ and let $\Omega$ be a bounded domain in $M$ with $C^1$ boundary. Then
$$
\mu_j\geq\left(\eta_1+(n-1)K\right)m_j
$$
for all $j\in\mathbb N$.
\end{thm}

\begin{proof}
The inequality is trivially true for $j=1$. Hence we assume $j\geq 2$. We observe that for any non-constant $u\in H^2(\Omega)$
\begin{multline}
\frac{\int_{\Omega}|D^2u|^2+{\rm Ric}(\nabla u,\nabla u)dv}{\int_{\Omega}u^2dv}=\frac{\int_{\Omega}|D^2u|^2dv}{\int_{\Omega}|\nabla u|^2dv}\cdot \frac{\int_{\Omega}|\nabla u|^2dv}{\int_{\Omega}u^2dv}+\frac{\int_{\Omega}{\rm Ric}(\nabla u,\nabla u)dv}{\int_{\Omega}u^2dv}\\
\geq \eta_1 \frac{\int_{\Omega}|\nabla u|^2dv}{\int_{\Omega}u^2dv}+(n-1)K\frac{\int_{\Omega}|\nabla u|^2dv}{\int_{\Omega}u^2dv}=\left(\eta_1+(n-1)K\right)\frac{\int_{\Omega}|\nabla u|^2dv}{\int_{\Omega}u^2dv}.
\end{multline}
Hence, for any subspace $U\subset H^2(M)$ of dimension $j\geq 2$ we have
$$
\max_{0\ne u\in U}\frac{\int_{\Omega}|D^2u|^2+{\rm Ric}(\nabla u,\nabla u)dv}{\int_{\Omega}u^2dv}\geq \left(\eta_1+(n-1)K\right)\max_{0\ne u\in U}\frac{\int_{\Omega}|\nabla u|^2dv}{\int_{\Omega}u^2dv}
$$

This implies
\begin{multline}
\mu_j=\min_{\substack{U\subset H^2(\Omega)\\{\rm dim}U=j}}\max_{0\ne u\in U}\frac{\int_{\Omega}|D^2u|^2+{\rm Ric}(\nabla u,\nabla u)dv}{\int_{\Omega}u^2dv}\\
\geq \left(\eta_1+(n-1)K\right)\min_{\substack{U\subset H^2(\Omega)\\{\rm dim}U=j}}\max_{0\ne u\in U}\frac{\int_{\Omega}|\nabla u|^2dv}{\int_{\Omega}u^2dv}\\
\geq \left(\eta_1+(n-1)K\right)\min_{\substack{U\subset H^1(\Omega)\\{\rm dim}U=j}}\max_{0\ne u\in U}\frac{\int_{\Omega}|\nabla u|^2dv}{\int_{\Omega}u^2dv}=\left(\eta_1+(n-1)K\right)m_j,
\end{multline}
where in the last inequality we have used the fact that $H^2(\Omega)\subset H^1(\Omega)$, hence the minimum decreases. The proof is now complete.
\end{proof}

\begin{rem}
Note that if $\Omega$ is a bounded domain with $II\geq 0$, we have that
$$
m_2\geq nK,
$$
with equality if and only if $\Omega$ is isometric to an $n$-dimensional Euclidean hemisphere of curvature $K$, see e.g., \cite{escobar}. This result is in the spirit of the well-known Obata-Lichnerowicz inequality, see \cite{chavel,Lichnerowicz,Obata}. Hence, for any bounded domain $\Omega$  with $II\geq 0$ we have from Theorem \ref{lower_pos} that
$$
\mu_2\geq (\eta_1+(n-1)K)nK.
$$
It is natural to conjecture that 
\begin{equation}\label{obata}
\mu_2\geq n^2K^2.
\end{equation}
\end{rem}
{\bf Open problem.} Prove \eqref{obata}.\\
\vskip .2pt

Thanks to Theorem \ref{lower_pos} we have the following inequality for all $j\in\mathbb N$, $j\geq 2$
\begin{equation}
\frac{\mu_j}{m_j^2}\geq\frac{\left(\eta_1+(n-1)K\right)}{m_j}.
\end{equation}
We recall now that for any $N\in\mathbb N$ there exists a sequence $\left\{\Omega_{\varepsilon,N}\right\}_{\varepsilon\in(0,\varepsilon_0)}$ such that $m_j\leq C\varepsilon$ for all $j\leq N$. These domains are obtained by connecting to a fixed domain $\Omega$, $N$ domains of fixed volume and disjoint from $\Omega$ with thin junctions, and by letting the width of the channels, represented by the parameter $\varepsilon>0$, go to zero. This is a standard construction (see \cite{anne,arrieta00}, see also Subsection \ref{sub_eucl}).

This implies the validity of the following theorem.
\begin{thm}
Let $(M,g)$ be a complete $n$-dimensional smooth Riemannian manifold with ${\rm Ric}\geq (n-1)K>0$. For all $N\in\mathbb N$ there exist a sequence $\left\{\Omega_{\varepsilon,N}\right\}_{\varepsilon\in(0,\varepsilon_0)}$ of domains such that
\begin{equation}
\lim_{\varepsilon\rightarrow 0^+}\frac{\mu_j}{m_j^2}=+\infty,
\end{equation}
for all $2\leq j\leq N$.
\end{thm}

On the other hand, if $\Omega$ is such that the second fundamental form of its boundary is non-negative, that is, $II\geq 0$, we have that
\begin{equation}\label{convex_ratio}
\frac{\mu_j}{m_j^2}\leq 1.
\end{equation}
We refer to Subsection \ref{sub_convex} for the proof of \eqref{convex_ratio}. 

%We observe then that there is not a monotonicity relation between $\mu_j$ and $m_j^2$. Indeed, there exist domains for which $\mu_j>m_j^2$ and domains for which $\mu_j\leq m_j^2$.

%%%%%%%%%%%%%%%%%%%%%%%%%%%%%%%%%%%%%%%%%%%%%%%%%%%%%%%%%%%%%%%%%%%%%%%%%%%%%%%%%%%%%%%%%%%%%%%%%%%%%%%%%%%%%%%%%%%%%%%%%%%%%%%%%%%%%%%%
%%%%%%%%%%%%%%%%%%%%%%%%%%%%%%%%%%%%%%%%%%%%%%%%%%%%%%%%%%%%%%%%%%%%%%%%%%%%%%%%%%%%%%%%%%%%%%%%%%%%%%%%%%%%%%%%%%%%%%%%%%%%%%%%%%%%%%%%
%%%%%%%%%%%%%%%%%%%%%%%%%%%%%%%%%%%%%%%%%%%%%%%%%%%%%%%%%%%%%%%%%%%%%%%%%%%%%%%%%%%%%%%%%%%%%%%%%%%%%%%%%%%%%%%%%%%%%%%%%%%%%%%%%%%%%%%%

%%%%%%%%%%%%%%%%%%%%%%%%%%%%%%%%%%%%%%%%%%%%                  hyperbolic SPACE             %%%%%%%%%%%%%%%%%%%%%%%%%%%%%%%%%%%%%%%%%%%%% 

%%%%%%%%%%%%%%%%%%%%%%%%%%%%%%%%%%%%%%%%%%%%%%%%%%%%%%%%%%%%%%%%%%%%%%%%%%%%%%%%%%%%%%%%%%%%%%%%%%%%%%%%%%%%%%%%%%%%%%%%%%%%%%%%%%%%%%%%
%%%%%%%%%%%%%%%%%%%%%%%%%%%%%%%%%%%%%%%%%%%%%%%%%%%%%%%%%%%%%%%%%%%%%%%%%%%%%%%%%%%%%%%%%%%%%%%%%%%%%%%%%%%%%%%%%%%%%%%%%%%%%%%%%%%%%%%%
%%%%%%%%%%%%%%%%%%%%%%%%%%%%%%%%%%%%%%%%%%%%%%%%%%%%%%%%%%%%%%%%%%%%%%%%%%%%%%%%%%%%%%%%%%%%%%%%%%%%%%%%%%%%%%%%%%%%%%%%%%%%%%%%%%%%%%%%

\subsection{Domains of the hyperbolic space}\label{neg_eig}

Given an eigenvalue $\mu$ of \eqref{variational_0} and a corresponding eigenfunction $u_{\mu}\in H^2(\Omega)$, we have
$$
\mu=\frac{\int_{\Omega}|D^2u_{\mu}|^2+{\rm Ric}(\nabla u_{\mu},\nabla u_{\mu})dv}{\int_{\Omega}u_{\mu}^2dv}.
$$
It is well-known, from Bochner's formula \eqref{bochner_f} and integration by parts, that, for any $u\in H^2_0(\Omega)$
$$
\int_{\Omega}|D^2u|^2+{\rm Ric}(\nabla u,\nabla u)dv=\int_{\Omega}(\Delta u)^2dv\geq 0.
$$
Note that for a complete, compact (boundaryless) smooth Riemannian manifold $M$  
$$
\int_{M}|D^2u|^2+{\rm Ric}(\nabla u,\nabla u)dv=\int_{M}(\Delta u)^2dv\geq 0,
$$
for all $u\in H^2(M)$.

In view of this,  the  natural question arises whether the biharmonic Neumann eigenvalues $\mu_j$ can be negative or not. Clearly, a necessary condition for the appearance of negative eigenvalues is that ${\rm Ric}\not\geq 0$. In this section we consider domains of the standard hyperbolic space $\mathbb H^n$. We have the following theorem.

\begin{thm}\label{negative_n}
Let $\Omega$ be a bounded domain of the hyperbolic space $\mathbb H^n$ with $C^1$ boundary. Then $\Omega$ admits at least $n$ strictly negative eigenvalues.
\end{thm}

\begin{proof}
We start by proving the result for $n=2$. To do so, we will use Fermi coordinates for $\mathbb H^2$. In this case the metric is given by
$$
g(x,y)=dx^2+\cosh^2(x)dy^2.
$$
The Christoffel symbols are 
$$
\Gamma_{1,1}^1=\Gamma_{1,1}^2=\Gamma_{1,2}^1=\Gamma_{2,2}^2=0\,,\ \ \ \Gamma_{1,2}^2=\tanh(x),\,\ \ \ \Gamma_{2,2}^1=-2\cosh(x)\sinh(x).
$$
For any smooth function $u(x,y)$ we have
$$
\nabla u=u_xdx+u_ydy,
$$
hence
$$
|\nabla u|^2=u_x^2+\frac{u_y^2}{\cosh(x)^2}.
$$
Moreover,
\begin{multline*}
D^2u=u_{xx} dx\otimes dx+\left(u_{xy}-\tanh(x)u_y\right)(dx\otimes dy+dy\otimes dx)\\
+(u_{yy}+\cosh(x)\sinh(x)u_x)dy\otimes dy,
\end{multline*}
therefore
$$
|D^2u|^2=u_{xx}^2+\frac{2\left(u_{xy}-\tanh(x)u_y\right)^2}{\cosh(x)^2}+\frac{(u_{yy}+\cosh(x)\sinh(x)u_x)^2}{\cosh(x)^4}
$$
A natural test function for the Rayleigh quotient in \eqref{minmax} is $\tilde u(x,y)=x$, which is the signed distance from the geodesic $x=0$. We have then
$$
|D^2u|^2+{\rm Ric}(\nabla u,\nabla u)=|D^2u|^2-|\nabla u|^2=\tanh(x)^2-1<0.
$$
This implies that
$$
\mu_1=\min_{u\in H^2(\Omega)}\frac{\int_{\Omega}|D^2u|^2+{\rm Ric}(\nabla u,\nabla u)dv}{\int_{\Omega}u^2}\leq\frac{\int_{\Omega}|D^2\tilde u|^2+{\rm Ric}(\nabla \tilde u,\nabla \tilde u)dv}{\int_{\Omega}\tilde u^2}<0.
$$
Thus, $\mu_1<0$. Now, fixed a domain $\Omega$, let $\mu_1<0$ be the first  eigenvalue of \eqref{variational_0} and let $u_1$ be an associated eigenfunction. Let $p\in\Omega$, $v\ne 0$ be a vector in $T_p\mathbb H^2$, and let $\left\{\gamma_{\theta}\right\}_{\theta\in[0,2\pi)}$ be the family of geodesics with $\gamma_{\theta}(0)=p$ and with the angle between $\gamma_{\theta}'(0)$ and $v$ equals to $\theta$. Let $h_{\theta}$ be the signed distance to $\gamma_{\theta}$. For all $\theta\in[0,2\pi)$, $|D^2h_{\theta}|^2+{\rm Ric}(\nabla h_{\theta},\nabla h_{\theta})<0$. Indeed we can perform the same computations above in a new system of Fermi coordinates where $x=h_{\theta}$. Moreover, we have that $h_{\pi}=-h_{0}$.\\
Let us consider the function $\theta\mapsto\rho(\theta)$ defined by $\rho(\theta):=\int_{\Omega}h_{\theta}u_1dv$. The function $\rho$ is continuous and satisfies $\rho(0)=-\rho(\pi)$. Therefore, there exists $0\leq\theta_0\leq\pi$ with $\rho(\theta_0)=0$. Hence, there exists at least one function $h_{\theta}$ with strictly negative Rayleigh quotient and orthogonal to $u_1$. From \eqref{minmax} we deduce that
$$
\mu_2=\min_{\substack{0\ne u\in H^2(\Omega)\\\int_{\Omega}uu_1dv=0}}\frac{|D^2u|^2+{\rm Ric}(\nabla u,\nabla u)dv}{u^2dv}\leq \frac{|D^2h_{\theta}|^2+{\rm Ric}(\nabla h_{\theta},\nabla h_{\theta})dv}{h_{\theta}^2dv}<0.
$$
This proves the existence of a second strictly negative eigenvalue. This concludes the proof in the case $n=2$.

The proof for $n\geq 3$ is similar, however we shall highlight only the main differences, omitting the standard but quite long analogous computations. 

The main tool in order to prove the result for $n\geq 3$ is the Borsuk-Ulam Theorem which states that if $g:\mathbb S^n\rightarrow \mathbb R^n$ is an odd function (that is, $g(p)=-g(-p)$ where $-p$ is the antipodal point to $p$ in $\mathbb S^n$), then there exists $p\in\mathbb S^n$ such that $g(p)=0$.

Let $q\in\Omega$ and let $\mathcal H_1$ be an hyperplane containing $q$. Let $f_1$ be the signed distance from $\mathcal H_1$. Then
$$
|D^2f_1|^2+{\rm Ric}(\nabla f_1,\nabla f_1)dv<0.
$$
The proof is analogous to that of the case $n=2$. It follows by explicit computations in Fermi coordinates $(x_1,...,x_n)$ where $x_1$ represents the signed distance from $\mathcal H_1$ and $(x_2,...,x_n)$ are normal coordinates on $\mathcal H_1=\mathbb H^{n-1}$ (in this system $q=(0,...,0)$). Therefore $\mu_1<0$ with $u_1$ a corresponding eigenfunction. 

Let $\pi_2$ be a fixed plane in $T_q\mathbb H^n$ and let $v_1\in\pi_2$ be a non-zero vector. For $\theta\in[0,2\pi)$, let $\mathcal H_{\theta}$ be the hyperplane whose tangent space at $q$ is normal to the vector $v_{\theta}\in\pi_2$, where $v_{\theta}$ is a unit vector which forms with $v_1$ an angle of width $\theta$ in $\pi_1$. Let $f_{\theta}$ be the signed distance from $\mathcal H_{\theta}$. The Rayleigh quotient of this function is again strictly negative. Moreover, we have that $f_{0}=-f_{\pi}$ and if we define $\rho_1(\theta):=\int_{\Omega}f_{\theta}u_1dv$, we find out that there exists $\theta\in[0,2\pi)$ such that $\rho_1(\theta)=0$. Thus we deduce the existence of a function with strictly negative Rayleigh quotient orthogonal to $u_1$. As in the case $n=2$, we deduce that $\mu_2<0$. Assume now that we have $\mu_1,...,\mu_k<0$, with $k<n$, and with associated eigenfunctions $u_1,...,u_k$. We prove that $\mu_{k+1}<0$. 

Let $\pi_{k+1}$ be a fixed $k+1$-dimensional subspace of $T_q\mathbb H^n$ and let $v_k$ be a non-zero vector in $\pi_{k+1}$. Let, for $\theta=(\theta_1,...,\theta_{k})\in\mathbb S^k$, $v_{\theta}$ be a vector in $\pi_{k+1}$ forming a directional angle $\theta=(\theta_1,...,\theta_{k})$ with $v_k$. Let $\mathcal H_{\theta}$ be the hyperplane whose tangent space at $q$ is normal to the vector $v_{\theta}\in\pi_{k+1}$. Let $f_{\theta}$ be the signed distance from $\mathcal H_{\theta}$. The Rayleigh quotient of this function is strictly negative. Moreover, we have that $f_{\theta}=-f_{-\theta}$ for all $\theta\in\mathbb S^k$. We define now $\rho_k:\mathbb S^k\rightarrow\mathbb S^k$ by $\rho_k(\theta):=(\int_{\Omega}f_{\theta}u_1dv,...,\int_{\Omega}f_{\theta}u_kdv)$. By construction $\rho_k$ is continuous and odd, hence there exists $\theta\in\mathbb S^k$ such that $\rho_k(\theta)=0$, hence  $\int_{\Omega}f_{\theta}u_1dv=\cdots=\int_{\Omega}f_{\theta}u_kdv=0$. Therefore there exists a function in $H^2(\Omega)$ with strictly negative Rayleigh quotient and orthogonal to $u_1,...,u_k$. From \eqref{minmax} we deduce that $\mu_{k+1}<0$. 

This concludes the proof.
\end{proof}

%\begin{thm}\label{negative_2}
%Let $\Omega$ be a bounded Lipschitz domain of the hyperbolic plane $\mathbb H^2$. Then $\Omega$ admits at least $2$ strictly negative eigenvalues.
%\end{thm}

%\begin{proof}

%\end{proof}

\begin{rem}\label{rem_rough}
We can give an upper bound on the number of negative eigenvalues of $\Omega$ in terms of the number of eigenvalues of the rough Laplacian smaller than one. Indeed, if we have $\mu_1,...,\mu_N$ negative eigenvalues with corresponding eigenfunctions $u_1,...,u_N$, then any $u=\sum_{i=1}^N\alpha_i u_i$ is such that
$$
\int_{\Omega}|D^2u|^2+{\rm Ric}(\nabla u,\nabla u)dv=\int_{\Omega}|D^2u|^2-|\nabla u|^2dv<0.
$$ 
%Hence, the number $N$ is bounded above by the number of eigenvalues $\eta_j$ of the rough Laplacian less than $1$.
\end{rem}

To end this section we show that there exists domains with an arbitrary number of arbitrarily large (in absolute value) negative eigenvalues. 

\begin{thm}\label{BIG_negative}
For any $N\in\mathbb N$ and $M>0$ there exist a bounded domain $\Omega_{M,N}$ of the hyperbolic space $\mathbb H^n$ with $|\Omega_{M,N}|=1$ and such that
$$
\mu_j\leq -M,
$$
for all $j\leq N$.
\end{thm}
\begin{proof}
We start by proving the theorem with $N=1$. Let $\gamma$ be a simple geodesic in $\mathbb H^n$ and let $\gamma_{\delta}$ a $\delta$-neighborhood of $\gamma$, that is $\gamma_{\delta}:=\left\{p\in \mathbb H^n:{\rm dist}(p,\gamma)<\delta\right\}$. In $\gamma_{\delta}$ we consider Fermi coordinates $(x_1,...,x_n)$, where $(x_1,0,...,0)$ correspond to the points on $\gamma$ and $(0,x_2,...,x_n)$ correspond to a normal neighborhood of $0$. Moreover, $g_{ij}(p)=\delta_{ij}$ and $\frac{\partial g_{ij}}{\partial x_k}(p)=0=\Gamma_{ij}^k(p)$ for all $p\in\gamma$.

Given $\varepsilon>0$, we find $\delta>0$ such that $|g_{ij}(p)-\delta_{ij}|<\varepsilon$, $\left|\frac{\partial g_{ij}}{\partial x_k}(p)\right|<\varepsilon$ and $|\Gamma_{ij}^k(p)|<\varepsilon$ for all $p\in\gamma_{\delta}$.

On the domain $D_{\delta,L}:=\left\{p\in\gamma_{\delta}:0<x_1<L\right\}$ we consider $n$ test functions $x_i$, $i=2,...,n$. We have, for all $i=2,...n$
$$
|D^2x_i|^2=|\nabla dx_i|^2\leq C'\varepsilon
$$
for some $C'>0$ independent of $\varepsilon$, and
$$
{\rm Ric}(\nabla x_i,\nabla x_i)=-1.
$$
Therefore
$$
\int_{D_{\delta,L}}|D^2x_i|^2+{\rm Ric}(\nabla x_i,\nabla x_i)dv\leq (C\varepsilon-1)L\delta^{n-1},
$$
while
$$
\int_{D_{\delta,L}}x_i^2dv\geq C''L\delta^{n-1}\delta^2,
$$
for some constant $C''>0$ independent of $\varepsilon,\delta$, since in $\gamma_{\delta}\setminus\gamma_{\delta/2}$, $|x_i|^2\geq\frac{\delta^2}{4}$. By choosing $\varepsilon>0$ sufficiently small, we conclude that
$$
\mu_1\leq -\frac{C}{\delta^2}.
$$
Moreover, by choosing $L=\frac{1}{\delta^{n-1}}$ we have that $|D_{\delta,L}|=O(1)$ as $\delta\rightarrow 0^+$. This proves the statement for $N=1$.

Let $N\in\mathbb N$ be fixed. Consider the domain $D_{\delta,NL}$. Let the points $p_1,...,p_N\in\gamma$ be given by $p_i=\left(L\left(i-\frac{1}{2}\right),0,...,0\right)$ for $i=1,...,N$ and let $B_i:=B\left(p_i,\frac{L}{4}\right)$ and $2B_i:=B\left(p_i,\frac{L}{2}\right)$. The balls $2B_i$ are disjoint. Associated with each $2B_i$ we define cut-off functions $\phi_i$ by
$$\phi_i(p)=
\begin{cases}
1\,, & {\rm if\ }p\in B_i,\\
4\left(\frac{2\,{\rm dist}(p,p_i)}{L}-1\right)^2\left(4\frac{2\,{\rm dist}(p,p_i)}{L}-1\right)\,, & {\rm if\ }p\in 2B_i\setminus\overline B_i,\\
0\,, & {\rm if\ }p\in \mathbb H^n\setminus\overline 2B_i.
\end{cases}
$$
The functions $u_i:=x_k\phi_i$, $i=1,...N$ (for some $k=2,...,n$) are $N$ disjointly supported functions in $H^2(D_{\delta,L})$, hence from \eqref{minmax} we deduce that
$$
\mu_N\leq\max_{i=1,...,N}\frac{\int_{D_{\delta,NL}}|D^2u_i|^2+{\rm Ric}(\nabla u_i,\nabla u_i)dv}{\int_{D_{\delta,NL}}u_i^2dv}.
$$
The same computations above show that the Rayleigh quotient of each $u_i$ is bounded above by $-\frac{C}{\delta^2}$, hence we have $N$ negative eigenvalues with arbitrary large absolute value. By choosing $NL=\frac{1}{\delta^{n-1}}$ we have also that $|D_{\delta,NL}|=O(1)$ as $\delta\rightarrow 0^+$. This concludes the proof.
\end{proof}

\begin{rem}
A consequence of Theorem \ref{BIG_negative} is that we can always locally perturb a fixed domain of the hyperbolic space $\mathbb H^n$ in order to obtain an arbitrary large number of negative eigenvalues with arbitrary large absolute value. This is done exactly as for the Neumann Laplacian, in which case we can deform locally a domain in order to have an arbitrary large number of eigenvalues arbitrarily close to zero. Indeed, it is sufficient to join to the domain a sufficient number of small balls by sufficiently thin junctions.
\end{rem}

\begin{rem}
A natural problem is to find classes of domains of the hyperbolic space which have exactly $n$ negative eigenvalues. A first immediate conjecture is that hyperbolic balls admit exactly $n$ negative eigenvalues. However, a simple proof of this fact is currently unavailable and we leave this as an open question. Nevertheless, it is possible to prove the result in the case of sufficiently small balls. In fact, from Remark \ref{rem_rough} we deduce that if $\eta_{n+1}>1$, where $\eta_{n+1}$ is the $n+1$-th eigenvalue of the rough Laplacian on $\Omega$, then $\Omega$ admits exactly $n$ negative eigenvalues. Using normal coordinates with origin in the center of a ball $B_{\varepsilon}$ of radius $\varepsilon$ it is possible to prove by means of explicit computations and max-min formula for the eigenvalues that $\eta_{n+1}\rightarrow+\infty$ as $\varepsilon\rightarrow 0^+$. Thus we deduce the existence of $\varepsilon_0>0$ such that every ball of radius smaller than $\varepsilon_0$ admits exactly $n$ negative eigenvalues.
\end{rem}

{\bf Open problem.} Prove that hyperbolic balls admit exactly $n$ negative eigenvalues. 

%%%%%%%%%%%%%%%%%%%%%%%%%%%%%%%%%%%%%%%%%%%%%%%%%%%%%%%%%%%%%%%%%%%%%%%%%%%%%%%%%%%%%%%%%%%%%%%%%%%%%%%%%%%%%%%%%%%%%%%%%%%%%%%%%%%%%%%%
%%%%%%%%%%%%%%%%%%%%%%%%%%%%%%%%%%%%%%%%%%%%%%%%%%%%%%%%%%%%%%%%%%%%%%%%%%%%%%%%%%%%%%%%%%%%%%%%%%%%%%%%%%%%%%%%%%%%%%%%%%%%%%%%%%%%%%%%
%%%%%%%%%%%%%%%%%%%%%%%%%%%%%%%%%%%%%%%%%%%%%%%%%%%%%%%%%%%%%%%%%%%%%%%%%%%%%%%%%%%%%%%%%%%%%%%%%%%%%%%%%%%%%%%%%%%%%%%%%%%%%%%%%%%%%%%%

%%%%%%%%%%%%%%%%%%%%%%%%%%%%%%%%%%%%%%%%%%%%                UPPER BOUNDS                   %%%%%%%%%%%%%%%%%%%%%%%%%%%%%%%%%%%%%%%%%%%%% 

%%%%%%%%%%%%%%%%%%%%%%%%%%%%%%%%%%%%%%%%%%%%%%%%%%%%%%%%%%%%%%%%%%%%%%%%%%%%%%%%%%%%%%%%%%%%%%%%%%%%%%%%%%%%%%%%%%%%%%%%%%%%%%%%%%%%%%%%
%%%%%%%%%%%%%%%%%%%%%%%%%%%%%%%%%%%%%%%%%%%%%%%%%%%%%%%%%%%%%%%%%%%%%%%%%%%%%%%%%%%%%%%%%%%%%%%%%%%%%%%%%%%%%%%%%%%%%%%%%%%%%%%%%%%%%%%%
%%%%%%%%%%%%%%%%%%%%%%%%%%%%%%%%%%%%%%%%%%%%%%%%%%%%%%%%%%%%%%%%%%%%%%%%%%%%%%%%%%%%%%%%%%%%%%%%%%%%%%%%%%%%%%%%%%%%%%%%%%%%%%%%%%%%%%%%

\section{Upper estimates for eigenvalues}\label{bounds}

In this section we provide upper bounds for the eigenvalues $\mu_j$ of \eqref{variational_0} which are compatible with Weyl's law \eqref{weyl}. As we have highlighted in Section \ref{afew}, there is in general no monotonicity between the eigenvalues $\mu_j$ and the squares of the Neumann eigenvalues of the Laplacian $m_j$, hence in general there is no hope to recover upper bounds for $\mu_j$ from the known upper bounds on $m_j$.

We remark that the situation is very different if we think of Dirichlet eigenvalues of the Laplacian and the biharmonic operator. Indeed, if we denote by $\lambda_j$ and by $\Lambda_j$ the eigenvalues of the Laplace and biharmonic operator respectively, with Dirichlet boundary conditions on a domain $\Omega$ of a complete $n$-dimensional smooth Riemannian manifold, then
$$
\lambda_j^2\leq\Lambda_j,
$$
for all $j\in\mathbb N$. This is an immediate consequence of the min-max principle for $\lambda_j$ and $\Lambda_j$. Hence, lower bounds for $\Lambda_j$ can be obtained from lower bounds on $\lambda_j$.

%%%%%%%%%%%%%%%%%%%%%%%%%%%%%%%%%%%%%%%%%%%%%%%%%%%%%%%%%%%%%%%%%%%%%%%%%%%%%%%%%%%%%%%%%%%%%%%%%%%%%%%%%%%%%%%%%%%%%%%%%%%%%%%%%%%%%%%%

%%%%%%%%%%%%%%%%%%%%%%%%%%%%%%%%%%%%%%%%%%%%             METRIC APPROACH                 %%%%%%%%%%%%%%%%%%%%%%%%%%%%%%%%%%%%%%%%%%%%% 

%%%%%%%%%%%%%%%%%%%%%%%%%%%%%%%%%%%%%%%%%%%%%%%%%%%%%%%%%%%%%%%%%%%%%%%%%%%%%%%%%%%%%%%%%%%%%%%%%%%%%%%%%%%%%%%%%%%%%%%%%%%%%%%%%%%%%%%%

\subsection{Decomposition of a metric measure space by capacitors}\label{sub_metric}

In this subsection we present the main technical tools which will be used to prove upper bounds for eigenvalues. We start with some definitions. 

We denote by $(X,{\rm dist},\varsigma)$ a metric measure space with a metric ${\rm dist}$ and a Borel measure $\varsigma$. We will call {\it capacitor} every couple $(A,D)$ of Borel sets of $X$ such that $A\subset D$. By an annulus in $X$ we mean any set $A\subset X$ of the form
\begin{equation*}
A=A(a,r,R)=\left\{x\in X:r<{\rm dist}(x,a)<R\right\},
\end{equation*}
where $a\in X$ and $0\leq r<R<+\infty$. By $2A$ we denote 
\begin{equation*}
2A=2A(a,r,R)=\left\{x\in X:\frac{r}{2}<{\rm dist}(x,a)<2R\right\}.
\end{equation*}
Moreover, for any $F\subset X$ and $r>0$ we denote the $r$-neighborhood of $F$ by $F^r$, namely
$$
F^r:=\left\{x\in X:{\rm dist}(x,F)<r\right\}.
$$

We recall the following metric construction of disjoint capacitors from \cite{gny}.

\begin{thm}[{\cite[Theorem 1.1]{gny}}]\label{gny}
Let $(X,{\rm dist},\varsigma)$ be a metric-measure space with $\varsigma$ a non-atomic finite  Borel measure. Assume that the following properties are satisfied:
\begin{enumerate}[i)]
\item there exists a constant $\Gamma$ such that any metric ball of radius $r$ can be covered by at most $\Gamma$ balls of radius $\frac{r}{2}$;
\item all metric balls in $X$ are precompact sets.
\end{enumerate}
Then for any integer $j$ there exists a sequence $\left\{A_i\right\}_{i=1}^j$ of $j$ annuli in $X$ such that, for any $i=1,...,j$
\begin{equation*}
\varsigma(A_i)\geq c\frac{\varsigma(X)}{j},
\end{equation*}
and the annuli $2A_i$ are pairwise disjoint. The constant $c$ depends only on the constant $\Gamma$ in i).
\end{thm}

As we shall see, Theorem \ref{gny} is not easy to use in many concrete cases, e.g., when $X$ is a domain in the standard hyperbolic space $\mathbb H^n$ and $\varsigma$ is the restriction of the Lebesgue measure on $X$. In fact, hypothesis $i)$ fails to hold with $\Gamma$ depending only on the dimension because of the exponential growth of the volume of balls. We state now the following lemma, which improves \cite[Lemma 4.1]{colboisgirouard}. We postpone its proof at the end of this subsection.

\begin{lem}\label{genCM}
Let $(X,{\rm dist},\varsigma)$ be a compact metric measure space with a finite measure $\varsigma$. Assume that for all $s>0$ there exists an integer $N(s)$ such that each ball of radius $5s$ can be covered by $N(s)$ balls of radius $s$. Let $\beta>0$ satisfying $\beta\leq \frac{\varsigma(X)}{2}$ and let $r>0$ be such that for all $x\in X$
$$
\varsigma(B(x,r))\leq\frac{\beta}{2N(r)}.
$$
Then there exist two open sets $A$ and $D$ of $X$ with $A\subset D$ such that:
\begin{enumerate}[i)]
\item $A=B(x_1,r)\cup\cdots\cup B(x_l,r)$ with ${\rm dist}(x_i,x_j)\geq 4r$ if $i\ne j$;
\item $D=A^{4r}=B(x_1,5r)\cup\cdots\cup B(x_l,5r)$;
\item $\varsigma(A)\geq\frac{\beta}{2N(r)}$, $\varsigma(D)\leq\beta$ and ${\rm dist}(A,D^c)\geq 4r$.
\end{enumerate}
\end{lem}

A consequence of Lemma \ref{genCM} is the following result providing a decomposition of a metric measure space by capacitors which is alternative of that of Theorem \ref{gny}.

\begin{lem}\label{genCM-cor}
Let $(X,{\rm dist},\varsigma)$ be a compact metric measure space with a finite measure $\varsigma$. Assume that for all $s>0$ there exists an integer $N(s)$ such that each ball of radius $5s$ can be covered by $N(s)$ balls of radius $s$. If there exists an integer $k>0$ and a real number $r>0$ such that, for each $x\in X$
$$
\varsigma(B(x,r))\leq\frac{\varsigma(X)}{4N(r)^2k},
$$
then there exist $k$ $\varsigma$-measurable subsets $A_1,...,A_k$ of $X$ such that
$$
\varsigma(A_i)\geq\frac{\varsigma(X)}{2N(r)k},
$$
for all $i\leq k$, ${\rm dist}(A_i,A_j)\geq 4r$ for $i\ne j$, and 
$$
A_i=B(x^i_1,r)\cup\cdots\cup B(x^i_{l_i},r).
$$
\end{lem}

The proof of Lemma \ref{genCM-cor} is a consequence of Lemma \ref{genCM} and follows exactly the same lines of the proof of \cite[Lemma 2.1]{colboisgirouard}. We remark that \cite[Lemma 2.1]{colboisgirouard} provides a decomposition of a metric measure space by capacitors given by union of balls. In Lemma \ref{genCM-cor} the decomposition is given by unions of {\it disjoint} balls.

A clever merging of Theorem \ref{gny} and Lemma \ref{genCM-cor} allows to obtain the following Theorem, which provides a further construction of disjoint families of capacitors. This is a construction which we will widely use in the next subsections. Its proof follows exactly the same lines as those of \cite[Theorem 2.1]{asma_conf}. In fact, the substantial difference is the use of Lemma \ref{genCM-cor} instead of \cite[Lemma 2.3]{asma_conf} (see also \cite[Lemma 2.1]{colboisgirouard} and \cite[Corollary 2.3]{colbois_maerten}).

\begin{thm}\label{corollary_small_annuli}
Let $(X,{\rm dist},\varsigma)$ be a compact metric-measure space with $\varsigma$ a non-atomic finite  Borel measure and let $a>0$. Assume that there exists a constant $\Gamma$ such that any metric ball of radius $0<r\leq a$ can be covered by at most $\Gamma$ balls of radius $\frac{r}{2}$. Then, for every $j\in\mathbb N$ there exists two families $\left\{A_i\right\}_{i=1}^j$ and  $\left\{D_i\right\}_{i=1}^j$ of Borel subsets of $X$ such that $A_i\subset D_i$, with the following properties:
\begin{enumerate}[i)]
\item $\varsigma(A_i)\geq c\frac{\varsigma(X)}{j}$, where $c$ depends only on $\Gamma$;
\item $D_i$ are pairwise disjoint;
\item the two families have one of the following form:
\begin{enumerate}[a)]
\item all the $A_i$ are annuli and $D_i=2A_i$, with outer radii smaller than $a$, or
\item all the $A_i$ are of the form $A_i=B(x_1^i,r_0)\cup\cdots\cup B(x_{l_i}^i,r_0)$, $G_i=A_i^{4r_0}$ and ${\rm dist}(x_k^i,x_l^i)\geq 4r_0$, where $r_0=\frac{4a}{1600}$.
\end{enumerate}
\end{enumerate}
\end{thm}

We remark that for a sufficiently large integer $j$ it is always possible to apply the construction of Theorem \ref{gny} and obtain a decomposition of the metric measure space by annuli (Theorem \ref{corollary_small_annuli} $i)$,$ii)$ and $iii)$-$a)$). In particular we have the following.

\begin{lem}\label{large_n}
Assume that the hypothesis of Theorem \ref{corollary_small_annuli} hold. Then there exists an integer $j_X$ such that for every $j\geq j_X$ there exists two families $\left\{A_i\right\}_{i=1}^j$ and  $\left\{D_i\right\}_{i=1}^j$ of Borel subsets of $X$ such that $A_i\subset D_i$ satisfying $i)$,$ii)$ and $iii)$-$a)$ of Theorem \ref{corollary_small_annuli}.
\end{lem}
We refer to \cite[Proposition 2.1]{asma_conf} for the proof of Lemma \ref{large_n}.

We state now a useful corollary of Theorem \ref{gny} which gives a lower bound of the inner radius of the annuli of the decomposition, see \cite[Remark\,3.13]{gny}.
\begin{cor}\label{corollarygny0}
Let the assumptions of Theorem \ref{gny} hold. Then each annulus $A_i$ has either internal radius $r_i$ such that
\begin{equation}\label{gny-rad-est}
r_i\geq\frac{1}{2}\inf\left\{r\in\mathbb R:V(r)\geq v_j\right\},
\end{equation}
where $V(r):=\sup_{x\in X}\varsigma(B(x,r))$ and $v_j=c\frac{\varsigma(X)}{j}$ , or is a ball of radius $r_i$ satisfying \eqref{gny-rad-est}.
\end{cor}
It turns out that Corollary \ref{corollarygny0} applies to the case $iii)$-$a)$ of Theorem \ref{corollary_small_annuli}, see also \cite{asma_conf}.

We conclude this subsection with the proof of Lemma \ref{genCM}

\begin{proof}[Proof of Lemma \ref{genCM}] 
\textbf{Step 1.} We construct the points $x_i$ by induction. Let $\Omega_1=X$. The point $x_1$ is such that $\varsigma(B(x_1,r)\cap \Omega_1)=\max\{\varsigma(B(x,r)\cap \Omega_1):x\in X\}$.

Let $\Omega_2=X\setminus B(x_1,5r)$. The point $x_2$ is such that 
$$
\varsigma(B(x_2,r)\cap \Omega_1)=\max\{\varsigma(B(x,r)\cap \Omega_2):x\in X\}.$$
Note that this definition implies that ${\rm dist}(x_2,x_1)\ge 4r$. If ${\rm dist}(x_2,x_1)<4r$, $B(x,r)\cap \Omega_2 = \emptyset$.

Suppose that we have constructed $x_1,...,x_j$. Let $\Omega_{j+1}=X\setminus (B(x_1,5r)\cup...\cup B(x_j,5r) $. Suppose $\Omega_{j+1}\not = \emptyset$. The point $x_{j+1}$ is such that $\varsigma(B(x_{j+1},r)\cap \Omega_{j+1})=\max\{\varsigma(B(x,r)\cap \Omega_{j+1}):x\in X\}$. Note that this definition implies that ${\rm dist}(x_{j+1},x_i)\ge 4r$. If ${\rm dist}(x,x_i)<4r$, $B(x,r)\cap \Omega_{j+1} = \emptyset$.

By compactness, the process has to stop: there exist only finitely many points on $X$ such that ${\rm dist}(x_i,x_j)\ge 4r$. Let $(x_1,...,x_k)$ the set of points we have constructed. We have $\varsigma(X\setminus (B(x_1,5r)\cup...\cup B(x_k,5r)))=0$ otherwise we could do another iteration. Then
$$
\varsigma(X)=\varsigma(B(x_1,5r)\cup...\cup B(x_k,5r)).
$$

\noindent
\textbf{Step 2.} We write
$$
B(x_1,5r)\cup...\cup B(x_k,5r)=(B(x_1,5r) \cap \Omega_1)\cup (B(x_2,5r)\cap \Omega_2)\cup...\cup (B(x_k,5r)\cap \Omega_k).
$$
Indeed, if $x\in B(x_i,5r)$ and $x\not \in B(x_i,5r)\cap \Omega_i$, we have by construction that $x\in B(x_1,5r)\cup...\cup B(x_{i-1},5r))$. Let $j$ be the smallest integer such that $x\in B(x_j,5r)$. Then $x\in \Omega_j$. In fact, if it not the case,  $x\in B(x_1,5r)\cup...\cup B(x_{j-1},5r))$ and this would contradict the fact that $j$ was minimal.

As this is a disjoint union, we get
$$
\varsigma(X)=\varsigma(B(x_1,5r)\cup...\cup B(x_k,5r))=\varsigma(B(x_1,5r) \cap \Omega_1)+...+\varsigma(B(x_k,5r) \cap \Omega_k).
$$

\noindent
\textbf{Step 3.}
For each $i$, let us show that
$$
\varsigma(B(x_i,5r) \cap \Omega_i)\le N(r) \varsigma(B(x_i,r)\cap \Omega_i).
$$
In fact by definition, 
$$
\varsigma(B(x_{i},r)\cap \Omega_{i})=\max\{\varsigma(B(x,r)\cap \Omega_{i}):x\in X\}
$$
and $B(x_i,5r)$ is covered by $N(r)$ balls of radius $r$. Let $B(y,r)$ one of these balls. We have 
$$
\varsigma(B(y,r)\cap \Omega_{i}) \le \varsigma(B(x_{i},r)\cap \Omega_{i}),
$$
so that $\varsigma(B(x_i,5r) \cap \Omega_i)\le N(r) \varsigma(B(x_i,r)\cap \Omega_i)$.

\noindent
\textbf{Step 4.}
We deduce
\begin{multline}
\varsigma(X) =\varsigma(B(x_1,5r) \cap \Omega_1)+...+\varsigma(B(x_k,5r) \cap \Omega_k) \\
\le N(r)(\varsigma(B(x_1,r) \cap \Omega_1)+...+\varsigma(B(x_k,r) \cap \Omega_k)\\
\le N(r)(\varsigma(B(x_1,r)+...+\varsigma(B(x_k,r)),
\end{multline}
and
$$
\varsigma(B(x_1,r))+...+\varsigma(B(x_k,r))\ge \frac{\varsigma(X)}{N(r)}.
$$
By hypothesis $\beta \le \frac{\varsigma(X)}{2}$. As $\varsigma(B(x_1,r)+...+\varsigma(B(x_k,r))\ge \frac{\varsigma(X)}{N(r)} \ge \frac{2\beta}{N(r)}$, we choose $l$ such that 
$$
\varsigma(B(x_1,r))+...+\varsigma(B(x_l,r))\ge \frac{\beta}{2N(r)}
$$
and
$$
\varsigma(B(x_1,r))+...+\varsigma(B(x_{l-1},r))\le \frac{\beta}{2N(r)}.
$$
Let $A=B(x_1,r)\cup...\cup B(x_l,r)$ and $D=B(x_1,5r)\cup...\cup B(x_l,5r)$. We have by construction ${\rm dist}(x_i,x_j)\ge 4r$ and we immediately deduce that $\varsigma(A)\ge \frac{\beta}{2N(r)}$, ${\rm dist}(A,D^r)\ge 4r$. It remains to show
$$
\varsigma(D)\le \beta.
$$
We can argue as before and deduce that
$$
 \varsigma(D)=\varsigma(B(x_1,5r))+...+\varsigma(B(x_l,5r) =\varsigma(B(x_1,5r) \cap \Omega_1)+...+\varsigma(B(x_l,5r) \cap \Omega_l)\le
$$
$$
\le N(r)(\varsigma(B(x_1,r)+...+\varsigma(B(x_l,r))) \le \frac{\beta}{2}.
$$
This concludes the proof.
\end{proof}

%%%%%%%%%%%%%%%%%%%%%%%%%%%%%%%%%%%%%%%%%%%%%%%%%%%%%%%%%%%%%%%%%%%%%%%%%%%%%%%%%%%%%%%%%%%%%%%%%%%%%%%%%%%%%%%%%%%%%%%%%%%%%%%%%%%%%%%%

%%%%%%%%%%%%%%%%%%%%%%%%%%%%%%%%%%%%%%%%%%%%                A FIRST ESTIMATE               %%%%%%%%%%%%%%%%%%%%%%%%%%%%%%%%%%%%%%%%%%%%% 

%%%%%%%%%%%%%%%%%%%%%%%%%%%%%%%%%%%%%%%%%%%%%%%%%%%%%%%%%%%%%%%%%%%%%%%%%%%%%%%%%%%%%%%%%%%%%%%%%%%%%%%%%%%%%%%%%%%%%%%%%%%%%%%%%%%%%%%%

\subsection{A first general estimate}\label{sub_main}

In this subsection we prove an upper bound which holds for any domain with $C^1$ boundary of a complete $n$-dimensional smooth Riemannian manifold $(M,g)$ with a given lower bound on the Ricci curvature of the form ${\rm Ric}\geq-(n-1)\kappa^2$, $\kappa\geq 0$. 

We need a few preliminary definitions. We denote by $r_{inj}(p)$ the injectivity radius of the manifold $(M,g)$ at $p$. We denote by $r_{inj}$ the injectivity radius of the  manifold $(M,g)$, which is defined as the infimum of $r_{inj}(p)$ for $p\in M$. We will use also the injectivity radius relative to $\Omega\subset M$, defined by:
\begin{equation}\label{injO}
r_{inj,\Omega}:=\inf_{p\in\overline\Omega}r_{inj}(p).
\end{equation}
If $\Omega$ is bounded, then the infimum in \eqref{injO} is actually a minimum and it is strictly positive.

We are ready to state the main result of this subsection.

\begin{thm}\label{main_general}
Let $(M,g)$ be a complete $n$-dimensional smooth Riemannian manifold with ${\rm Ric}\geq-(n-1)\kappa^2$, $\kappa\geq 0$ and let $\Omega$ be a bounded domain of $M$ of class $C^1$. Let $a:=\min\left\{\frac{1}{\kappa},\frac{r_{inj,\Omega}}{2}\right\}$. Then 
\begin{equation}\label{main_estimate_general}
\mu_j\leq A_n\left(\frac{j}{|\Omega|}\right)^{\frac{4}{n}}+B_n\frac{|\partial\Omega|^4}{|\Omega|^4}+\frac{C_n}{a^4}.
\end{equation}
for all $j\in\mathbb N$, where $A_n,B_n,C_n$ are positive constants which depend only on the dimension.
\end{thm}

The strategy of the proof of Theorem \ref{main_general} is to build, for each $j\in\mathbb N$, $j$ disjointly supported functions $u_1,...,u_j\in H^2(\Omega)$. Hence, the linear space $U_j$ spanned by $u_1,...,u_j$ is $j$-dimensional and we can use $U_j$ in the min-max formula \eqref{minmax}. The fact that the functions $u_1,...,u_j$ have disjoint support makes easy to estimate the Rayleigh quotient of any function in $U_j$: it is in fact sufficient to estimate the Rayleigh quotient of each of the $u_i$.

Suitable test functions for the Rayleigh quotient in \eqref{minmax} are built, in this subsection, in terms of the Riemannian distance function from a point $p\in M$. For any $x,p\in M$ we denote by $\delta_p(x)$ the function
$$
\delta_p(x):={\rm dist}(x,p).
$$ 
 We also denote by ${\rm cut}(p)$ the cut-locus of a point $p\in M$. We will make use of the Laplacian Comparison Theorem, see e.g., \cite[\S\,9]{petersen} for details.
\begin{thm}\label{lap_comp_0}
Let $(M,g)$ be a complete $n$-dimensional smooth Riemannian manifold satisfying ${\rm Ric}\geq-(n-1)\kappa^2$, $\kappa\geq 0$ and let $p\in M$. Then, for any $x\in M\setminus(\left\{p\right\}\cup{\rm cut}(p))$ 
\begin{enumerate}[i)]
\item $\Delta\delta_p(x)\leq (n-1)\kappa\coth(\kappa\delta_p(x))$ if $\kappa>0$;
\item $\Delta\delta_p(x)\leq \frac{n-1}{\delta_p(x)}$ if $\kappa=0$.
\end{enumerate}
\end{thm}

We prove now the following lemma.

\begin{lem}\label{lem_lap_small_ball}
Let $(M,g)$ be a complete $n$-dimensional smooth Riemannian manifold  with ${\rm Ric}\geq-(n-1)\kappa^2$, $\kappa\geq 0$. Then, for any $p\in M$ and any $x\in B\left(p,\frac{r_{inj}(p)}{2}\right)$ we have
\begin{enumerate}[i)]
\item $|\Delta\delta_p(x)|\leq (n-1)\kappa\coth(\kappa\delta_p(x))$ if $\kappa>0$;
\item $|\Delta\delta_p(x)|\leq \frac{n-1}{\delta_p(x)}$ if $\kappa=0$.
\end{enumerate}
In particular, for any $\kappa\geq 0$
\begin{equation}\label{ineq_gen}
|\Delta\delta_p(x)|\leq \frac{n-1}{\delta_p(x)}+(n-1)\kappa.
\end{equation}
\end{lem}

\begin{proof}
We prove point $ii)$.  Let $p\in M$ and let $x\in B\left(p,\frac{r_{inj}(p)}{2}\right)$. Let $p'$ be the unique point such that $\delta_p(p')=\frac{r_{inj}(p)}{2}$ and $x$ belongs to the geodesic joining $p$ and $p'$. From Theorem \ref{lap_comp_0} it follows that
$$
\Delta\delta_p(x)\leq \frac{n-1}{\delta_p(x)}.
$$
Moreover, since $x$ belongs to the geodesic connecting $p$ with $p'$ we see that 
\begin{multline*}
\Delta\delta_p(x)=\Delta\left(\frac{r_{inj}(p)}{2}-\delta_{p'}(x)\right)=-\Delta\delta_{p'}(x)\geq-\frac{n-1}{\delta_{p'}(x)}\\
=-\frac{n-1}{\frac{r_{inj}(p)}{2}-\delta_{p}(x)}\geq -\frac{n-1}{\delta_{p}(x)}.
\end{multline*}
Point $ii)$ is now proved. Point $i)$ is proved exactly in the same way. The last statement follows by observing that $\cosh(t)\leq\frac{1}{t}-1$ for all $t>0$, which implies \eqref{ineq_gen}.
\end{proof}
We note that the restriction $x\in B\left(p,\frac{r_{inj}(p)}{2}\right)$ is somehow natural. We may think of the unit $2$-dimensional standard sphere and coordinates $(\theta,\phi)\in[0,\pi]\times[0,2\pi)$ and $p$ being the north pole ($\theta=0$). In this case $r_{inj}=\pi$. Hence $\delta_p(\theta,\phi)=\theta$ and $\Delta\delta_p(\theta,\phi)=\frac{1}{\tan(\theta)}$. Then, for any $(\theta,\phi)\in\left[0,\frac{\pi}{2}\right]\times[0,2\pi)$, $0\leq\Delta\delta_p(\theta,\phi)\leq\frac{1}{\delta_p(\theta,\phi)}$, which in particular implies $|\Delta\delta_p(\theta,\phi)|\leq\frac{1}{\delta_p(\theta,\phi)}$. This last inequality is no longer true for any $(\theta,\phi)\in\left(\frac{\pi}{2},\pi\right]\times[0,2\pi)$. Actually, it stills remains true for $\theta\in\left(\frac{\pi}{2},\theta^*\right]$, for some $\theta^*\in\left(\frac{\pi}{2},\pi\right)$ ($\theta^*\approx 2.02876$), but fails for $\theta\in(\theta^*,\pi]$.

We are now ready to prove Theorem \ref{main_general}.

\begin{proof}[Proof of Theorem \ref{main_general}]
We first apply Theorem \ref{corollary_small_annuli} with $a=\min\left\{\frac{1}{\kappa},\frac{r_{inj,\Omega}}{2}\right\}$ (if $\kappa=0$, then we take $a=\frac{r_{inj,\Omega}}{2}$). We take $X=\Omega$ endowed with the induced Riemannian distance, and with the measure $\varsigma$ defined as the restriction to $\Omega$ of the Lebesgue measure of $M$, namely $\varsigma(E)=|E\cap\Omega|$ for all measurable set $E$.

\noindent{\bf Step 1 (large $j$).} From Lemma \ref{large_n} we deduce that there exists $j_{\Omega}\in\mathbb N$ such that for all $j\geq j_{\Omega}$ there exists a sequence $\left\{A_i\right\}_{i=1}^{4j}$ of $4j$ annuli such that $2A_i$ are pairwise disjoint and 
\begin{equation}\label{ineq_gen_1}
|\Omega\cap A_i|\geq c\frac{|\Omega|}{4j}.
\end{equation}
The constant $c$ depends only on $\Gamma$ of Theorem \ref{corollary_small_annuli}, hence it depends only on the dimension and can be determined explicitly (see \cite{asma_conf}). Since we have $4j$ annuli, we can pick at least $2j$ of them such that
\begin{equation}\label{ineq_gen_2}
|\Omega\cap 2A_i|\leq \frac{|\Omega|}{j}.
\end{equation}
Among these last $2j$ annuli, we can pick at least $j$ of them such that
\begin{equation}\label{ineq_gen_3}
|\partial\Omega\cap 2A_i|\leq \frac{|\partial\Omega|}{j}.
\end{equation}
We take this family of $j$ annuli, and denote it by $\left\{A_i\right\}_{i=1}^j$.

%From Theorem \ref{corollary_small_annuli} we also know that the outer radius of each annulus $2A_i$ is less or equal to $a$. The index $j_{\Omega}$ is given by
%\begin{equation}\label{jO}
%j_{\Omega}=\frac{b_n D^{2n}}{|\Omega|a^n}\left(\frac{\sinh(\kappa D)}{\kappa D}\right)^{2(n-1)},
%\end{equation}
%where $b_n$ depends only on $n$ through the constants $c$ of Theorem \ref{corollary_small_annuli} and $a_n$ of Lemma \ref{covering_lem}. This follows immediately from \eqref{covering_lem} with $\alpha=\frac{a}{1600D}$. For $\kappa=0$ the quantity $\frac{\sinh(\kappa D)}{\kappa D}$ is set to be one. Moreover, if $\frac{a}{1600}>D$, then $j_{\Omega}=1$ (see also Remark \ref{rem_index}).

Subordinated to this decomposition  we construct a family of $j$ disjointly supported functions $u_1,...,u_j$. If $u_1,...,u_j$ belong to $H^2(\Omega)$, then from \eqref{minmax}
\begin{equation}\label{minmax2}
\mu_j\leq\max_{i=1,...,j}\frac{\int_{\Omega}|D^2u_i|^2+{\rm Ric}(\nabla u_i,\nabla u_i)dv}{\int_{\Omega}u_i^2dv},
\end{equation}
Thus, in order to estimate $\mu_j$ it is sufficient to estimate the Rayleigh quotient of each of the test functions. 

Let  $f:[0,\infty)\rightarrow[0,1]$ be defined as follows:
\begin{equation}\label{f}
f(t)=
\begin{cases}
4(t-1)^2(4t-1), & t\in\left[\frac{1}{2},1\right],\\
1, & t\in\left[0,\frac{1}{2}\right],\\
0, & t\in[1,+\infty[.
\end{cases}
\end{equation}

By construction $f\in C^{1,1}[0,+\infty)$. Moreover $f\in C^2(\left[\frac{1}{2},1\right])$. We consider test functions of the form $f(\eta\delta_p(x))$ for some $\eta\in\mathbb R$ and $p\in M$. We note that
\begin{equation}\label{nabla_f}
\nabla f(\eta\delta_p(x))=\eta f'(\eta\delta_p(x))\nabla\delta_p(x)
\end{equation}
and
\begin{equation}\label{delta_f}
\Delta f(\eta\delta_p(x))=\eta^2 f''(\eta\delta_p(x))+\eta f'(\eta\delta_p(x))\Delta\delta_p(x).
\end{equation}
In \eqref{delta_f} we have used the fact that $|\nabla\delta_p(x)|=1$ for almost all $x\in M$, the equality holding pointwise in $M\setminus(\left\{p\right\}\cup{\rm cut}(p))$). Standard computations show that
\begin{equation}\label{f1}
|f'(t)|\leq 3
\end{equation}
and
\begin{equation}\label{f2}
|f''(t)|\leq 24.
\end{equation}
Let now $A_i$ be an annulus of the family $\left\{A_i\right\}_{i=1}^j$. We have two possibilities. Either $A_i$ is a proper annulus with  $0<r_i<R_i\leq\frac{r_{inj,\Omega}}{4}$, or is a ball of radius $0<r_i\leq \frac{r_{inj,\Omega}}{4}$.
\begin{enumerate}[{\bf Case} \bf a]
\item {\bf (ball).} Assume that $A_i$ is a ball of radius $0<r_i\leq\frac{r_{inj,\Omega}}{4}$ and center $p_i$. Associated to $A_i$ we define a function $u_i$ as follows
\begin{equation}\label{ball}
u_i(x)=
\begin{cases}
1, & 0\leq\delta_{p_i}(x)\leq\frac{r_i}{2}\\
f(\frac{\delta_{p_i}(x)}{2r_i}), & \frac{r_i}{2}\leq\delta_{p_i}(x)\leq r_i\\
0, & {\rm otherwise}. 
\end{cases}
\end{equation}
By construction, ${u_i}_{|_{\Omega}}\in H^2(\Omega)$. Standard computations (see \eqref{nabla_f}-\eqref{f2}) and Lemma \ref{lem_lap_small_ball} show that 
\begin{equation}\label{grad_ball}
|\nabla u_i|\leq \frac{3}{r_i}
\end{equation}
and
\begin{equation}\label{lap_ball}
|\Delta u_i|\leq\frac{6}{r_i^2}+\frac{3}{r_i}|\Delta\delta_{p_i}(x)|\leq\frac{3(n+1)}{r_i^2}+\frac{3(n-1)\kappa}{r_i}.
\end{equation}
When estimating the Rayleigh quotient of $u_i$ we will also need to estimate $|\nabla|\nabla u_i|^2|$. We have that
$$
|\nabla u_i(x)|^2=\frac{f'(\delta_{p_i}(x)/2r_i)^2}{4r_i^2},
$$
hence
\begin{equation}\label{nablanabla2_0}
|\nabla |\nabla u_i(x)|^2|=\frac{|f'(\delta_{p_i}(x)/2r_i)f''(\delta_{p_i}(x)/2r_i)|}{4r_i^3}\leq\frac{36}{r_i^3}.
\end{equation}
\item {\bf (annulus).} Assume that $A_i$ is a proper annulus  of radii $0<r_i<R_i\leq\frac{r_{inj,\Omega}}{4}$ and center $p_i$. Associated to $A_i$ we define a function $u_i$ as follows
\begin{equation}\label{prop_annulus}
u_i(x)=
\begin{cases}
1-f(\frac{\delta_{p_i}(x)}{r_i}), & \frac{r_i}{2}\leq\delta_{p_i}(x)\leq r_i\\
1 , & r_i\leq\delta_{p_i}(x)\leq R_i\\
f(\frac{\delta_{p_i}(x)}{2R_i}), & R_i\leq\delta_{p_i}(x)\leq 2R_i\\
0, & {\rm otherwise}.
\end{cases}
\end{equation}
By construction, ${u_i}_{|_{\Omega}}\in H^2(\Omega)$. Standard computations (see \eqref{nabla_f}-\eqref{f2}) show that
\begin{equation}
|\nabla u_i(x)|\leq 
\begin{cases}
\frac{3}{R_i}, & R_i\leq\delta_{p_i}(x)\leq 2R_i,\\
\frac{6}{r_i}, & \frac{r_i}{2}\leq\delta_{p_i}(x)\leq r_i,\\
0, & {\rm otherwise}.
\end{cases}
\end{equation}
In any case then
\begin{equation}\label{grad_ann}
|\nabla u_i(x)|\leq \frac{6}{r_i}.
\end{equation}
Moreover, from Lemma \ref{lem_lap_small_ball} we have
\begin{equation}\label{grad_ann_0}
|\Delta u_i(x)|\leq
\begin{cases}
\frac{3(n+1)}{R_i^2}+\frac{3(n-1)\kappa}{R_i}, & R_i\leq\delta_{p_i}(x)\leq 2R_i,\\
\frac{12(n+1)}{r_i^2}+\frac{6(n-1)\kappa}{r_i}, & \frac{r_i}{2}\leq\delta_{p_i}(x)\leq r_i,\\
0, & {\rm otherwise}.
\end{cases}
\end{equation}
In any case then
\begin{equation}\label{lap_ann}
|\Delta u_i(x)|\leq \frac{12(n+1)}{r_i^2}+\frac{6(n-1)\kappa}{r_i}.
\end{equation}

We will also need an estimate on $|\nabla|\nabla u_i|^2|$. As for \eqref{nablanabla2_0} we find that
\begin{equation}\label{nablanabla2_0_1}
|\nabla |\nabla u_i(x)|^2|\leq
\begin{cases}
\leq\frac{36}{R_i^3}, & R_i\leq\delta_{p_i}(x)\leq 2R_i,\\
\leq\frac{288}{r_i^3}, & \frac{r_i}{2}\leq\delta_{p_i}(x)\leq r_i,\\
0, & {\rm otherwise}.
\end{cases}
\end{equation}
In any case then
\begin{equation}\label{nablanabla2}
|\nabla |\nabla u_i(x)|^2|\leq \frac{288}{r_i^3}.
\end{equation}
\end{enumerate}

We  also need an upper bound for the volume of $2A_i$. Since the outer radius of $2A_i$ is by construction smaller than $\frac{1}{\kappa}$, we have, from the volume comparison and standard calculus that
\begin{equation}\label{hyp_vol}
|2A_i|\leq 2^n\sinh(1)^{n-1}\omega_n R_i^n.
\end{equation}
From Bochner's formula \eqref{bochner_f} we deduce that
\begin{multline}\label{bochner_main}
\int_{\Omega\cap 2A_i}|D^2u_i|^2+{\rm Ric}(\nabla u_i,\nabla u_i)dv=\int_{\Omega\cap 2A_i}\frac{1}{2}\Delta(|\nabla u_i|^2)-\langle\nabla\Delta u_i,\nabla u_i\rangle dv\\
=\int_{\partial\Omega\cap 2A_i}\frac{1}{2}\frac{\partial|\nabla u_i|^2}{\partial\nu}-\Delta u_i\frac{\partial u_i}{\partial\nu}d\sigma+\int_{\Omega\cap 2A_i}(\Delta u_i)^2dv\\
\leq \int_{\partial\Omega\cap 2A_i}\frac{1}{2}|\nabla|\nabla u_i|^2|+|\Delta u_i||\nabla u_i|d\sigma+\int_{\Omega\cap 2A_i}(\Delta u_i)^2dv
\end{multline}
We note that the boundary integrals are taken on $\partial\Omega\cap 2A_i$ since by construction $u_i\in H^2_0(2A_i)$. From \eqref{ineq_gen_3}, \eqref{grad_ball}, \eqref{lap_ball},  \eqref{nablanabla2_0}, \eqref{grad_ann}, \eqref{lap_ann} and \eqref{nablanabla2} we deduce that
\begin{multline}
\int_{\Omega\cap 2A_i}|D^2u_i|^2+{\rm Ric}(\nabla u,\nabla u)dv\\
\leq \int_{\Omega\cap 2A_i}(\Delta u_i)^2dv+\frac{|\partial\Omega|}{j}\left(\frac{18(4n+5)}{r_i^3}+\frac{36(n-1)\kappa}{r_i^2}\right)\\
\leq\int_{\Omega\cap 2A_i}(\Delta u_i)^2dv+\frac{|\partial\Omega|}{j}18\left(\frac{5n+4}{r_i^3}\right).
\end{multline}
where the last inequality follows from the fact that $r_i\leq R_i\leq\frac{1}{2\kappa}$ .

Corollary \ref{corollarygny0} gives us information on the size of the radius $r_i$, in fact 
\begin{equation}
r_i\geq\frac{1}{2}\tilde r:=\frac{1}{2}\inf\mathcal B
\end{equation}
where
\begin{equation}
\mathcal B:=\left\{r\in\mathbb R:V(r)\geq \frac{c|\Omega|}{j}\right\}.
\end{equation}
We observe that each $r\in\mathcal B$ is such that
$$
\frac{c|\Omega|}{j}\leq V(r)=\sup_{x\in\Omega}|B(x,r)\cap\Omega|\leq |B(x,r)|\leq |B(p',r)|_{\kappa}
$$
by volume comparison, where $|B(p',r)|_{\kappa}$ denotes the volume of the ball of radius $r$ in the space form of constant curvature $-\kappa^2$. If $\kappa=0$ then each $r\in\mathcal B$ is such that
$$
\frac{c|\Omega|}{j}\leq \omega_n r^n. 
$$
Hence any $r\in\mathcal B$ is such that
$$
r\geq\left(\frac{c|\Omega|}{\omega_n j}\right)^{\frac{1}{n}},
$$
therefore
\begin{equation}\label{lower_ri}
r_i\geq\frac{1}{2}\left(\frac{c|\Omega|}{\omega_n j}\right)^{\frac{1}{n}}.
\end{equation}
If $\kappa>0$, then $\tilde r\leq 2r_i\leq\frac{2}{\kappa}$ by construction, and since $\tilde r=\inf\mathcal B$, from volume comparison and standard calculus (see also \eqref{hyp_vol})
$$
\frac{c|\Omega|}{j}\leq \sinh(2)^{n-1}\omega_n \tilde r^n.
$$
Therefore
\begin{equation}\label{lower_ri_2}
r_i\geq \frac{\tilde r}{2}\geq\frac{1}{2}\left(\frac{c|\Omega|}{\sinh(2)^{n-1}\omega_n j}\right)^{\frac{1}{n}}.
\end{equation}
We note that \eqref{lower_ri} implies \eqref{lower_ri_2} which holds true for any $\kappa\geq 0$. We conclude that
\begin{equation}\label{boundary_est}
\int_{\Omega\cap 2A_i}|D^2u_i|^2+{\rm Ric}(\nabla u,\nabla u)dv\leq \int_{\Omega\cap 2A_i}(\Delta u_i)^2dv+\alpha_n\frac{|\partial\Omega|}{j}\left(\frac{j}{|\Omega|}\right)^{\frac{3}{n}}
\end{equation}
where
\begin{equation}\label{alpha_n}
\alpha_n=144(5n+4)\left(\frac{\omega_n\sinh(2)^{n-1}}{c}\right)^{\frac{3}{n}}
\end{equation}
In order to complete the estimates, it remains to bound the term $\int_{\Omega\cap 2A_i}(\Delta u_i)^2dv$.  We need to distinguish the case  $n=2,3,4$ and $n>4$.
\begin{enumerate}[{\bf Case} \bf a']
\item {\bf (lower dimensions).} Let $n\leq 4$. We note that in this case it irrelevant to know that $|2A_i\cap\Omega|\leq\frac{|\Omega|}{j}$. This fact is crucial only for higher dimensions. 

If $A_i$ is a ball of radius $r_i\leq\frac{1}{2\kappa}$, and hence $|2A_i|\leq 2^n\sinh(1)^{n-1}\omega_nr_i^n$, we have
\begin{multline}
\int_{\Omega\cap 2A_i}(\Delta u_i)^2dv\leq \int_{2A_i}(\Delta u_i)^2dv\leq\left(\frac{3(n+1)}{r_i^2}+\frac{3(n-1)\kappa}{r_i}\right)^2|2A_i|\\
\leq \left(\frac{3(n+1)}{r_i^2}+\frac{3(n-1)\kappa}{r_i}\right)^22^n\omega_n\sinh(1)^{n-1}r_i^n\\
\leq 18\left(\frac{(n+1)^2}{r_i^{4-n}}+\frac{(n-1)^2\kappa^2}{r_i^{2-n}}\right)2^n\sinh(1)^{n-1}\omega_n\\
\leq 18\left(\frac{(n+1)^2}{r_i^{4-n}}+\frac{(n-1)^2}{4r_i^{4-n}}\right)2^n\sinh(1)^{n-1}\omega_n\\
=\frac{18 (4(n+1)^2+(n-1)^2)2^{n-2}\omega_n\sinh(1)^{n-1}}{r_i^{4-n}},
\end{multline}
where in the last line we have used the fact $2r_i\leq\frac{1}{\kappa}$ if $\kappa>0$. From \eqref{lower_ri_2} we obtain that
\begin{equation}\label{ineq_ball_small_n}
\int_{\Omega\cap 2A_i}(\Delta u_i)^2dv\leq \beta_n'\left(\frac{j}{|\Omega|}\right)^{\frac{4}{n}-1},
\end{equation}
where we set
\begin{equation}\label{beta_n}
\beta_n'=72\omega_n\sinh(1)^{n-1}(4(n+1)^2+(n-1)^2)\left(\frac{\sinh(2)\omega_n}{c}\right)^{\frac{4}{n}-1}.
\end{equation}

In the very same way it is possible to prove that \eqref{ineq_ball_small_n} holds if $A_i$ is a proper annulus, possibly with a different $\beta_n'$, but still dependent only on $n$. It is sufficient to split the integral $\int_{2A_i}(\Delta u_i)^2dv$ as the sum of the integrals of $(\Delta u_i)^2$ on the annulus $\frac{r_i}{2}\leq\delta_{p_i}(x)\leq r_i$ and on the annulus $R_i\leq\delta_{p_i}(x)\leq 2R_i$ and use \eqref{grad_ann_0} and \eqref{hyp_vol} in each case.
\item {\bf (higher dimensions).} Let $n>4$. Let $A_i$ be a ball of radius $r_i$. Then, by H\"older's inequality,
\begin{multline}\label{ineq_ball_large_n}
\int_{\Omega\cap 2A_i}(\Delta u_i)^2dv\leq |\Omega\cap 2A_i|^{1-\frac{4}{n}}\left(\int_{\Omega\cap 2A_i}(\Delta u_i)^{\frac{n}{2}}dv\right)^{\frac{4}{n}}\\
\leq  |\Omega\cap 2A_i|^{1-\frac{4}{n}}\left(\int_{2A_i}(\Delta u_i)^{\frac{n}{2}}dv\right)^{\frac{4}{n}}\\
\leq \left(\frac{|\Omega|}{j}\right)^{1-\frac{4}{n}}\left(\frac{3(n+1)}{r_i^2}+\frac{3(n-1)\kappa}{r_i}\right)^2|2A_i|^{\frac{4}{n}}\\
\leq\left(\frac{|\Omega|}{j}\right)^{1-\frac{4}{n}} \frac{72 (4(n+1)^2+(n-1)^2)}{r_i^4} (\omega_n\sinh(1)^{n-1}4^nr_i^n)^{\frac{4}{n}}\\
=\beta_n'' \left(\frac{|\Omega|}{j}\right)^{1-\frac{4}{n}},
\end{multline}
where
\begin{equation}\label{beta1_n}
\beta_n''=72(4(n+1)^2+(n-1)^2)(\omega_n\sinh(1)^{n-1})^{\frac{4}{n}}.
\end{equation}
In the same way it is possible to prove that \eqref{ineq_ball_large_n} holds if $A_i$ is a proper annulus, possibly with a different $\beta_n''$, but still dependent only on $n$ . It is sufficient to split the integral $\int_{2A_i}(\Delta u_i)^2dv$ as the sum of the integrals of $(\Delta u_i)^2$ on the annulus $\frac{r_i}{2}\leq\delta_{p_i}(x)\leq r_i$ and on the annulus $R_i\leq\delta_{p_i}(x)\leq 2R_i$ and use \eqref{grad_ann_0} and \eqref{hyp_vol} on each annulus.
\end{enumerate}

We have then proved that, for all dimensions $n\geq 2$,
\begin{equation}\label{general_bound_laplacian}
\int_{\Omega\cap 2A_i}(\Delta u_i)^2dv\leq \beta_n \left(\frac{j}{|\Omega|}\right)^{\frac{4}{n}-1},
\end{equation}
where $\beta_n$ is a constant which depends only on $n$ and is explicitly computable. This concludes the estimate of the numerator of the Rayleigh quotient for $u_i$. As for the denominator we have
$$
\int_{\Omega\cap 2A_i}u_i^2dv\geq \int_{\Omega\cap A_i}u_i^2dv=|\Omega\cap A_i|\geq c\frac{|\Omega|}{4j}.
$$
Then, we have for all $i=1,...,j$
\begin{multline}\label{general_bound_laplacian_0}
\frac{\int_{\Omega}|D^2u_i|^2+{\rm Ric}(\nabla u_i,\nabla u_i)dv}{\int_{\Omega}u_i^2dv}\leq\frac{4\beta_n}{c}\left(\frac{j}{|\Omega|}\right)^{\frac{4}{n}}+\frac{4\alpha_n}{c}\frac{|\partial\Omega|}{|\Omega|}\left(\frac{j}{|\Omega|}\right)^{\frac{3}{n}}\\
\leq \left(\frac{4\beta_n+3\alpha_n}{c}\right)\left(\frac{j}{|\Omega|}\right)^{\frac{4}{n}}+\frac{\alpha_n}{c}\frac{|\partial\Omega|^4}{|\Omega|^4},
\end{multline}
where we have used Young's inequality in the last passage. We have proved then that
\begin{equation}\label{large_j}
\mu_j\leq A_n\left(\frac{j}{|\Omega|}\right)^{\frac{4}{n}}+\frac{B_n}{2}\frac{|\partial\Omega|^4}{|\Omega|^4},
\end{equation}
 for all $j\geq j_{\Omega}$, where
\begin{equation}\label{A_n}
A_n:=\left(\frac{4\beta_n+3\alpha_n}{c}\right)
\end{equation}
and
\begin{equation}\label{B_n}
B_n:=2\frac{\alpha_n}{c}.
\end{equation}

\noindent{\bf Step 2 (small $j$).} Let now $j<j_{\Omega}$ be fixed. By using Theorem \ref{corollary_small_annuli} as in Step 1, we find that there exists a sequence of $4j$ sets $\left\{A_i\right\}_{i=1}^{4j}$ such that $|\Omega\cap A_i|\geq c\frac{|\Omega|}{4j}$. If the sets $A_i$ are annuli, we can proceed as in Step 1 and deduce the validity of \eqref{large_j}. Assume now that $j$ is such that the sets $A_i$ of the decomposition are of the form
$$
A_i=B(x_1^i,r_0)\cup\cdots\cup B(x_{l_i}^i,r_0),
$$
where $r_0=\frac{4a}{1600}$, $D_i=A_i^{4r_0}$ are pairwise disjoint, and $\delta_{x_l^i}(x_k^i)\geq 4r_0$ if $l\ne k$. By definition $D_i=B(x_1^i,5r_0)\cup\cdots\cup B(x_{l_i}^i,5r_0)$ and $5r_0\leq\frac{r_{inj,\Omega}}{160}$. Since we have $4j$ disjoint sets $D_i$, we can pick $j$ of them such that $|\Omega\cap D_i|\leq\frac{|\Omega|}{j}$ and $|\partial\Omega\cap D_i|\leq\frac{|\partial\Omega|}{j}$. We take from now on this family of $j$ capacitors. Note that $D_i$ is a disjoint union of $l_i$ balls $B(x_1^i,5r_0), \cdots, B(x_{l_i}^i,5r_0)$ of radius $5r_0$. Associated to each $B(x_k^i,5r_0)$, $k=1,...,l$ we construct test functions $u_k^i$ as in \eqref{ball}. Then we define the function $u_i$ associated with the capacitor $(A_i,D_i)$ by setting $u_i=u_k^i$ on $B(x_k^i,5r_0)$. We have $j$ disjointly supported test functions in $H^2(\Omega)$. We estimate the Rayleigh quotient of each of the $u_i$ as in Step 1. As in \eqref{grad_ball}, \eqref{lap_ball} and \eqref{nablanabla2_0} we estimate $|\nabla u_i|,|\Delta u_i|$ and $|\nabla|\nabla u_i|^2|$. In particular, we find a universal constant $c_0$  such that
\begin{equation}\label{est_small_j}
|\nabla u_i|\leq\frac{c_0}{r_0}\,,\ \ \ |\Delta u_i|\leq\frac{c_0}{r_0^2}+\frac{c_0\kappa}{r_0}\,,\ \ \ |\nabla|\nabla u_i|^2|\leq\frac{c_0}{r_0^3}.
\end{equation}
By using \eqref{bochner_main}, as we have done for \eqref{boundary_est}, \eqref{ineq_ball_small_n} and \eqref{ineq_ball_large_n}, we find that
$$
\frac{\int_{\Omega}|D^2u_i|^2+{\rm Ric}(\nabla u_i,\nabla u_i)dv}{\int_{\Omega}u_i^2dv}\leq \frac{A_n}{a^4}+\frac{B_n}{a^3}\frac{|\partial\Omega|}{|\Omega|},
$$
for all $i=1,...,j$, which implies, by using Young's inequality as in \eqref{general_bound_laplacian_0}, that
\begin{equation}\label{small_j}
\mu_j\leq \frac{3A_n}{4a^4}+\frac{B_n}{a^4}\frac{|\partial\Omega|}{4|\Omega|}.
\end{equation}

The proof of \eqref{main_estimate_general} follows by combining \eqref{large_j} and \eqref{small_j}, possibly re-defining the constants $A_n,B_n$.
\end{proof}

\begin{rem}
We point out that in the proof of Theorem \ref{main_general} the inequality \eqref{small_j} appears. Apparently this may look like a nonsense, in fact the right-hand side of the inequality does not depend on $j$. However, note that this situation may occur only for a finite number of eigenvalues $\mu_j$, since, starting from a certain $j_{\Omega}$ (of which it is possible in principle to give a lower bound), the capacitors of the decomposition given by \eqref{corollary_small_annuli} are of the form $iii)-a)$, hence the estimate \eqref{large_j} holds starting from a certain $j_{\Omega}$.

\end{rem}

In the next subsections we will present more refined estimates under additional assumptions.

%%%%%%%%%%%%%%%%%%%%%%%%%%%%%%%%%%%%%%%%%%%%%%%%%%%%%%%%%%%%%%%%%%%%%%%%%%%%%%%%%%%%%%%%%%%%%%%%%%%%%%%%%%%%%%%%%%%%%%%%%%%%%%%%%%%%%%%%

%%%%%%%%%%%%%%%%%%%%%%%%%%%%%%%%%%%%%%%%%%%%            Ric>0 and low dimension            %%%%%%%%%%%%%%%%%%%%%%%%%%%%%%%%%%%%%%%%%%%%% 

%%%%%%%%%%%%%%%%%%%%%%%%%%%%%%%%%%%%%%%%%%%%%%%%%%%%%%%%%%%%%%%%%%%%%%%%%%%%%%%%%%%%%%%%%%%%%%%%%%%%%%%%%%%%%%%%%%%%%%%%%%%%%%%%%%%%%%%%

\subsection{Manifolds with ${\rm Ric}\geq 0$ and $n=2,3,4$}\label{sub_small}
In this subsection we establish upper bounds for bounded domains in Riemannian manifolds satisfying ${\rm Ric}\geq 0$. In this case we will use Theorem \ref{gny}, noting that  the constant $\Gamma$ depends only on $n$. Moreover, test functions in this case are not given in terms of the distance from a point, thus avoiding the obstruction of the lack of regularity in correspondence of the cut-locus.

We state the main theorem of this subsection.

\begin{thm}\label{low_dim_positive}
Let $(M,g)$ be a complete $n$-dimensional smooth Riemannian manifolds with ${\rm Ric}\geq 0$ and $n=2,3,4$. Let $\Omega$ be a bounded domain of $M$ with $C^1$ boundary. Then
\begin{equation}\label{est_dim_positive}
\mu_j\leq A_n\left(\frac{j}{|\Omega|}\right)^{\frac{4}{n}},
\end{equation}
for all $j\in\mathbb N$, the constant $A_n$ depending only on the dimension.
\end{thm}

The proof is similar to that of Theorem \ref{main_general}. However we will use different test functions. To do so, we adapt a construction originally contained in \cite[Theorem 6.33]{cheeger_colding}. We have the following lemma (see \cite[Theorem 2.2]{guneysu}, see also \cite{bianchi_setti}).

\begin{lem}\label{cheeger}
Let $(M,g)$ be a complete $n$-dimensional smooth Riemannian manifold with ${\rm Ric}\geq 0$. Then, for any $p\in M$ and $r>0$ there exists a function $\phi_r:M\rightarrow[0,1]$, $\phi\in C^{\infty}(M)$, such that
\begin{enumerate}
\item $\phi_r\equiv 1$ on $B\left(p,\frac{r}{2}\right)$;
\item ${\rm supp}(\phi_r)\subset B(p,r)$;
\item $|\nabla\phi_r|\leq \frac{C(n)}{r}$;
\item $|\Delta\phi_r|\leq \frac{C(n)}{r^2}$,
\end{enumerate}
the constant $C(n)$ depending only on the dimension.
\end{lem}

\begin{proof}
Let $p\in M$ and $r>0$ be fixed. 
Let us consider another metric on $M$, namely $g_{r}:=\frac{1}{r^2}g$. In this new metric
$$
B(p,r)=B_{g_{r}}\left(p,1\right)
$$
and
$$
B\left(p,\frac{r}{2}\right)=B_{g_{r}}\left(p,\frac{1}{2}\right).
$$
Here we denote by $B(p,r)$ and $B\left(p,\frac{r}{2}\right)$ the balls of center $p$ and radius $r$ and $\frac{r}{2}$ in the original metric $g$. Moreover, ${\rm Ric}_{g_{r}}\geq 0$, hence from Theorem 6.33 of \cite{cheeger_colding} we deduce that there exists a constant $c(n)$ and $\phi:M\rightarrow [0,1]$, $\phi\in C^{\infty}(M)$, such that
\begin{enumerate}
\item $\phi\equiv 1$ on $B_{g_{r}}\left(p,\frac{1}{2}\right)$;
\item ${\rm supp}(\phi)\subset B_{g_{r}}(p,1)$;
\item $|\nabla\phi|\leq c(n)$;
\item $|\Delta\phi|\leq c(n)$.
\end{enumerate}
Since $
\Delta_{g_{r}}=r^2\Delta$ and $|\omega|_{g_{r}}^2=r^2|\omega|$ for all $1$-forms $\omega$ on $M$,
we conclude that
$$
|\nabla \phi|^2\leq \frac{c(n)}{r}
$$
and
$$
|\Delta\phi|^2\leq\frac{c(n)}{r^2.}
$$
This concludes the proof by taking $\phi_r:=\phi$.
\end{proof}

%We prove now Theorem \ref{low_dim_positive}.

\begin{proof}[Proof of Theorem \ref{low_dim_positive}]
We use Theorem \ref{gny}, which provides, for all indexes $j\in\mathbb N$, a family $\left\{A_i\right\}_{i=1}^j$ of annuli such that
$$
|A_i\cap\Omega|\geq c\frac{|\Omega|}{j}
$$
and the annuli $2A_i$ are pairwise disjoint. Moreover, from Corollary \ref{corollarygny0} we deduce that each annulus $A_i$ has either internal radius $r_i$ satisfying \eqref{gny-rad-est} or is a ball of radius $r_i$ satisfying \eqref{gny-rad-est}. 
In this case $\mathcal B=\left\{r\in\mathbb R:V(r)\geq c\frac{|\Omega|}{j}\right\}$ and $V(r):=\sup_{x\in\Omega}|B(x,r)\cap\Omega|$. Since ${\rm Ric}\geq 0$, from volume comparison we know that every ball $B(p,r)\subset M$ satisfies $|B(x,r)|\leq\omega_nr^n$. Hence any $r\in \mathcal B$ is such that
\begin{equation}\label{est_rad_2}
c\frac{|\Omega|}{j}\leq V(r)\leq\omega_nr^n\iff r\geq\left(\frac{c|\Omega|}{j\omega_n}\right)^{\frac{1}{n}}.
\end{equation}
Subordinated to the family $\left\{A_i\right\}_{i=1}^j$ we build a family of test functions $\left\{u_i\right\}_{i=1}^j$ in $H^2(\Omega)$ in the following way. If $A_i$ is a ball of radius $r_i$, we take $u_i=\phi_{r_i}$, where $\phi_{r_i}$ is defined in Lemma \ref{cheeger}. Hence $u_i$ is supported in $2A_i$ and $u_i\equiv 1$ on $A_i$. If $A_i$ is a proper annulus of radii $0<r_i<R_i$, we take $u_i=\phi_{R_i}-\phi_{r_i}$. Again, $u_i$ is supported in $2A_i$ and $u_i\equiv 1$ on $A_i$. The functions $\left\{{u_i}_{|_{\Omega}}\right\}_{i=1}^j$ are disjointly supported and belong to $H^2(\Omega)$. Hence, from \eqref{minmax} we deduce that
$$
\mu_j\leq\max_{i=1,...,j}\frac{\int_{\Omega}|D^2u_i|^2+{\rm Ric}(\nabla u_i,\nabla u_i)dv}{\int_{\Omega}u_i^2dv}.
$$
We estimate now the Rayleigh quotient of each of the $u_i$.  Assume that $A_i$ is a ball of radius $r_i$. We have, for the numerator
\begin{multline}\label{low_dim_step_1}
\int_{\Omega}|D^2u_i|^2+{\rm Ric}(\nabla u_i,\nabla u_i)dv=\int_{\Omega\cap 2A_i}|D^2u_i|^2+{\rm Ric}(\nabla u_i,\nabla u_i)dv\\
\leq \int_{2A_i}|D^2u_i|^2+{\rm Ric}(\nabla u_i,\nabla u_i)dv=\int_{2A_i}(\Delta u_i)^2dv\leq \frac{|2A_i|c(n)^2}{r_i^4}\\
\leq c(n)^2\omega_n r_i^{n-4}\leq c(n)^2\omega_n^{\frac{4}{n}}{c}^{1-\frac{4}{n}}\left(\frac{|\Omega|}{j}\right)^{1-\frac{4}{n}}.
\end{multline}
In the first inequality we have estimated $\int_{\Omega\cap 2A_i}|D^2u_i|^2+{\rm Ric}(\nabla u_i,\nabla u_i)dv$ with the integral on the whole ball $2A_i$, being the integrand a non-negative function. Moreover, from Bochner's formula, since $u_i\in H^2_0(2A_i)$, we have that $\int_{2A_i}|D^2u_i|^2+{\rm Ric}(\nabla u_i,\nabla u_i)dv=\int_{2A_i}(\Delta u_i)^2dv\leq \frac{|2A_i|c(n)^2}{r_i^4}$. Finally we have used \eqref{est_rad_2} since $n\leq 4$.

For the denominator we have
\begin{equation}\label{low_dim_step_2}
\int_{\Omega}u_i^2dv=\int_{\Omega\cap 2A_i}u_i^2dv\geq\int_{\Omega\cap A_i}u_i^2dv=|\Omega\cap 2A_i|\geq c\frac{|\Omega|}{j}.
\end{equation}
From \eqref{low_dim_step_1} and \eqref{low_dim_step_2} we deduce
\begin{equation}\label{low_dim_step_3}
\frac{\int_{\Omega}|D^2u_i|^2+{\rm Ric}(\nabla u_i,\nabla u_i)dv}{\int_{\Omega}u_i^2dv}\leq c(n)^2\frac{\omega_n^{\frac{4}{n}}}{{c}^{\frac{4}{n}}}\left(\frac{j}{|\Omega|}\right)^{\frac{4}{n}}.
\end{equation}
In the very same way it is possible to prove that \eqref{low_dim_step_1} holds if $A_i$ is a proper annulus, possibly with a different dimensional constant in front of the term $\left(\frac{j}{|\Omega|}\right)^{\frac{4}{n}}$. It is sufficient to split the integral $\int_{2A_i}(\Delta u_i)^2dv$ as the sum of the integrals of $(\Delta u_i)^2$ on the annulus $\frac{r_i}{2}\leq\delta_{p_i}(x)\leq r_i$ and on the annulus $R_i\leq\delta_{p_i}(x)\leq 2R_i$.

The proof is now complete.
\end{proof}

\begin{rem}
We remark that, differently from the proof of Theorem \ref{main_general}, we did not choose the annuli of the decomposition in such a way that $|\Omega\cap 2A_i|\leq\frac{|\Omega|}{j}$. In fact, in the case $n=2,3,4$, an upper bound on the size of the supports of test functions seems to be irrelevant in the estimates. Moreover, being ${\rm Ric}\geq 0$, the quadratic form $|D^2u|^2+{\rm Ric}(\nabla u,\nabla u)$ is always non-negative, hence we can estimate its integral over $\Omega\cap 2A_i$ with the whole integral over $2A_i$ and use Bochner's formula. Of course we can do this passage also for $n>4$. However, in this case we would obtain
\begin{multline}\label{low_dim_step_4}
\int_{\Omega\cap 2A_i}|D^2u_i|^2+{\rm Ric}(\nabla u_i,\nabla u_i)dv\leq \int_{2A_i}|D^2u_i|^2+{\rm Ric}(\nabla u_i,\nabla u_i)dv\\
=\int_{2A_i}(\Delta u_i)^2dv\leq \frac{|2A_i|c(n)^2}{r_i^4}\leq c(n)^2\omega_n r_i^{n-4},
\end{multline}
but $n-4>0$, and inequality \eqref{low_dim_step_4} is useless to obtain uniform estimates.

In the case $n>4$ the strategy would be rather to choose, according to Theorem \ref{gny}, a family $\left\{A_i\right\}_{i=1}^{4j}$ annuli such that $|\Omega\cap A_i|\geq c\frac{|\Omega|}{4j}$ and then choose $j$ annuli among the $4j$ of the family in such a way that $|\Omega\cap 2A_i|\leq\frac{|\Omega|}{j}$ and $|\partial\Omega\cap 2A_i|\leq\frac{|\partial\Omega|}{j}$, as in the proof of Theorem \ref{main_general}.  We build then test functions $u_i$ as in the proof of Theorem \ref{low_dim_positive}. From Bochner's formula, as in \eqref{bochner_main}, we have
\begin{multline*}
\int_{\Omega\cap 2A_i}|D^2u_i|^2+{\rm Ric}(\nabla u_i,\nabla u_i)dv\\
\leq \int_{\partial\Omega\cap 2A_i}\frac{1}{2}|\nabla|\nabla u_i|^2|+|\Delta u_i||\nabla u_i|d\sigma+\int_{\Omega\cap 2A_i}(\Delta u_i)^2dv.
\end{multline*}
Assume that $A_i$ is a ball of radius $r_i$. The term $\int_{\partial\Omega\cap 2A_i}|\Delta u_i||\nabla u_i|d\sigma$ is estimated by
\begin{equation*}
\int_{\partial\Omega\cap 2A_i}|\Delta u_i||\nabla u_i|d\sigma\leq\frac{|\partial\Omega|}{j}\frac{c(n)^2}{r_i^3}\leq \frac{|\partial\Omega|}{j} \frac{c(n)^2\omega_n^{\frac{3}{n}}}{{c}^{\frac{3}{n}}}\left(\frac{j}{|\Omega|}\right)^{\frac{3}{n}}
\end{equation*}
For the term $\int_{\Omega\cap 2A_i}(\Delta u_i)^2dv$ we have
\begin{multline}
\int_{\Omega\cap 2A_i}(\Delta u_i)^2dv\leq\left(\int_{\Omega\cap 2A_i}(\Delta u_i)^\frac{n}{2}dv\right)^{\frac{4}{n}}|\Omega\cap 2A_i|^{1-\frac{4}{n}}\\
\leq \left(\int_{2A_i}(\Delta u_i)^\frac{n}{2}dv\right)^{\frac{4}{n}}\left(\frac{|\Omega|}{j}\right)^{1-\frac{4}{n}}\leq c(n)^2\frac{|2A_i|^\frac{4}{n}}{r_i^4}\left(\frac{|\Omega|}{j}\right)^{1-\frac{4}{n}}\\
\leq 16c(n)^2\omega_n^{\frac{4}{n}}\left(\frac{|\Omega|}{j}\right)^{1-\frac{4}{n}}.
\end{multline}
Hence, as for \eqref{general_bound_laplacian_0}, we find that
\begin{equation}\label{bound_ric_hig}
\frac{\int_{\Omega}|D^2u_i|^2+{\rm Ric}(\nabla u_i,\nabla u_i)dv}{\int_{\Omega}u_i^2dv}\leq A_n\left(\frac{j}{|\Omega|}\right)^{\frac{4}{n}}+B_n \frac{|\partial\Omega|^4}{|\Omega|^4}+C_n\|\nabla|\nabla u_i|^2\|_{\infty},
\end{equation}
for some constants $A_n,B_n,C_n$ which depend only on the dimension. In the case that $A_i$ is a proper annulus, inequality \eqref{bound_ric_hig} still holds, with possibly different dimensional constants. Unfortunately an estimate of the form $|\nabla|\nabla \phi_r|^2|\leq \frac{c(n)}{r^3}$ is not available for a function $\phi_r$ as in Lemma \ref{cheeger}. If such an inequality would hold, then we would immediately have
\begin{equation}
\frac{\int_{\Omega}|D^2u_i|^2+{\rm Ric}(\nabla u_i,\nabla u_i)dv}{\int_{\Omega}u_i^2dv}\leq A_n\left(\frac{j}{|\Omega|}\right)^{\frac{4}{n}}+B_n \frac{|\partial\Omega|^4}{|\Omega|^4},
\end{equation}
and therefore
\begin{equation}\label{bound_positive_high}
\mu_j\leq A_n\left(\frac{j}{|\Omega|}\right)^{\frac{4}{n}}+B_n \frac{|\partial\Omega|^4}{|\Omega|^4}.
\end{equation}
{\bf Open problem.} Prove \eqref{bound_positive_high}  for domains in complete smooth manifolds with ${\rm Ric}\geq 0$ and $n>4$. Prove inequality \eqref{est_dim_positive} for domains in complete smooth manifolds with ${\rm Ric}\geq 0$ and $n\geq 2$.
\end{rem}

%%%%%%%%%%%%%%%%%%%%%%%%%%%%%%%%%%%%%%%%%%%%%%%%%%%%%%%%%%%%%%%%%%%%%%%%%%%%%%%%%%%%%%%%%%%%%%%%%%%%%%%%%%%%%%%%%%%%%%%%%%%%%%%%%%%%%%%%

%%%%%%%%%%%%%%%%%%%%%%%%%%%%%%%%%%%%%%%%%%%%                Ric>0 and small diamet         %%%%%%%%%%%%%%%%%%%%%%%%%%%%%%%%%%%%%%%%%%%%% 

%%%%%%%%%%%%%%%%%%%%%%%%%%%%%%%%%%%%%%%%%%%%%%%%%%%%%%%%%%%%%%%%%%%%%%%%%%%%%%%%%%%%%%%%%%%%%%%%%%%%%%%%%%%%%%%%%%%%%%%%%%%%%%%%%%%%%%%%

\subsection{Manifolds with ${\rm Ric}\geq 0$ and small diameter}\label{sub_small_2}

If ${\rm Ric}\geq 0$ and the diameter is sufficiently small compared to $r_{inj,\Omega}$, it is possible to build test functions for the Rayleigh quotient in terms of the distance function. In particular, we have the following.

\begin{thm}\label{thm_ric_pos_dist}
Let $(M,g)$ be a complete $n$-dimensional smooth Riemannian manifold with ${\rm Ric}\geq 0$ and let $\Omega$ be a bounded domain in $M$ with $C^1$ boundary and diameter $D$. If $D<\frac{r_{inj,\Omega}}{2}$, then
\begin{equation}\label{ineq_ric_pos_dist}
\mu_j\leq A_n\left(\frac{j}{|\Omega|}\right)^{\frac{4}{n}}+B_n\frac{|\partial\Omega|^4}{|\Omega|^4},
\end{equation}
for all $j\in\mathbb N$, where $A_n,B_n$ depend only on the dimension. If $n\leq 4$ we can choose $B_n=0$.
\end{thm}

\begin{proof}
In order to prove Theorem \ref{thm_ric_pos_dist} we exploit Theorem \ref{gny} and Corollary \ref{corollarygny0}. We find, as in the proof of Theorem \ref{main_general}, for all $j\in\mathbb N$ a family of $j$ annuli $\left\{A_i\right\}_{i=1}^j$ such that $|\Omega\cap A_i|\geq c\frac{|\Omega|}{4j}$, $|\partial\Omega\cap2A_i|\leq \frac{|\partial\Omega|}{j}$, $|\Omega\cap2A_i|\leq \frac{|\Omega|}{j}$ and the annuli $2A_i$ are pairwise disjoint. Moreover, since we have taken $D<\frac{r_{inj,\Omega}}{2}$, the annuli $2A_i$ can be chosen such that the outer radius is strictly smaller that $r_{inj,\Omega}$. Associated with $A_i$ we then build test functions $u_i$ of the form \eqref{ball} if $A_i$ is a ball or \eqref{prop_annulus} if $A_i$ is a proper annulus. It follows then, as for the proof of \eqref{general_bound_laplacian_0}
$$
\frac{\int_{\Omega}|D^2u_i|^2+{\rm Ric}(\nabla u_i,\nabla u_i)dv}{\int_{\Omega}u_i^2dv} \leq A_n\left(\frac{j}{|\Omega|}\right)^{\frac{4}{n}}+B_n\frac{|\partial\Omega|^4}{|\Omega|^4},
$$
where the constants $A_n$, $B_n$ depend only on the dimension. This implies \eqref{ineq_ric_pos_dist} for all $j\in\mathbb N$. The last statement follows immediately from Theorem \ref{low_dim_positive}. This concludes the proof.
\end{proof}

%\begin{rem}
%This theorem allows to find upper bounds with the correct asymptotc behavior for domains in the hemisphere. Actually, it will be possible to obtain estimates for domains of the sphere. This will be done in the following section.
%\end{rem}

%%%%%%%%%%%%%%%%%%%%%%%%%%%%%%%%%%%%%%%%%%%%%%%%%%%%%%%%%%%%%%%%%%%%%%%%%%%%%%%%%%%%%%%%%%%%%%%%%%%%%%%%%%%%%%%%%%%%%%%%%%%%%%%%%%%%%%%%

%%%%%%%%%%%%%%%%%%%%%%%%%%%%%%%%%%%%%%%%%%%%                DOMAINS OF THE SPHERE          %%%%%%%%%%%%%%%%%%%%%%%%%%%%%%%%%%%%%%%%%%%%% 

%%%%%%%%%%%%%%%%%%%%%%%%%%%%%%%%%%%%%%%%%%%%%%%%%%%%%%%%%%%%%%%%%%%%%%%%%%%%%%%%%%%%%%%%%%%%%%%%%%%%%%%%%%%%%%%%%%%%%%%%%%%%%%%%%%%%%%%%

\subsection{Domains on the sphere}\label{sub_sphere}
In this section we obtain bounds for domains of the standard sphere $\mathbb S_R^n$, namely we have the following theorem.

\begin{thm}\label{thm_sphere} 
Let $(M,g)=\mathbb S_R^n$ be the sphere of radius $R$ with standard round metric and let $\Omega$ be a  domain in $\mathbb S^n$ with $C^1$ boundary. Then
$$
\mu_j\leq A_n\left(\frac{j}{|\Omega|}\right)^{\frac{4}{n}},
$$
for all $j\in\mathbb N$.
\end{thm}

\begin{proof}
%We have that $Ric(X,Y)=\frac{(n-1)}{R^2}X\cdot Y$ (here we write $X\cdot Y=g(X,Y)$). 
Through all the proof we assume $n>4$. The validity of the theorem for $n\leq 4$ follows from Theorem \ref{low_dim_positive}. Given two points $p_1,p_2\in\mathbb S^n_R$, the maximal distance among them is attained when they are antipodal points. In this case $\delta_{p_1}(p_2)=\pi R$.  Moreover, $\mu_1=0$ and the corresponding eigenfunctions are the constant functions on $\Omega$, and $\mu_2>0$ (see Section \ref{few_pos}).

%As in the proof of Theorem \ref{main_general}, assume that for all $j\geq 2$ we have $j$ disjointly supported functions $u_1,...,u_j\in H^2(\Omega)$. Then from \eqref{minmax} we conclude that
%\begin{equation}%\label{minmax2}
%\mu_j\leq\max_{i=1,...,j}\frac{\int_{\Omega}|D^2u_i|^2+Ric(\nabla u_i,\nabla u_i)dvol}{\int_{\Omega}u_i^2dvol},
%\end{equation}
%that is, it will be sufficient to estimate the Rayleigh quotient of each of the $k$ test functions.

We apply Theorem \ref{gny} with $X=\Omega$ with the Riemannian distance, $\varsigma(A)=|A\cap\Omega|$. Since we have positive Ricci curvature, point $i)$ of Theorem \ref{gny} is satisfied for some $\Gamma$ which depends only on $n$. Points ii) and iii) are easily seen to hold. Hence we deduce that there exists $c$ which depends only on the dimension such that for any $j\geq 2$, there exists $A_1,...,A_j$ annuli with
\begin{enumerate}[i)]
\item $|A_i\cap\Omega|\geq c\frac{|\Omega|}{j}$;
\item $2A_i$ are pairwise disjoint;
\item $|2A_i\cap\Omega|\leq\frac{|\Omega|}{j}$;
%\item $|2A_i\cap\partial\Omega|\leq\frac{|\partial\Omega|}{j}$
\item each annulus $2A_i$ has outer radius less than $\frac{\pi R}{2}$.
\end{enumerate}
Points iii)  and iv) follow from the fact that we shall actually apply Theorem \ref{gny} with $2j+1$ and by observing that we can choose first $2j$ among the $2j+1$ annuli given by the construction so that iv) holds. Indeed, if an annulus $A_i$ with center $p_i$ has outer radius strictly greater than $\frac{\pi R}{2}$, the set $\mathbb S^n_R\setminus2\overline A_i$ is $B(p_1,r)\cup B(p_i',r')$, where $p_i'$ is the antipodal point of $p_i$, $r,r'<\frac{\pi R}{2}$ ($B(p',r')=\emptyset$ if  $A_i$ is a ball), and all other annuli $2A_k$ of the decomposition need belong to $B(p_1,r)\cup B(p_i',r')$. Hence, no more than one annulus can satisfy iv). Moreover, among the remaining $2j$ annuli, we can chose $j$ annuli such that  iii) holds (see also the proof of Theorem \ref{main_general}). We shall denote by $r_i$ and $R_i$ the inner and outer radius of $A_i$, if $A_i$ is an actual annulus, while we shall denote by $r_i$ the radius, if $A_i$ is a ball. By $p_i$ we denote the center of the annuli $A_i$.

Associated to each of the $j$ annuli $A_1,...,A_j$ satisfying i)-iv) we construct test functions $u_i$ as in \eqref{ball} (if $A_i$ is a ball) or in \eqref{prop_annulus} if $A_i$ is a proper annulus. The functions $u_i$ are of the form $u_i(x)=f\left(\delta_p(x)\right)$.

We apply  now Bochner's formula to a function of the form $f(\delta_p(x))$, and we use the fact that $|\nabla \delta_p(x)|=1$ almost everywhere. We have
\begin{multline}
\frac{1}{2}\Delta\left(|\nabla f(\delta_p(x))|^2\right)=\frac{1}{2}\Delta\left(|f'(\delta_p(x))\nabla\delta_p(x) |^2\right)\\
=\frac{1}{2}\Delta\left(f'(\delta_p(x))^2\right)\\
=\nabla(f'(\delta_p(x))f''(\delta_p(x)))\cdot\nabla\delta_p(x)+f'(\delta_p(x))f''(\delta_p(x))\Delta\delta_p(x)\\
=f''(\delta_p(x))^2+f'(\delta_p(x))f'''(\delta_p(x))+f'(\delta_p(x))f''(\delta_p(x))\Delta\delta_p(x).
\end{multline}
On the other hand
\begin{multline}
\nabla\Delta f(\delta_p(x))\cdot\nabla f(\delta_p(x))\\
=f'(\delta_p(x))\nabla\delta_p(x)\cdot\nabla(f''(\delta_p(x))+f'(\delta_p(x))\Delta\delta_p(x))\\
=f'(\delta_p(x))f'''(\delta_p(x))+f'(\delta_p(x))f''(\delta_p(x))\Delta\delta_p(x)\\
+(f'(\delta_p(x)))^2\nabla\delta_p(x)\cdot\nabla\Delta\delta_p(x).
\end{multline}
We deduce then
\begin{multline}\label{bochner_dist}
|D^2f(\delta_p(x))|^2+{\rm Ric}(\nabla f(\delta_p(x)),\nabla f(\delta_p(x)))\\=f''(\delta_p(x))^2-(f'(\delta_p(x)))^2\nabla\delta_p(x)\cdot\nabla\Delta\delta_p(x).
\end{multline}
Moreover,
\begin{multline}
\nabla\delta_p(x)\cdot\nabla\Delta\delta_p(x)=\nabla\left(\frac{(n-1)}{R}\cot\left(\frac{\delta_p(x)}{R}\right)\right)\cdot\nabla\delta_p(x)\\
=-\frac{n-1}{R^2\sin^2(\delta_p(x)/R)}|\nabla\delta_p(x)|^2=-\frac{n-1}{R^2\sin^2(\delta_p(x)/R)}
\end{multline}
for all $x\ne p,p'$, where $p'$ is the antipodal point to $p$, and
$$
\sin^2(\delta_p(x)/R)\leq\frac{4\delta_p(x)^2}{\pi^2R^2},
$$
for all $x$ such that $0\leq\delta_p(x)\leq\frac{\pi R}{2}$.

Now, since $|D^2f(\delta_p(x))|^2+{\rm Ric}(\nabla f(\delta_p(x)),\nabla f(\delta_p(x)))\geq 0$ and $n>4$,
\begin{multline}\label{chain_sphere}
\int_{\Omega\cap 2A_i} |D^2f(\delta_p(x))|^2+{\rm Ric}(\nabla f(\delta_p(x))dv\\
\leq |\Omega\cap 2A_i|^{1-\frac{4}{n}}\left(\int_{\Omega\cap 2A_i}\left(|D^2f(\delta_p(x))|^2+{\rm Ric}(\nabla f(\delta_p(x)),\nabla f(\delta_p(x)))\right)^{\frac{n}{4}}dv\right)^{\frac{4}{n}}\\
\leq |\Omega\cap 2A_i|^{1-\frac{4}{n}}\left(\int_{2A_i}\left(|D^2f(\delta_p(x))|^2+{\rm Ric}(\nabla f(\delta_p(x)),\nabla f(\delta_p(x)))\right)^{\frac{n}{4}}dv\right)^{\frac{4}{n}}\\
=  |\Omega\cap 2A_i|^{1-\frac{4}{n}}\left(\int_{2A_i}\left(f''(\delta_p(x))-(f'(\delta_p(x)))^2\nabla\delta_p(x)\cdot\nabla\Delta\delta_p(x)\right)^{\frac{n}{4}}dv\right)^{\frac{4}{n}}\\
=|\Omega\cap 2A_i|^{1-\frac{4}{n}}\left(\int_{2A_i}\left(f''(\delta_p(x))^2+(f'(\delta_p(x)))^2\frac{(n-1)}{R^2\sin^2(\delta_p(x)/R)}\right)^{\frac{n}{4}}dv\right)^{\frac{4}{n}}\\
\leq |\Omega\cap 2A_i|^{1-\frac{4}{n}}\left(\int_{2A_i}\left(f''(\delta_p(x))^2+\frac{(n-1)\pi^2(f'(\delta_p(x)))^2}{4\delta_p(x)^2}\right)^{\frac{n}{4}}dv\right)^{\frac{4}{n}}\\
\leq A_n'|\Omega\cap 2A_i|^{1-\frac{4}{n}}\leq A_n' \left(\frac{|\Omega|}{j}\right)^{1-\frac{4}{n}},
\end{multline}
where the constant $A_n'$ depends only on the dimension. For the denominator, we have
$$
\int_{\Omega\cap 2A_i}u_i^2dv\geq\int_{\Omega\cap A_i}u_i^2dv=|\Omega\cap A_i|\geq c\frac{|\Omega|}{j}.
$$
The proof follows now the same lines as that of Theorem \ref{main_general}.

\end{proof}

\begin{rem}
%In this case we have used functions which explicitly depend on $\delta_p(x)$ in order to use Bochner's formula to get \eqref{bochner_dist} and have fairly simple quantities (gradient and Laplacian) to estimate. In fact it is not natural, in the case of ${\rm Ric}\geq 0$ to split the quadratic form $|D^2u|^2+{\rm Ric}(\nabla u,\nabla u)$ and estimate separately the Hessian and the gradient terms. We would obtain in the estimate an upper bound for the Ricci curvature which is somehow unnatural. On the other hand, in some cases when ${\rm Ric}\leq 0$ and ${\rm Sec}\leq 0$, it is more convenient to estimate directly the Hessian. This we shall see in the next subsection.
We remark that explicit constructions like the one just presented for domains on the sphere are difficult already in other cases of manifolds for which we know the exact structure of the cut-locus of a point. This is the case of domains on an infinitely long cylinder. In this case, obtaining good estimates with the technique used in the proof of Theorem \ref{thm_sphere} seems quite involved.
\end{rem}

%%%%%%%%%%%%%%%%%%%%%%%%%%%%%%%%%%%%%%%%%%%%%%%%%%%%%%%%%%%%%%%%%%%%%%%%%%%%%%%%%%%%%%%%%%%%%%%%%%%%%%%%%%%%%%%%%%%%%%%%%%%%%%%%%%%%%%%%

%%%%%%%%%%%%%%%%%%%%%%%%%%%%%%%%%%%%%%%%%%%%       DOMAINS OF THE hyperbolic SPACE         %%%%%%%%%%%%%%%%%%%%%%%%%%%%%%%%%%%%%%%%%%%%% 

%%%%%%%%%%%%%%%%%%%%%%%%%%%%%%%%%%%%%%%%%%%%%%%%%%%%%%%%%%%%%%%%%%%%%%%%%%%%%%%%%%%%%%%%%%%%%%%%%%%%%%%%%%%%%%%%%%%%%%%%%%%%%%%%%%%%%%%%

\subsection{Domains of the hyperbolic space}\label{sub_hyp}

In this subsection we provide estimates for domains of the standard hyperbolic space $\mathbb H^n_{\kappa}$. We have the following theorem.
\begin{thm}\label{main_hyp}
Let $(M,g)=\mathbb H^n_{\kappa}$ be the standard $n$-dimensional hyperbolic space of curvature $-\kappa$, $\kappa>0$ and let $\Omega$ be a bounded domain in $\mathbb H^n_{\kappa}$ with $C^1$ boundary. Then 
\begin{equation}\label{est_hyp}
\mu_j\leq A_n\left(\frac{j}{|\Omega|}\right)^{\frac{4}{n}}+B_n\kappa^4,
\end{equation}
for all $j\in\mathbb N$.
\end{thm}

\begin{proof}
We apply Theorem \ref{corollary_small_annuli} with $a=\frac{1}{\kappa}$. 

\noindent {\bf Step 1.} There exists $j\in\mathbb N$ such that for all $j\geq j_{\Omega}$ we find a sequence of $j$ annuli $\left\{A_i\right\}_{i=1}^j$ such that $|\Omega\cap A_i|\geq c\frac{|\Omega|}{2j}$,  $|\Omega\cap 2_i|\leq\frac{|\Omega|}{j}$ and $2A_i$ pairwise disjoint. 

Associated to $A_i$ we construct test functions $u_i$ given by \eqref{ball} if $A_i$ is a ball, or by \eqref{prop_annulus} if $A_i$ is an a proper annulus. The distance function $\delta_p$ is smooth on all $\mathbb H^n_{\kappa}\setminus\left\{p\right\}$, for all $p\in \mathbb H^n_{\kappa}$, hence by construction ${u_i}_{|_{\Omega}}\in H^2(\Omega)$. We estimate now the Rayleigh quotient of the $u_i$. From \eqref{bochner_dist} and from the fact that
$$
\nabla\delta_p\cdot\nabla\Delta\delta_p=-\frac{(n-1)^2\kappa^2}{\sinh^2(\kappa\delta_p)}\leq-\frac{(n-1)^2}{\delta_p^2}
$$
we deduce, as in \eqref{chain_sphere}, that
\begin{multline}\label{chain_hyp}
\int_{\Omega\cap 2A_i} |D^2f(\delta_p(x))|^2+{\rm Ric}(\nabla f(\delta_p(x))dv\\
\leq |\Omega\cap 2A_i|^{1-\frac{4}{n}}\left(\int_{2A_i}\left(f''(\delta_p(x))^2+(f'(\delta_p(x)))^2\frac{(n-1)}{\delta_p(x)^2}\right)^{\frac{n}{4}}dv\right)^{\frac{4}{n}}\\
\leq A_n'|\Omega\cap 2A_i|^{1-\frac{4}{n}}\leq A_n' \left(\frac{|\Omega|}{j}\right)^{1-\frac{4}{n}},
\end{multline}
 if $n>4$.

If $n=2,3,4$ and $A_i$ is a ball of radius $r_i$
\begin{multline}\label{last_ineq}
\int_{\Omega\cap 2A_i} |D^2f(\delta_p(x))|^2+{\rm Ric}(\nabla f(\delta_p(x))dv\\
=\int_{\Omega\cap 2A_i}f''(\delta_p(x))^2+(f'(\delta_p(x)))^2\frac{(n-1)}{\delta_p(x)^2}dv\\
\leq A_n'\frac{|\Omega\cap 2A_i|}{r_i^4}\leq A_n'\frac{|2A_i|}{r_i^4}\leq A_n''r_i^{n-4}\\
\leq A_n'''\left(\frac{|\Omega|}{j}\right)^{1-\frac{4}{n}},
\end{multline}
where we have used \eqref{lower_ri}. Analogous computations show that inequality \eqref{last_ineq} holds also in the case that $A_i$ is a proper annulus, possibly with a different value of the constant $A_n'''$.

For the denominator of the Rayleigh quotient we have
$$
\int_{\Omega\cap 2A_i}u_i^2dv\geq\int_{\Omega\cap A_i}u_i^2dv=|\Omega\cap A_i|\geq c\frac{|\Omega|}{j}.
$$
Therefore
\begin{equation}\label{hyp_large_j}
\mu_j\leq A_n\left(\frac{j}{|\Omega|}\right)^{\frac{4}{n}},
\end{equation}
for all $j\geq j_{\Omega}$.

\noindent{\bf Step 2.} If $j<j_{\Omega}$, we proceed as in Step 2 of the proof of Theorem \ref{main_general}. By applying Theorem \ref{corollary_small_annuli} we find that there exists a family $\left\{A_i\right\}_{i=1}^{2j}$ of sets with the property $|\Omega\cap A_i|\geq c\frac{|\Omega|}{2j}$. If the $A_i$ are annuli, we proceed as in Step 1 and deduce the validity of \eqref{hyp_large_j}. Assume now that Theorem \ref{corollary_small_annuli} provides $2j$ sets such that
$$
A_i=B(x_1^i,r_0)\cup\cdots\cup B(x_{l_i}^i,r_0),
$$
with $r_0=\frac{4a}{1600}$, $D_i=A_i^{4r_0}$ are pairwise disjoint, and $\delta_{x_l^i}(x_k^i)\geq 4r_0$ if $l\ne k$. Since we have $2j$ disjoint sets, we can pick $j$ of them such that $|\Omega\cap D_i|\leq\frac{|\Omega|}{j}$. Note that each $D_i$ is a disjoint union of balls $B(x_1^i,r_0),...,B(x_{l_i}^i,r_0)$.  Associated with each $B(x_k^i,5r_0)$ we define a test function $u_k^i$ as in \eqref{ball}. Then, for any $i=1,...,j$ we define the function $u_i$ by setting $u_i=u_k^i$ on $B(x_k^i,r_0)$. Now, analogous computations as those in Step 1 allow to conclude that
\begin{equation}\label{hyp_small_j}
\mu_j\leq \frac{B_n}{a^4}.
\end{equation}
Since $a=\frac{1}{\kappa}$, from \eqref{hyp_large_j} and \eqref{hyp_small_j} we deduce the validity of \eqref{est_hyp}. This concludes the proof.

\end{proof}

%\begin{rem}
%If we know that for $j<j_{\Omega}$ there exists a decomposition by sets $\left\{E_i\right\}_{i=1^j}$ such that $|\Omega\cap E_i|\geq c\frac{|\Omega|}{j}$ and the sets $E_i':=\left\{x\in\Omega:{\rm dist}(x,E_i)<\frac{1}{1600\kappa}\right\}$ are pairwise disjont, and moreover each set $E_i'$ is a disjoint union of balls, than the additive constant in Theorem \ref{main_hyp} would be $B_n\kappa^4$ for some dimensional constant $B_n$.
%\end{rem}

%{\bf Open problem.} Prove \eqref{est_hyp} with the second term in the right-hand side replaced by $B_n\kappa^4$.
%%%%%%%%%%%%%%%%%%%%%%%%%%%%%%%%%%%%%%%%%%%%%%%%%%%%%%%%%%%%%%%%%%%%%%%%%%%%%%%%%%%%%%%%%%%%%%%%%%%%%%%%%%%%%%%%%%%%%%%%%%%%%%%%%%%%%%%%

%%%%%%%%%%%%%%%%%%%%%%%%%%%%%%%%%%%%%%%%%%%%               DOMAINS IN C-H MANIFOLD         %%%%%%%%%%%%%%%%%%%%%%%%%%%%%%%%%%%%%%%%%%%%% 

%%%%%%%%%%%%%%%%%%%%%%%%%%%%%%%%%%%%%%%%%%%%%%%%%%%%%%%%%%%%%%%%%%%%%%%%%%%%%%%%%%%%%%%%%%%%%%%%%%%%%%%%%%%%%%%%%%%%%%%%%%%%%%%%%%%%%%%%

\subsection{Domains of Cartan-Hadamard manifolds}\label{sub_cart}
A Cartan-Hadamard manifold is a complete, simply-connected Riemannian manifold $(M,g)$  with non-positive sectional curvature. As a corollary of Theorem \ref{main_general} we have the following.
\begin{thm}\label{main_cartan}
Let $(M,g)$ be a $n$-dimensional Cartan-Hadamard manifold with ${\rm Ric}\geq-(n-1)\kappa^2$, $\kappa>0$ and let $\Omega$ be a bounded domain in $M$ of class $C^1$. Then
$$
\mu_j\leq A_n\left(\frac{j}{|\Omega|}\right)^{\frac{4}{n}}+B_n\frac{|\partial\Omega|^4}{|\Omega|^4}+C_n\kappa^4,
$$
for all $j\in\mathbb N$
\end{thm}
\begin{proof}
We note that we can choose $a=\frac{1}{\kappa}$ in Theorem \ref{main_general}, since for any $p\in  M$, $\delta_p$ is smooth on the whole $M\setminus\left\{p\right\}$. Moreover, $\Delta\delta_p\geq 0$ for any $p\in M$, see e.g., \cite{grillo_vasquez}.
\end{proof}

\begin{rem}
A way of getting rid of the term $\frac{|\partial\Omega|^4}{|\Omega|^4}$ is to have explicit expressions for the right-hand side of \eqref{bochner_dist}. This is the case of domains of standard spheres or for the hyperbolic space. However, we note that $\nabla\delta_p\cdot\nabla\Delta\delta_p$ is exactly $\frac{\partial \mathcal H(r)}{\partial r}$, where $\mathcal H(r)$ is the mean curvature of the sphere centered at $p$ in $M$ and $r$ is the radial direction. From Bochner's formula we can just recover
$$
\nabla\delta_p\cdot\nabla\Delta\delta_p=-|D^2\delta_p|^2-{\rm Ric}(\nabla\delta_p,\nabla\delta_p)\leq-\frac{(\Delta\delta_p)^2}{n}+(n-1)\kappa^2.
$$
However, a lower bound for such quantity is needed. Otherwise, we necessarily have to pass through an integration by parts as in Theorem \ref{main_general}, and this involves boundary terms.
\end{rem}

%%%%%%%%%%%%%%%%%%%%%%%%%%%%%%%%%%%%%%%%%%%%%%%%%%%%%%%%%%%%%%%%%%%%%%%%%%%%%%%%%%%%%%%%%%%%%%%%%%%%%%%%%%%%%%%%%%%%%%%%%%%%%%%%%%%%%%%%

%%%%%%%%%%%%%%%%%%%%%%%%%%%%%%%%%%%%%%%%%%%%               MANIFOLDS WITHOUT BOUNDARY      %%%%%%%%%%%%%%%%%%%%%%%%%%%%%%%%%%%%%%%%%%%%% 

%%%%%%%%%%%%%%%%%%%%%%%%%%%%%%%%%%%%%%%%%%%%%%%%%%%%%%%%%%%%%%%%%%%%%%%%%%%%%%%%%%%%%%%%%%%%%%%%%%%%%%%%%%%%%%%%%%%%%%%%%%%%%%%%%%%%%%%%

\subsection{Manifolds without boundary}\label{boundaryless}
In this subsection $\Omega=M$, with $(M,g)$ a compact complete $n$-dimensional smooth Riemannian manifold (without boundary) and ${\rm Ric}\geq-(n-1)\kappa^2$, $\kappa\geq 0$. 

A double integration by parts and Bochner's formula imply that
$$
\int_M |D^2u|^2+{\rm Ric}(\nabla u,\nabla u)dv=\int_{M}(\Delta u)^2dv
$$
for all $u\in H^2(M)$. In particular we have that
$$
0=\mu_1<\mu_2\leq\cdots\leq\mu_j\leq\cdots\nearrow +\infty.
$$
In fact, one easily checks that all the eigenvalues are non-negative, and that there is only one zero eigenvalue with associated eigenfunctions the constant functions on $M$. 

We prove now that the eigenvalues $\mu_j$ of \eqref{variational_0} are exactly the squares of the eigenvalues of the Laplacian on $M$. Recall that the weak formulation of the closed eigenvalue problem for the Laplacian is
\begin{equation}\label{weak_N}
\int_M\langle\nabla u,\nabla\phi\rangle dv=m\int_{M}u\phi dv\,,\ \ \ \forall\phi\in H^1(M),
\end{equation}
in the unknowns $(u,m)\in H^1(M)\times\mathbb R$. Problem \eqref{weak_N} admits an increasing sequence of non-negative eigenvalues of finite multiplicity
$$
0=m_1<m_2\leq \cdots\leq m_j\leq\cdots\nearrow+\infty
$$
and the corresponding eigenfunctions can be chosen to form a orthonormal basis of $L^2(M)$. We have the following theorem.

\begin{thm}\label{close_comp}
Let $(M,g)$  be a compact complete $n$-dimensional smooth Riemannian manifold. Let $\left\{m_j\right\}_{j=1}^{\infty}$ denote the eigenvalues of the Laplacian on $M$. Then for all $j\in\mathbb N$
$$
\mu_j=m_j^2.
$$
and the corresponding eigenfunctions can be chosen to be the same.
\end{thm}

\begin{proof}
Let $v_i$ denote the  eigenfunctions associated with $m_i$ normalized by\\ $\int_{\Omega}v_iv_kdv=\delta_{ik}$. Since the metric is smooth, we have that $v_i\in H^2(M)$ and $-\Delta v_i=m_iv_i$ in $L^2(M)$. Hence, by setting $V:=<v_1,...,v_j>$, we have that $V$ is a $j$-dimensional subspace of $H^2(M)$ and a function $v\in V$ is of the form $v=\sum_{i=1}^j\alpha_i v_i$ for some $\alpha_1,....\alpha_j\in\mathbb R$. Hence, from \eqref{minmax} we have
$$
\mu_j\leq\max_{v\in V}\frac{\int_{M}(\Delta v)^2dv}{\int_Mv^2dv}=\max_{(\alpha_1,...,\alpha_j)\in\mathbb R^j}\frac{\sum_{i=1}^j\alpha_i^2m_i^2}{\sum_{i=1}^j\alpha_i^2}=m_j^2.
$$
On the other hand, the well-known min-max principle for the eigenvalues of the Laplacian on $M$ states that
$$
m_j=\min_{\substack{U\subset H^1(M)\\{\rm dim}U=j}}\max_{\substack{u\in U\\u\ne 0}}\frac{\int_M|\nabla u|^2dv}{\int_M u^2dv}.
$$
We choose $U:=<u_1,...,u_j>$ where $u_1,...,u_j$ are the eigenfunctions associated with the eigenvalues $\mu_1,...,\mu_j$ of the biharmonic operator on $M$ normalized by $\int_{M}u_iu_k dv=\delta_{ik}$. Then $\int_M\Delta u_i\Delta u_k dv=\mu_i\delta_{ik}$. Any $u\in U$ is of the form $u=\sum_{i=1}^j\alpha_iu_i$ for some $\alpha_1,...,\alpha_j\in\mathbb R$. We note that
\begin{multline}
\int_{M}|\nabla u|^2dv=-\int_M u\Delta u dv\leq \left(\int_M u^2 dv\right)^{\frac{1}{2}}\left(\int_M(\Delta u)^2dv\right)^{\frac{1}{2}}\\
=\left(\sum_{i=1}^j\alpha_i^2\right)^{\frac{1}{2}}\left(\sum_{i=1}^j\alpha_i^2\mu_i\right)^{1/2},
\end{multline}
hence
$$
m_j\leq\max_{(\alpha_1,...,\alpha_j)\in\mathbb R^j}\left(\frac{\sum_{i=1}^j\alpha_i^2\mu_i}{\sum_{i=1}^j\alpha_i^2}\right)^{\frac{1}{2}}=\mu_j^{\frac{1}{2}}.
$$
The rest of the proof is straightforward.
\end{proof}

From Theorem \ref{close_comp} and from \eqref{buser} we deduce the following corollary.

\begin{cor}\label{cor_boundaryless}
Let $(M,g)$  be a compact complete $n$-dimensional smooth Riemannian manifold (without boundary) with ${\rm Ric}\geq-(n-1)\kappa^2$, $\kappa\geq 0$.  Then for all $j\in\mathbb N$
$$
\mu_j\leq\left(\frac{(n-1)^2}{4}\kappa^2+C_n\left(\frac{j}{|\Omega|}\right)^{\frac{2}{n}}\right)^2.
$$
\end{cor}

%%%%%%%%%%%%%%%%%%%%%%%%%%%%%%%%%%%%%%%%%%%%%%%%%%%%%%%%%%%%%%%%%%%%%%%%%%%%%%%%%%%%%%%%%%%%%%%%%%%%%%%%%%%%%%%%%%%%%%%%%%%%%%%%%%%%%%%%

%%%%%%%%%%%%%%%%%%%%%%%%%%%%%%%%%%%%%%%%%%%%              DOMAINS WITH CONVEX BOUNDARY     %%%%%%%%%%%%%%%%%%%%%%%%%%%%%%%%%%%%%%%%%%%%% 

%%%%%%%%%%%%%%%%%%%%%%%%%%%%%%%%%%%%%%%%%%%%%%%%%%%%%%%%%%%%%%%%%%%%%%%%%%%%%%%%%%%%%%%%%%%%%%%%%%%%%%%%%%%%%%%%%%%%%%%%%%%%%%%%%%%%%%%%

\subsection{Domains with convex boundary}\label{sub_convex}
In this last subsection we shall present a case in which upper bounds for biharmonic Neumann eigenvalues $\mu_j$ can be deduced directly by comparison with Neumann eigenvalues of the Laplacian and by \eqref{colbois_maerten}.

Let $(M,g)$ be a complete $n$-dimensional smooth Riemannian manifold with ${\rm Ric}\geq-(n-1)\kappa^2$, $\kappa\geq 0$, and let $\Omega$ be a bounded domain in $M$ with $C^2$ boundary. If $II\geq 0$ then we can compare the eigenvalues of \eqref{variational_0} with the squares of the eigenvalues of the Neumann Laplacian on $\Omega$. We recall that the weak formulation of the Neumann problem for the Laplace operator on $\Omega$ is given by \eqref{weak_N} with $M$ replaced by $\Omega$. Neumann eigenvalues of the Laplacian have finite multiplicity, are non-negative and form an increasing sequence
$$
0=m_1<m_2\leq \cdots\leq m_j\leq\cdots\nearrow+\infty.
$$
The associated eigenfunctions are denoted by $\left\{v_i\right\}_{i=1}^{\infty}$ can be chosen to form a orthonormal basis of $L^2(\Omega)$. We have the following

\begin{thm}\label{comp_conv}
Let $(M,g)$ be a complete $n$-dimensional smooth Riemannian manifold and let $\Omega$ will be a bounded domain of $M$ of class $C^2$ with $II\geq 0$. Then
$$
\mu_j\leq m_j^2,
$$
 for all $j\in\mathbb N$
\end{thm}

\begin{proof}
We have seen that for all $u,\phi\in H^2(\Omega)$ (see \eqref{step-1} and \eqref{step-2})
\begin{multline}
\int_{\Omega}\langle D^2 u, D^2\phi\rangle+{\rm Ric}(\nabla u,\nabla\phi)dv=\int_{\Omega}\Delta u\Delta\phi dv\\
-\int_{\partial\Omega}\left((n-1)\mathcal H\frac{\partial u}{\partial\nu}+\Delta_{\partial\Omega}u\right)\frac{\partial\phi}{\partial\nu}d\sigma\\
-\int_{\partial\Omega}\left(II(\nabla_{\partial\Omega}u,\nabla_{\partial\Omega}\phi)+\frac{\partial u}{\partial\nu}\Delta_{\partial\Omega}\phi\right)d\sigma.
\end{multline}
Since the domain is of class $C^2$, by standard elliptic regularity we have that the eigenfunctions $v_i$ of the Neumann Laplacian belong to $H^2(\Omega)$. Therefore $-\Delta v_i=m_iv_i$ in $L^2(\Omega)$ and $\frac{\partial v_i}{\partial\nu}=0$ in $L^2(\partial\Omega)$. We deduce that for any linear combination $v=\sum_{i=1}^j \alpha_i v_i$ with $\alpha_i\in\mathbb R$
\begin{multline}
\int_{\Omega}|D^2v|^2+{\rm Ric}(\nabla v,\nabla v)dv=\int_{\Omega}(\Delta v)^2dv-\int_{\partial\Omega}II(\nabla_{\partial\Omega}v,\nabla_{\partial\Omega}v)d\sigma\\
\leq \int_{\Omega}(\Delta v)^2dv=\int_{\Omega}\left(\sum_{i=1}^j\alpha_im_iv_i\right)^2dv=\sum_{i=1}^j\alpha_i^2m_i^2.
\end{multline}
Consider then $V:=<v_1,...,v_j>$ the $j$-dimensional space spanned by the first $j$ eigenfunctions of the Neumann Laplacian. This is a subspace of $H^2(\Omega)$ of dimension $j$. Each $v\in V$ is of the form $v=\sum_{i=1}^j\alpha_i v_i$ for some $\alpha_1,...,\alpha_j\in\mathbb R$. Moreover 
$$
\int_{\Omega}v^2dv=\sum_{i=1}^j\alpha_i^2.
$$
From \eqref{minmax} we have that
$$
\mu_j\leq\max_{v\in V}\frac{\int_{\Omega}|D^2v|^2+{\rm Ric}(\nabla v,\nabla v)dv}{\int_{\Omega}v^2dv}\leq\max_{(\alpha_1,...,\alpha_j)\in\mathbb R^j}\frac{\sum_{i=1}^j\alpha_i^2m_i^2}{\sum_{i=1}^j\alpha_i^2}=m_j^2.
$$
This concludes the proof.
\end{proof}

Theorem \ref{comp_conv} and inequality \eqref{colbois_maerten} imply the following corollary.
\begin{cor}\label{cor_convex}
Let $(M,g)$ be a complete $n$-dimensional smooth Riemannian manifold with ${\rm Ric}\geq-(n-1)\kappa^2$, $\kappa\geq 0$, and let $\Omega$ be a bounded domain of $M$ of class $C^2$ with $II\geq 0$. Then 
\begin{equation}\label{conjecture}
\mu_j\leq\left(A_n\kappa^2+B_n\left(\frac{j}{|\Omega|}\right)^{\frac{2}{n}}\right)^2,
\end{equation}
for all $j\in\mathbb N$.
\end{cor}

We remark that this bound holds independently of the size of $\Omega$,  its diameter and the injectivity radius of $M$. Hence it is natural to pose the following question, whose answer seems quite complicated at this stage.\\
\vskip .2pt
{\bf Open problem.} Prove inequality \eqref{conjecture} for any bounded domain $\Omega$ with $C^1$ boundary in a complete $n$-dimensional smooth Riemannian manifold with ${\rm Ric}\geq -(n-1)\kappa^2$, $\kappa\geq 0$.\\
\vskip .2pt

%\begin{rem}
%Let us consider as an example the unit ball $B$ in $\mathbb R^n$. Given a normalized eigenfunction $v_i$ of the Neumann Laplacian on $B$, it can be expressed in the spherical coordinates $(r,\theta_1,...,\theta_{n-1})$ as
%$$
%v_i=R(r)H(\theta),
%$$
%where $\theta=(\theta_1,...,\theta_{n-1})\in\mathbb S^{n-1}$. The spherical part $H(\theta)$ is of the form $H(\theta)=H_l(\theta)$ for some $l\in\mathbb N$, where $H_l$ denote the spherical harmonics of degree $l$ normalized by $\int_{\mathbb S^{n-1}}H_l(\theta)H_m(\theta)d\sigma(\theta)=\delta_{lm}$. Moreover, the functions $H_l$ are the solution of the Laplace equation on $\mathbb S^{n-1}$, i.e.,
%$$
%-\Delta_{\mathbb S^{n-1}}H_l=l(l+n-2)H_l
%$$
%and in particular
%$$
%\int_{\mathbb S^{n-1}}H_l(\theta)H_m(\theta)d\sigma(\theta)=l(l+n-2)\delta_{lm}.
%%$$
%Moreover 
%$$
%\int_{\mathbb S^{n-1}}II(H_l(\theta),H_m(\theta))d\sigma(\theta)=l(l+n-2)\delta_{lm}.
%$$
%Hence we find that 
%\begin{equation}
%\mu_j\leq m_j^2-\max_{(\alpha_1,...,\alpha_j)\in\mathbb R^j}\frac{\sum_{i=1}^j\alpha_i^2l_i(l_i+n-2)}{\sum_{i=1}^j\alpha_i^2}=m_j^2-\max_{i=1,...,j}l_i(l_i+n-2),
%\end{equation}
%where $l_i$ is such that $v_i$ has angular part $H_{l_i}$. In particular, a trivial bound is the following, holding for all $j\geq 2$
%$$
%\mu_j\leq m_j^2-(n-1).
%$$
%\end{rem}
\thanks{The authors are thankful to Davide Bianchi and Alberto G. Setti for fruitful discussion on the construction of Laplacian cut-off functions on Riemannian manifolds. The first author is grateful to the Dipartimento di Matematica ``Tullio Levi-Civita'' of the University of Padova for hospitality
that supported this collaboration. The second author acknowledges hospitality of the Institut de Math\'emathiques of the University of Neuch\^atel, where this collaboration began. The second author is member of the Gruppo Nazionale per l'Analisi Matematica, la Probabilit\`a e le loro Applicazioni (GNAMPA) of the Istituto Nazionale di Alta Matematica (INdAM).}

%\end{comment}

\bibliography{bibliography}{}

\def\cprime{$'$} \def\cprime{$'$} \def\cprime{$'$} \def\cprime{$'$}
  \def\cprime{$'$}
\begin{thebibliography}{10}

\bibitem{anne}
C.~Ann\'{e}.
\newblock A note on the generalized dumbbell problem.
\newblock {\em Proc. Amer. Math. Soc.}, 123(8):2595--2599, 1995.

\bibitem{arrieta00}
J.~M. Arrieta.
\newblock Rates of eigenvalues on a dumbbell domain. {S}imple eigenvalue case.
\newblock {\em Trans. Amer. Math. Soc.}, 347(9):3503--3531, 1995.

\bibitem{ferraresso}
J.~M. Arrieta, F.~Ferraresso, and P.~D. Lamberti.
\newblock Spectral analysis of the biharmonic operator subject to {N}eumann
  boundary conditions on dumbbell domains.
\newblock {\em Integral Equations Operator Theory}, 89(3):377--408, 2017.

\bibitem{aubin}
T.~Aubin.
\newblock {\em Nonlinear analysis on manifolds. {M}onge-{A}mp\`ere equations},
  volume 252 of {\em Grundlehren der Mathematischen Wissenschaften [Fundamental
  Principles of Mathematical Sciences]}.
\newblock Springer-Verlag, New York, 1982.

\bibitem{bianchi_setti}
D.~Bianchi and A.~G. Setti.
\newblock Laplacian cut-offs, porous and fast diffusion on manifolds and other
  applications.
\newblock {\em Calc. Var. Partial Differential Equations}, 57(1):Art. 4, 33,
  2018.

\bibitem{grillo_vasquez}
M.~Bonforte, G.~Grillo, and J.~L. Vazquez.
\newblock Fast diffusion flow on manifolds of nonpositive curvature.
\newblock {\em J. Evol. Equ.}, 8(1):99--128, 2008.

\bibitem{bourlard}
M.~Bourlard and S.~Nicaise.
\newblock Abstract {G}reen formula and applications to boundary integral
  equations.
\newblock {\em Numer. Funct. Anal. Optim.}, 18(7-8):667--689, 1997.

\bibitem{buosopalinuro}
D.~Buoso.
\newblock Analyticity and criticality of the eigenvalues of the biharmonic
  operator.
\newblock {\em Submitted}, 2015.

\bibitem{bcp}
D.~Buoso, L.~M. Chasman, and L.~Provenzano.
\newblock On the stability of some isoperimetric inequalities for the
  fundamental tones of free plates.
\newblock {\em J. Spectr. Theory}, 8(3):843--869, 2018.

\bibitem{buser}
P.~Buser.
\newblock Beispiele f\"ur {$\lambda _{1}$} auf kompakten {M}annigfaltigkeiten.
\newblock {\em Math. Z.}, 165(2):107--133, 1979.

\bibitem{chasman}
L.~M. Chasman.
\newblock An isoperimetric inequality for fundamental tones of free plates.
\newblock {\em Comm. Math. Phys.}, 303(2):421--449, 2011.

\bibitem{chasmancircular}
L.~M. Chasman.
\newblock Vibrational modes of circular free plates under tension.
\newblock {\em Appl. Anal.}, 90(12):1877--1895, 2011.

\bibitem{chasman2}
L.~M. Chasman.
\newblock An isoperimetric inequality for fundamental tones of free plates with
  nonzero {P}oisson's ratio.
\newblock {\em Appl. Anal.}, 95(8):1700--1735, 2016.

\bibitem{chavel}
I.~Chavel.
\newblock {\em Eigenvalues in {R}iemannian geometry}, volume 115 of {\em Pure
  and Applied Mathematics}.
\newblock Academic Press, Inc., Orlando, FL, 1984.
\newblock Including a chapter by Burton Randol, With an appendix by Jozef
  Dodziuk.

\bibitem{cheeger_colding}
J.~Cheeger and T.~H. Colding.
\newblock Lower bounds on {R}icci curvature and the almost rigidity of warped
  products.
\newblock {\em Ann. of Math. (2)}, 144(1):189--237, 1996.

\bibitem{cheng_clamped}
Q.-M. Cheng, T.~Ichikawa, and S.~Mametsuka.
\newblock Estimates for eigenvalues of a clamped plate problem on {R}iemannian
  manifolds.
\newblock {\em J. Math. Soc. Japan}, 62(2):673--686, 2010.

\bibitem{colboisgirouard}
B.~Colbois, A.~El~Soufi, and A.~Girouard.
\newblock Isoperimetric control of the spectrum of a compact hypersurface.
\newblock {\em J. Reine Angew. Math.}, 683:49--65, 2013.

\bibitem{colbois_maerten}
B.~Colbois and D.~Maerten.
\newblock Eigenvalues estimate for the {N}eumann problem of a bounded domain.
\newblock {\em J. Geom. Anal.}, 18(4):1022--1032, 2008.

\bibitem{colbois_rough}
B.~Colbois and D.~Maerten.
\newblock Eigenvalue estimate for the rough {L}aplacian on differential forms.
\newblock {\em Manuscripta Math.}, 132(3-4):399--413, 2010.

\bibitem{cohil}
R.~Courant and D.~Hilbert.
\newblock {\em Methods of mathematical physics. {V}ol. {I}}.
\newblock Interscience Publishers, Inc., New York, N.Y., 1953.

\bibitem{escobar}
J.~F. Escobar.
\newblock Uniqueness theorems on conformal deformation of metrics, {S}obolev
  inequalities, and an eigenvalue estimate.
\newblock {\em Comm. Pure Appl. Math.}, 43(7):857--883, 1990.

\bibitem{giroire}
J.~Giroire and J.-C. N\'{e}d\'{e}lec.
\newblock A new system of boundary integral equations for plates with free
  edges.
\newblock {\em Math. Methods Appl. Sci.}, 18(10):755--772, 1995.

\bibitem{gny}
A.~Grigor'yan, Y.~Netrusov, and S.-T. Yau.
\newblock Eigenvalues of elliptic operators and geometric applications.
\newblock In {\em Surveys in differential geometry. {V}ol. {IX}}, volume~9 of
  {\em Surv. Differ. Geom.}, pages 147--217. Int. Press, Somerville, MA, 2004.

\bibitem{guneysu}
B.~G\"{u}neysu.
\newblock Sequences of {L}aplacian cut-off functions.
\newblock {\em J. Geom. Anal.}, 26(1):171--184, 2016.

\bibitem{asma_conf}
A.~Hassannezhad.
\newblock Conformal upper bounds for the eigenvalues of the {L}aplacian and
  {S}teklov problem.
\newblock {\em J. Funct. Anal.}, 261(12):3419--3436, 2011.

\bibitem{hebey}
E.~Hebey.
\newblock {\em Sobolev spaces on {R}iemannian manifolds}, volume 1635 of {\em
  Lecture Notes in Mathematics}.
\newblock Springer-Verlag, Berlin, 1996.

\bibitem{jimbomorita}
S.~Jimbo and Y.~Morita.
\newblock Remarks on the behavior of certain eigenvalues on a singularly
  perturbed domain with several thin channels.
\newblock {\em Comm. Partial Differential Equations}, 17(3-4):523--552, 1992.

\bibitem{Kro}
P.~Kr{\"o}ger.
\newblock Upper bounds for the {N}eumann eigenvalues on a bounded domain in
  {E}uclidean space.
\newblock {\em J. Funct. Anal.}, 106(2):353--357, 1992.

\bibitem{Lap1997}
A.~Laptev.
\newblock Dirichlet and {N}eumann eigenvalue problems on domains in {E}uclidean
  spaces.
\newblock {\em J. Funct. Anal.}, 151(2):531--545, 1997.

\bibitem{Lichnerowicz}
A.~Lichnerowicz.
\newblock {\em G\'{e}om\'{e}trie des groupes de transformations}.
\newblock Travaux et Recherches Math\'{e}matiques, III. Dunod, Paris, 1958.

\bibitem{nadai}
A.~Nadai.
\newblock {\em Theory of flow and fracture of solids. 1}.
\newblock Engineering Societies monographs. McGraw-Hill, 1950.

\bibitem{nazaret}
C.~Nazaret.
\newblock A system of boundary integral equations for polygonal plates with
  free edges.
\newblock {\em Math. Methods Appl. Sci.}, 21(2):165--185, 1998.

\bibitem{Obata}
M.~Obata.
\newblock Certain conditions for a {R}iemannian manifold to be isometric with a
  sphere.
\newblock {\em J. Math. Soc. Japan}, 14:333--340, 1962.

\bibitem{petersen}
P.~Petersen.
\newblock {\em Riemannian geometry}, volume 171 of {\em Graduate Texts in
  Mathematics}.
\newblock Springer, New York, second edition, 2006.

\bibitem{pleijel_plate_1}
A.~k. Pleijel.
\newblock On the eigenvalues and eigenfunctions of elastic plates.
\newblock {\em Comm. Pure Appl. Math.}, 3:1--10, 1950.

\bibitem{pleijel_plate_2}
A.~k. Pleijel.
\newblock On {G}reen's functions for elastic plates with clamped, supported and
  free edges.
\newblock In {\em Proceedings of the {S}ymposium on {S}pectral {T}heory and
  {D}ifferential {P}roblems}, pages 413--437. Oklahoma Agricultural and
  Mechanical College, Stillwater, Okla., 1951.

\bibitem{kalamata}
L.~Provenzano.
\newblock A note on the {N}eumann eigenvalues of the biharmonic operator.
\newblock {\em Math. Methods Appl. Sci.}, 41(3):1005--1012, 2018.

\bibitem{reilly}
R.~C. Reilly.
\newblock Applications of the {H}essian operator in a {R}iemannian manifold.
\newblock {\em Indiana Univ. Math. J.}, 26(3):459--472, 1977.

\bibitem{verchota}
G.~C. Verchota.
\newblock The biharmonic {N}eumann problem in {L}ipschitz domains.
\newblock {\em Acta Math.}, 194(2):217--279, 2005.

\bibitem{wang_clamped}
Q.~Wang and C.~Xia.
\newblock Universal bounds for eigenvalues of the biharmonic operator on
  {R}iemannian manifolds.
\newblock {\em J. Funct. Anal.}, 245(1):334--352, 2007.

\bibitem{xia_clamped}
Q.~Wang and C.~Xia.
\newblock Inequalities for eigenvalues of a clamped plate problem.
\newblock {\em Calc. Var. Partial Differential Equations}, 40(1-2):273--289,
  2011.

\end{thebibliography}
\bibliographystyle{abbrv}

% ----------------------------------------------------------------

\end{document}